\documentclass[12pt]{amsbook}
\usepackage[all,cmtip]{xy}

\usepackage{a4wide}
\usepackage{pxfonts}
\usepackage{wasysym}
\usepackage{amssymb,amsxtra,mathrsfs}
\usepackage{eucal}
\usepackage{setspace,upgreek}
\usepackage{mathtools}
\usepackage{url}
\usepackage{hyperref}

\newcommand{\mf}[1]{\mathfrak{#1}}
\newcommand{\mc}[1]{\mathcal{#1}}
\newcommand{\mh}[1]{\mathscr{#1}}

\newcommand{\mr}[1]{\mathrm{#1}}

\newcommand{\ov}{\overline}

\newcommand{\lr}[2]{\langle #1,#2\rangle}

\newcommand{\ze}{\mathbf{0}} 
\newcommand{\un}{\mathbf{1}}

\newcommand{\K}{\mathbb{K}}
\newcommand{\R}{\mathbb{R}}

\newcommand{\C}{\mathbb{C}}
\newcommand{\Z}{\mathbb{Z}}

\newtheorem{theorem}{Theorem}
\newtheorem{definition}[theorem]{Definition}
\newtheorem{lemma}[theorem]{Lemma}
\newtheorem{proposition}[theorem]{Proposition}
\newtheorem{corollary}[theorem]{Corollary}
\newtheorem{assumption}[theorem]{Assumption}
\newtheorem{convention}[theorem]{Convention}
\newtheorem{notations}[theorem]{Notation}

\theoremstyle{definition}

\newtheorem{remark}
[theorem]
{Remark}

\newtheorem{example}
[theorem]
{Example}

\numberwithin{subsection}{section}
\numberwithin{theorem}{section}
\numberwithin{equation}{section}

\begin{document}
\title{Universal operatorial Noether current associated with a quasi-complete locally convex linear space valued Lagrangian 
defined on an open set of a product of normed spaces.}
\author{Benedetto Silvestri}
\email{bndsilvestri@proton.me}
\address{440 W Kelly Ave Apt 6 Jackson, WY 83001-8544 USA}
\address{PO BOX 14845 Jackson, WY 83002-4845 USA}
\begin{abstract}
In the context of locally convex vector valued integrable tensors on open sets of a $n$-dimensional manifold $M$ developped in 
our previous work \cite{sil09261621}, in this paper we construct a framework in which to address both a variational problem 
associated with a Lagrangian $\mh{L}$, and a conservation law on-shell associated with actions $\phi$ and $\omega$. 
$\mh{L}$ is a $C^{p}$-map valued in a quasi-complete Hausdorff locally convex space $\mc{Z}$ and defined on an open set $O$ of a 
normed space isometric to the product of $n+1$ copies of a normed space $Y_{0}$; $\phi$ is a $C^{p}$-action parametrized by an 
open set $A$ of a normed space $T$, acting over the space $\mf{Z}^{O}$ of the fields of the problem, fields valued in $Y_{0}$ 
and defined on $W$ the domain of a chart of $M$; $\omega$ is a $C^{p}$-action on $W$ parametrized by $A$.
With the above described data and by employing a result in \cite{sil09261621}, 
we construct a map $\mf{A}:\mf{Z}^{O}\times A\to\mc{M}(W,\mc{Z}_{\R})$ where for $u\in\mf{Z}^{O}$ and $\eta\in A$ we interpret 
$\mf{A}(u,\eta)$ as the action functional associated with the Lagrangian $\mh{L}$, the field $u$ and the perturbation parameter 
$\eta$. Here $\mc{M}(W,\mc{Z}_{\R})=\mc{L}_{s}(\mc{H}(W),\mc{Z}_{\R})$ is the locally convex space of vector 
measure defined on $W$ and valued in $\mc{Z}_{\R}$ endowed by the topology of simple convergence. 
By letting $\mc{A}(u):A\to\mc{M}(W,\mc{Z}_{\R})$ be such that $\mc{A}(u)(\eta)=\mf{A}(u,\eta)$, 
we establish that $\mc{A}(u)$ is differentiable off-shell, determine its differential map off and then on-shell.
After having defined the main object of the paper: the universal operatorial Noether current $\mh{N}[u]$ relative to $u$, 
a $\mc{L}_{b}(T,\mc{Z}_{\R})$-valued $(p-1)$-times continuously differentiable vector field of $M$ on $U$ the range of the chart 
on $W$, we establish the main result: if the differential map of $\mc{A}(u)$ on $\ze$ is $\ze$, then $\mh{N}[u]$ is conserved 
on-shell. Here on-shell means that $u\in\mr{EL}_{W}(\mh{L})$ i.e. $u$ satisfies the operatorial Euler-Lagrange equation 
associated with $\mh{L}$. In the second main result we prove that if in addition $p=\infty$, then 
the conservation on-shell of $\mh{N}[u]$ can be exhibited by stating that the divergence of $\mh{N}[u]$ equals $\ze$, 
where the divergence of a general locally convex space valued vector field is another object defined and analyzed in 
\cite{sil09261621}. 
Finally by letting $\mf{a}(\eta):\mf{Z}^{O}\to\mc{M}(W,\mc{Z}_{\R})$, such that $\mf{a}(\eta)(u)\coloneqq\mf{A}(u,\eta)$, 
among many other results we provides a sufficient condition ensuring the differentiability of $\mf{a}(\ze)$ and establishes its 
differential map; we display a characterization of $\mr{vp}(\mf{a}(\ze))$, solution set of the variational problem defined
by $\mf{a}(\ze)$, in terms of the solution set $\mr{sEL}_{W}(\mh{L})$ of a functional differential equation; we state that 
$\mr{sEL}_{W}(\mh{L})$ and then $\mr{vp}(\mf{a}(\ze))$ is a subset of $\mr{EL}_{W}(\mh{L})$ and establish a sufficient condition 
ensuring that the invariance $\mf{a}(\eta)=\mf{a}(\ze)$ would imply that $\phi_{\eta}$ is a symmetry of the solution set 
$\mr{sEL}_{W}(\mh{L})$.
Three are the preeminent original features that permit the introduction of other new components and grant the novel 
contributions above outlined.
Firstly by employing the result in our previous work \cite[Thm.3.57]{sil09261621}, we construct the map $\mf{A}$ 
defined on $\mf{Z}^{O}\times A$ with values in the locally convex linear space of $\mc{Z}_{\R}$-valued measures on $W$ and 
such that $\mc{A}(u)$ and under additional conditions also $\mf{a}(\eta)$ are differentiable. 
Here recall that $\mc{Z}$, is a general quasi-complete Hausdorff locally convex space, while $T$ and $Y_{0}$ are 
possibly infinite dimensional normed spaces. 
Secondly and foremostly we construct the $\mc{L}_{b}(T,\mc{Z}_{\R})$-valued vector field $\mh{N}[u]$ that 
naturally encodes in a unique object all the possible conservation laws associated with $\mh{L}$ and $\phi$, each one of them 
obtained evaluating $\mh{N}[u]$ over an element of $T$.
Thirdly by employing two results in our previous work \cite[Cor.3.41 and Cor.3.48]{sil09261621},
we construct the \emph{divergence of the operator valued vector field $\mh{N}[u]$}.
\end{abstract}

\maketitle

\begin{notations}
\label{12152204}
In the present Notation section we let $X$ be a topological space, $\K\in\{\R,\C\}$, $Z$, $Q$ and $R$ be linear topological 
spaces on $\K$, $E$ and $F$ be a locally convex spaces on $\K$, $\Gamma_{E}$ be a fundamental set of seminorms of $E$. 
Furthermore let $N$ be a normed space on $\K$ whose norm is $\|\cdot\|_{N}$, 
\par
If $A,B$ are sets we let $\mc{F}(A,B)$ be the set of maps defined on $A$ and with values in $B$, if $C\subseteq B$, then we let 
$\imath_{C}^{B}$ or $\mr{Id}_{C\to B}$ be the identity injection of $C$ into $B$ and set $\mr{Id}_{C}=\mr{Id}_{C\to C}$.
If $f:A\to B$, then we let $f_{\natural}:A\to f(A)$, $x\mapsto f(a)$.
If $S_{0}\subseteq\mc{F}(A,B)$, $S$ is a topological space whose underlying set equals $S_{0}$ and $a\in A$, then we let 
$\mr{ev}_{a}^{S}:S\to B$, $f\mapsto f(a)$. 
If $x\in X$, then we let $\mf{T}_{x}^{X}$ be the filter of neighbourhoods of $x$ in $X$, if $P$ is a subset of $X$, then we let 
$\partial P$ be the frontier in $X$ of $P$.
We let $\mc{L}(Z,Q)$ be the linear space on $\K$ of $\K$-linear and continuous maps defined on $Z$ and valued in $Q$, set 
$\mc{L}(Z)=\mc{L}(Z,Z)$. Let $\updelta^{\K}:\R\to\K$ be the ring morphism $\lambda\mapsto\lambda\un_{\K}$, we let $Z_{\R}$ be the 
linear topological space on $\R$ underlying $Z$ whose underlying linear space on $\R$ equals $\updelta_{\star}^{\K}(Z)$ see 
\cite[Notation 0.1]{sil09261621} and whose topology equals that of $Z$. $Z_{\R}$ is well-defined since the topology on $Z$ is 
compatible also with the structure of linear space on $\R$ above defined. This implies that if $T\in\mc{L}(Z,Q)$ and 
$S\in\mc{L}(Q,R)$, then 
\begin{equation}
\label{12152119}
\begin{aligned}
\updelta_{\star}^{\K}(T)&\in\mc{L}(Z_{\R},Q_{\R})
\\
\updelta_{\star}^{\K}(S\circ T)&=\updelta_{\star}^{\K}(S)\circ\updelta_{\star}^{\K}(T).
\end{aligned}
\end{equation}
Furthermore if $\K=\C$, then let 
\begin{equation*}
\upeta_{Z}:Z\to Z,\,x\mapsto ix 
\end{equation*}
thus, $\upeta_{Z}\in\mc{L}(Z)$ and $\upeta_{Z}\circ\upeta_{Z}=-\mr{Id}_{Z}$ in particular $\upeta_{Z}$ is an isomorphism of 
topological linear spaces on $\C$ thus, by \eqref{12152119} we obtain
\begin{equation}
\label{12152056}
\begin{aligned}
\updelta_{\star}^{\K}(\upeta_{Z})&\in\mc{L}(Z_{\R})
\\
\updelta_{\star}^{\K}(\upeta_{Z})^{-1}&\in\mc{L}(Z_{\R}).
\end{aligned}
\end{equation}
Now the discussion present in \cite[II.65]{EVT} states what in our language reads as follows
\begin{equation*}
\updelta_{\star}^{\K}(\mc{L}(Z,Q))
=
\left\{\phi-\updelta_{\star}^{\K}(\upeta_{Q})\circ\phi\circ\updelta_{\star}^{\K}(\upeta_{Z})\,\vert\,\phi\in\mc{L}(Z_{\R},Q_{\R})\right\},
\end{equation*}
in particular we have the following strict inclusion of linear spaces on $\R$
\begin{equation}
\label{12152057}
\updelta_{\star}^{\K}(\mc{L}(Z,Q))\subset\mc{L}(Z_{\R},Q_{\R}).
\end{equation}
Next $\Gamma_{E}$ is a fundamental set of seminorms of $E_{\R}$. 
We let $\mc{C}(X,Z)$ be the linear space on $\K$ of continuous maps defined on $X$ and with values in $Z$, 
and let $\mc{C}_{c}(X,Z)$ be the linear subspace on $\K$ of those maps in $\mc{C}(X,Z)$ with compact support. If $A$ is a 
uniform space, then we let $\mc{C}_{s}(A,E)$ be the linear topological locally convex space on $\K$ of continuous maps
defined on $A$ and valued in $E$ provided by the topology of pointwise convergence which is locally convex and generated by the 
set of seminorms $\{q_{a}^{p}\,|\,a\in A,p\in\Gamma_{E}\}$ with $q_{a}^{p}(f)\coloneqq p(f(a))$.
We let $\mc{L}_{s}(Z,E)$ be the topological subspace of the those maps $f\in\mc{C}_{s}(Z,E)$ which are linear.
We let $B_{r}^{N}(\ze)\coloneqq\{v\in N\,\vert\,\|v\|_{N}<r\}$, let $B_{r}^{N}(p)\coloneqq p+B_{r}^{N}(\ze)$ and let 
$\ov{B}_{r}^{N}(p)$ denote $\ov{B_{r}^{N}(p)}$ for every $p\in N$.
We let $\mr{Bnd}(F)$ be the set of bounded subsets of $F$. If $n\in\Z_{>0}$ and $E_{i}$ is a locally convex space on $\K$ for 
every $i\in[1,n]\cap\Z$, then we let $\mc{L}(E_{1},\dots,E_{n};F)$ be the linear space on $\K$ of continuous $\K$-multilinear
maps defined on $\prod_{i=1}^{n}E_{i}$ and valued in $F$, while let $\mc{L}_{b}^{n}(E_{1},\dots,E_{n};F)$ be the space 
$\mc{L}(E_{1},\dots,E_{n};F)$ provided by the topology of uniform 
convergence over bounded sets which is locally convex and generated by the following set of seminorms: If 
$\Gamma_{F}$ and $\Gamma_{E_{i}}$ are fundamental sets of seminorms of $F$ and $E_{i}$ respectively with $i\in[1,n]\cap\Z$, then
$\{p_{B}\,\vert\,B\in\prod_{i=1}^{n}\mr{Bnd}(E_{i})\}$ is a fundamental set of seminorms of $\mc{L}_{b}^{n}(E_{1},\dots,E_{n};F)$ where
\begin{equation*}
p_{B}(f)\coloneqq\sup_{x\in\prod_{i=1}^{n}B_{i}}(p\circ f)(x).
\end{equation*}
If in addition $N_{i}$ is a normed space on $\K$ for every $i\in[1,n]\cap\Z$, then $\{\overset{n}{p}\,\vert\,p\in\Gamma_{F}\}$ 
is also a fundamental set of seminorms of $\mc{L}_{b}^{n}(N_{1},\dots,N_{n};F)$ where $\overset{n}{p}\coloneqq p_{B}$ with 
$B_{i}=\ov{B}_{1}^{N_{i}}(\ze)$ for every $i\in[1,n]\cap\Z$. We let 
$\Gamma_{\lr{N,2}{F}}\in\mc{L}(\mc{L}_{b}(N,\mc{L}_{b}(N,F)),\mc{L}_{b}(N,N;F))$ be such that 
$\Gamma_{\lr{N,2}{F}}(T)(a,b)\coloneqq T(a)(b)$ for every $(a,b)\in N\times N$.
\par
In this work we employ measure theory as developped in \cite{IntBourb}, we remark what follows.
Since the discussion in \cite[VI.40 end of $n^{\circ}$10]{IntBourb} and that in \cite[II.65]{EVT} every definition given in
\cite{IntBourb} regarding locally convex spaces over $\R$ can be generalized to the case of locally convex spaces over $\C$ 
therefore, when we quote \cite{IntBourb} we always assume implicitly quoted the discussion in 
\cite[VI.40 end of $n^{\circ}$10]{IntBourb} and that in \cite[II.65]{EVT}. 
\par
\emph{Now assume for the remaining of this Notation section that the locally convex spaces $E$ and $F$ are Hausdorff.} 
\par
If $S$ is a locally compact space, $P\subseteq S$, then let $\mc{H}(S,P;E)$ be the linear space on $\K$ of compactly supported 
continuous maps defined on $S$, valued in $E$ and whose support is contained in $P$. If $K$ is a compact set of $S$, then we 
provide $\mc{H}(S,K;E)$ by the locally convex topology of uniform convergence in $K$. Next let $\mc{H}(S;E)$ be the locally 
convex linear space on $\K$ of compactly supported continuous maps defined on $S$ and valued in $E$ provided by the locally 
convex final topology of the family 
$\{\mc{H}(S,K;E)\}_{K\in\mr{Cmp}(S)}$ and the family of $\K$-linear maps $\{\imath_{\mc{H}(S,K;E)}^{\mc{H}(S;E)}\}_{K\in\mr{Cmp}(S)}$. 
The general definition of a locally convex final topology is provided in \cite[II.29 Prp.5]{EVT}. We set 
\begin{equation*}
\mc{H}(S)\coloneqq\mc{H}(S;\R).
\end{equation*}
Next let $\mc{M}(S,E)$ be the locally convex linear space on $\K$ of measures on $S$ and valued in $E$ defined in 
\cite[VI.18 Def.1]{IntBourb}, provided by the vague topology namely the topology of simple convergence. In other words 
\begin{equation*}
\mc{M}(S,E)\coloneqq\mc{L}_{s}(\mc{H}(S;\K),E).
\end{equation*}
\par
If $n\in\Z_{>0}$, then we let $\{e_{k}[n]\}_{k=1}^{n}$ or simply $\{e_{k}\}_{k=1}^{n}$ whenever it will not cause confusion
be the standard basis of $\R^{n}$. If $F$ is a linear space over $\K$, then we let $\mathbb{M}_{n}(F)\coloneqq F^{[1,n]\times[1,n]}$ 
the $\K$-linear space of the square matrices over $F$ of order $n$; if in addition $F$ is a normed space, then we provide
$\mathbb{M}_{n}(F)$ with the $\max$-norm by identifying $\mathbb{M}_{n}(F)$ with $F^{n^{2}}$. Furthermore we let $\mr{det}_{n}$ and 
$\mr{Tr}_{n}$, or simply $\mr{det}$ and $\mr{Tr}$ whenever $n$ is understood, be the determinant and the trace maps on 
$\mathbb{M}_{n}(\K)$ respectively. 
We let $\mf{p}:\mathbb{M}_{n}(\K)\times\mathbb{M}_{n}(\K)\to\mathbb{M}_{n}(\K)$ denote the standard product of matrices, while 
let $\mathbb{GL}_{n}(\K)$ denote the group of the invertible elements of $\mathbb{M}_{n}(\K)$ and let 
$\mr{Inv}:\mathbb{GL}_{n}(\K)\to\mathbb{GL}_{n}(\K),a\mapsto a^{-1}$.
\par
\emph{Assume for the remaining of this Introduction that the normed space $N$ and the locally convex spaces $E$ and $F$ are all 
on $\R$.} 
\par
We employ the bounded differential calculus of maps defined on a open set of a real normed space and valued in a real Hlcs.
We mention the following fact: under the restriction that say $B$ is an open set of a normed space, we have that 
any $C^{1}$-map from $B$ to $E$ in the sense of bounded differential calculus is a $C^{1}$-map from $B$ to $E$ 
in the sense of Bastiani calculus, while $C^{\infty}$-maps from $B$ to $E$ in the sense of Bastiani calculus are $C^{\infty}$-maps 
from $B$ to $E$ in the sense of bounded calculus.
We opted for the bounded calculus because it is one of the direct generalizations of the standard Frechet calculus 
(between normed spaces) which considers the differential map as valued in the space of linear and continous operators endoved 
by the topology of uniform convergence over bounded sets. Bastiani calculus for maps defined and valued in locally convex spaces
is developped in \cite{glo} see also \cite{kll}, while for Bastiani calculus involving also maps between convergence spaces see 
\cite{bas}, finally the bounded calculus is treated in \cite{yam}. Said that here we only recall what follows.
If $A$ is an open set of $N$, $x\in A$ and $f:A\to E$, then we say that $f$ is differentiable at $x$ iff there exists a 
(necessarily unique) $u\in\mc{L}(N,E)$ such that 
\begin{equation*}
\lim_{h\to\ze,h\in\complement_{B_{r}(\ze)}\{\ze\}}\frac{f(x+h)-f(x)-u(h)}{\|h\|}=\ze,
\end{equation*}
with $r>0$ such that $x+B_{r}(\ze)\subseteq A$. The unique $u$ is denoted by $f^{1}(x)$. 
If $C\subseteq A$, then $f$ is differentiable on $C$ iff it is differentiable at $x$ for every $x\in C$, while 
$f$ is differentiable iff $f$ is differentiable on $A$. If $f$ is differentiable, then we let 
\begin{equation*}
f^{(1)}:A\to\mc{L}_{b}(N,E) 
\end{equation*}
be the differential map of $f$. Therefore higher differentiation and differentials are defined inductively as 
usual and in case $N$ is a product of a finite family of normed spaces on $\R$, then the partial differentiations and partial
differentials are defined as usual in terms of differentiation of suitable maps.
As already mentioned and since $N$ is normed, smoothness with differentiation defined as above is equivalent to smoothness in 
the sense of Bastiani calculus. For any $p\in\Z_{\geq 0}\cup\{\infty\}$ let $\mc{C}^{p}(A,E)$ be the set of $p$-continuously 
differentiable maps defined on $A$ and valued in $E$, where as usual $\mc{C}^{0}(A,E)=\mc{C}(A,E)$.
In this work we shall deal with locally convex spaces on different fields thus, we need to specify what follows.
Let $W$ be an open set of $F$, and let $G$ be a Hausdorff locally convex on $\C$. Since \eqref{12152056} we have for every 
$p\in\Z_{\geq 0}\cup\{+\infty\}$ that 
\begin{equation}
\label{12121915b}
f\in\mc{C}^{p}(W,G_{\R})\Leftrightarrow\updelta_{\star}^{\K}(\upeta_{G})\circ f\in\mc{C}^{p}(W,G_{\R})
\end{equation}
Now \eqref{12121915b} permits to define an action of $\C$ on $\mc{C}^{p}(W,G_{\R})$ as follows
\begin{equation}
\label{12152125}
(\cdot):\C\times\mc{C}^{p}(W,G_{\R})\to\mc{C}^{p}(W,G_{\R}),\,(\lambda,f)\mapsto\updelta_{\star}^{\K}(\upeta_{G})\circ f.
\end{equation}
If $H$ is a Hausdorff locally convex space on $\K_{a}\in\{\R,\C\}$, $W$ is an open set of $H$, and $G$ is a 
Hausdorff locally convex space on $\K_{b}\in\{\R,\C\}$, then $W$ is also an open set of $H_{\R}$ and in the present work by abuse 
of language we
\begin{equation}
\label{12201914}
\textbf{let $\mc{C}^{p}(W,G)$ denote the $\K_{b}$-linear space $\mc{C}^{p}(W,G_{\R})$,} 
\end{equation}
\textbf{where in case $\K_{b}=\C$ the $\K_{b}$-action is provided in \eqref{12152125}}. 
Furthermore by saying that a map $f:W\to G$ is differentiable at $x$ with $x\in W$, we always mean that
$W$ is an open set of $H_{\R}$ and that the map $f:W\to G_{\R}$ is differentiable at $x$ with differential 
\begin{equation*}
f^{1}(x)\in\mc{L}(H_{\R},G_{\R}).
\end{equation*}
\par
If $\mh{Z}=\{(Z_{j},\|\cdot\|_{j})\}_{j=p}^{q}$ with $p<q$ is a family of normed spaces on $\K_{a}\in\{\R,\C\}$ and if 
$P=\prod_{j\in[p,q]\cap\Z}Z_{j}$, then for every $k\in[p,q]\cap\Z$ we let $\Pr_{k}^{\mh{Z}}:P\to Z_{k}$, $v\mapsto v_{k}$ and 
$\mf{i}_{k}^{\mh{Z}}:Z_{k}\to P$ be such that $\Pr_{i}^{\mh{Z}}\circ\mf{i}_{k}^{\mh{Z}}$ equals $\mr{Id}_{Z_{k}}$ if $i=k$, and 
$\Pr_{i}^{\mh{Z}}\circ\mf{i}_{k}^{\mh{Z}}$ equals the constant map $Z_{k}\to Z_{i}$, $w\mapsto\ze_{i}$ if $i\neq k$. Furthermore we let
$\|\cdot\|_{\max}^{\mh{Z}}$, be the norm on $P$ defined by $\|v\|_{\max}^{\mh{Z}}\coloneqq\max_{j\in[p,q]\cap\Z}\|v_{j}\|_{j}$. We notice 
that $P_{\R}=\prod_{j\in[p,q]\cap\Z}(Z_{j})_{\R}$.
If $Z_{i}=\R$ for every $i\in[p,q]\cap\Z$, then for any $p\in P$ and $r>0$, we let $B_{r}(p)$ denote $B_{r}^{\|\cdot\|_{\max}^{\mh{Z}}}(p)$.
If $G$ is a Hausdorff locally convex space on $\K_{b}\in\{\R,\C\}$, $C\subseteq B$ where $B$ is an open set of $P$, and 
$f:B\to G$ is differentiable on $C$ (so by the above convention $f:B\to G_{\R}$ is differentiable on $C$ where we consider $B$ an
open set of $P_{\R}$), then we let 
\begin{equation*}
\frac{\partial f}{\partial Z_{j}}:C\to\mc{L}_{b}((Z_{j})_{\R},G_{\R})
\end{equation*}
be the partial differential map of $f:B\to G_{\R}$ with respect to the subspace $(Z_{j})_{\R}$ with $j\in[p,q]\cap\Z$. 
If $f$ is $2$-times differentiable on $C$, then we let $f^{(2)}:C\to\mc{L}_{b}(P,\mc{L}_{b}(P,E))$ be the $2$-differential map of 
$f$. If $w$ is any differentiable Hausdorff locally convex space valued map defined on an open set $U$ of $\R^{n}$, then 
for any $l\in[1,n]\cap\Z$ we let $\partial_{l}w$ denote $\frac{\partial w}{\partial e_{l}[n]}$ (which denotes by abuse of language
the partial derivative not the partial differential) and let $\partial_{0}w$ denote $w$.
\par
A smooth manifold is assumed finite dimensional, if $M$ is a smooth manifold, then $\mr{Crt}(M)$ is the set of all smooth charts
of $M$. If $N$ is a smooth manifold, then we let $\mc{C}^{\infty}(M,N)$ be the set of smooth maps defined on $M$ and valued in $N$.
If $(W,\xi)$ is a smooth chart of $M$ and $n$ is the dimension of $M$ and $\mu\in[1,n]\cap\Z$, then we let 
$\xi_{\mu}\coloneqq\Pr_{\mu}\circ\imath_{\xi(W)}^{\R^{n}}\circ\xi$. 
If $E$ is a Hausdorff locally convex linear space over $\K$, then let $\mc{C}^{\infty}(M,E)$ be the linear space 
$\mc{B}^{\infty}(M,E)$ over $\K$ of smooth maps defined on $M$ and with values in $E$ as defined in 
\cite[Definition 1.1]{sil09261621}. 
If $i\in[1,n]\cap\Z$, then we let $\partial_{i}^{\xi,E}:\mc{C}^{\infty}(W,E)\to\mc{C}^{\infty}(W,E)$ be the map defined in 
\cite[Definition 1.4]{sil09261621}.
We convein to denote $\partial_{i}^{\xi,E}$ simply by $\partial_{i}^{\xi}$ whenever it is clear by the context 
which space $E$ is involved  moreover we let $\partial_{0}^{\xi}g=g$. 
Define $\mc{C}_{c}^{\infty}(M,E)\coloneqq\mc{C}^{\infty}(M,E)\cap\mc{C}_{c}(M,E)$.
\par
If $U$ is a locally compact subspace of $\R^{n}$, then we let 
$\lambda_{U}$ be the restriction at $U$ of the Lebesgue measure on $\R^{n}$. $\R^{n}$ is assumed oriented by its standard volume 
form $\wedge_{i=1}^{n}d\Pr_{i}$. For any normed space $T$, any open set $A$ of $T$, any $n\in\Z_{>0}$, any open set $U$ of $\R^{n}$ 
and any differentiable locally convex space valued map $g$ defined on $A\times U$ we let $D_{1}g$ denote 
$\frac{\partial g}{\partial T}$ id est the partial differential of $g$ w.r.t. the subspace $T$, and let $D_{2}g$ denote 
$\frac{\partial g}{\partial\R^{n}}$ id est the partial differential of $g$ w.r.t. the subspace $\R^{n}$ moreover 
for any $i\in[1,n]\cap\Z$ we let $\frac{\partial g}{\partial e_{i}^{2}[n]}$ denote $D_{(\ze,e_{i}[n])}g$, where $D_{v}$ is the 
directional derivative w.r.t. the direction $v$. Then we obtain for any $(\eta,x)\in A\times U$ 
\begin{equation}
\label{08092146}
\frac{\partial g}{\partial e_{i}^{2}[n]}(\eta,x)
=
g^{(1)}(\eta,x)(\ze,e_{i}[n])
=
(D_{2}g)(\eta,x)(e_{i}[n]).
\end{equation}
\end{notations}
\begin{assumption}
\label{12152203}
In this paper we let $\K\in\{\R,\C\}$ and $\K_{1}\in\{\R,\C\}$. 
If $S$ is a locally compact space, then we let $\mc{H}(S)$ be $\mc{H}(S;\K)$.
Let $\mc{Z}$ be a \textbf{quasi-complete} Hausdorff locally convex space over $\K$, let $\Gamma$ be a fundamental system of 
seminorms of $\mc{Z}$. Let $T$ be a normed space over $\R$ and $A$ be an open set of $T$. Let $p\in\Z_{>1}$ and $n\in\Z_{>0}$. 
Let $M$ be an $n$-dimensional oriented smooth manifold with $n\in\Z_{>0}$, $(W,\xi)$ be a positively oriented chart of $M$ and 
$U=\xi(W)$. Let $\omega:A\to\mc{C}^{\infty}(W,\R^{n})$ be such that for every $\eta\in A$ we have that $(W,\omega_{\natural}^{\eta})$ 
is a positively oriented chart of $M$, $\omega^{\eta}(W)=U$ and $\omega_{\natural}^{\ze}=\xi$. Let $X$ be a Banach space over 
$\K_{1}$, $Y$ be the constant family on $[0,n]\cap\Z$ equal to $X$, let $\mh{Y}\coloneqq\prod_{\nu\in[0,n]\cap\Z}Y_{\nu}$, 
$\mf{Y}\coloneqq\prod_{\nu\in[1,n]\cap\Z}Y_{\nu}$ and $\mc{Y}\coloneqq Y_{0}\times\mf{Y}$ all of them provided by the $\max$-norm, 
let $O$ be an open set of $\mc{Y}$. Furthermore let $\mc{B}$ be the family obtained by restricting $Y$ to $[1,n]\cap\Z$, id est 
$\mc{B}_{\nu}=Y_{\nu}$ for any $\nu\in[1,n]\cap\Z$, while $\mc{E}$ be the family on $[1,2]\cap\Z$ such that $\mc{E}_{1}=Y_{0}$ and 
$\mc{E}_{2}=\mf{Y}$. Corrispondently we have the projectors over the components given by
$\Pr_{i}^{Y}:\mh{Y}\to Y_{i}$ for $i\in[0,n]\cap\Z$, $\Pr_{\nu}^{\mc{B}}:\mf{Y}\to Y_{\nu}$ for $\nu\in[1,n]\cap\Z$ 
and by $\Pr_{r}^{\mc{E}}:\mc{Y}\to\mc{E}_{r}$ for $r\in[1,2]\cap\Z$ such that for every $c\in\mh{Y}$, $a\in Y_{0}$ and 
$b\in\mf{Y}$ we have $\Pr_{i}^{Y}c=c_{i}$, $\Pr_{\nu}^{\mc{B}}b=b_{\nu}$, and $\Pr_{1}^{\mc{E}}(a,b)=a$, $\Pr_{2}^{\mc{E}}(a,b)=b$.
\end{assumption}
\section{Introduction}
In this introduction we employ Notation \ref{12152204} and assume Assumption \ref{12152203}. 
The following are the principal original contributions of the present paper:
\begin{itemize}
\item
Def.\ref{08280414} of the map $\mf{A}:\mf{Z}^{O}\times A\to\mc{M}(W,\mc{Z}_{\R})$ which reads explicitly as in \eqref{10011549}.
We call $\mf{A}(u,\eta)$ \emph{the action functional associated with the Lagrangian $\mh{L}\in\mc{C}^{p}(O,\mc{Z})$, the field 
$u\in\mf{Z}^{O}$ and the perturbation paramenter $\eta\in A$}. Here it is worth recalling that $\mc{Z}$ is a quasi-complete 
Hausdorff locally convex space, $O$ an open set of a normed space isometric to a product of $n+1$ copies of the normed space 
$Y_{0}$, $W$ is an open set domain of a chart of an $n$-dimensional manifold $M$, $\mf{Z}^{O}$ is an open set depending by $O$
of the space $\mf{Z}$ of fields defined on $W$ and valued in $Y_{0}$, finally $A$ is an open set of the normed 
space $T$. Furthermore $\mc{M}(W,\mc{Z}_{\R})=\mc{L}_{s}(\mc{H}(W),\mc{Z}_{\R})$ is the locally convex space of 
\textbf{vector measures defined on $W$ and valued in $\mc{Z}_{\R}$} endowed by the topology of simple convergence.
Here we employ a result in our recent work \cite[Thm.3.57]{sil09261621} which associate a $\mc{Z}_{\R}$-valued measure to 
every $\mc{Z}_{\R}$-valued locally integrable compactly supported form of $M$ defined on $W$.
\item
Thm.\ref{07222143} stating that $\mc{A}(u):A\to\mc{M}(W,\mc{Z}_{\R})$ is \emph{differentiable} and determining its differential 
map with $\mc{A}(u)(\eta)\coloneqq\mf{A}(u,\eta)$ for any $u\in\mf{Z}^{O}$ and $\eta\in A$; 
\item
Def.\ref{09302310} the main concept of the paper: the \textbf{universal operatorial Noether current $\mh{N}[u]$ relative to $u$}
which is a $\mc{L}_{b}(T,\mc{Z}_{\R})$-valued $(p-1)$-times continuously differentiable vector field of $M$ on $U$ which under a
certain condition is conserved on shell see Cor.\ref{08110158} and Cor.\ref{12142339};
\item
Cor.\ref{12131752} establishing the \textbf{divergence of $\mh{N}[u]$}.
Here we employ the divergence of a locally convex space valued $C^{p}$-vector field which is constructed and respectively 
displayed in a special feature in the following results of our recent work \cite[Cor.3.41 and Cor.3.48]{sil09261621}.
\item
Thm.\ref{08070358} exhibits the differential of $\mc{A}(u)$ at $\eta=\ze$ and on-shell i.e. by taking $u\in\mr{EL}_{W}(\mh{L})$
the set of solutions defined on $U=\xi(W)$ of the operatorial Euler-Lagrange equation associated with $\mc{L}$ 
Def.\ref{07230430};
\item
\textbf{Cor.\ref{08110158}} the first main result establishing the conservation of the universal operatorial Noether current 
$\mh{N}[u]$ on-shell under the condition that the differential map of $\mc{A}(u)$ on $\ze$ is $\ze$ namely 
$\mc{A}(u)^{(1)}(\ze)=\ze$;
\item
\textbf{Cor.\ref{12142339}} the second main result establishing that under the hypothesis $p=\infty$, $\mc{A}(u)^{(1)}(\ze)=\ze$ 
implies that the divergence of the universal operatorial Noether current $\mh{N}[u]$ on-shell is $\ze$;
\item
Cor.\ref{08110158} and Cor.\ref{12142339} communicate a fundamental message:
\textbf{$\mh{N}[u]$ naturally encodes in a unique object all the possible conservation laws associated with $\mh{L}$ and $\phi$, 
each one of them obtained by evaluating $\mh{N}[u]$ over an element of $T$.}
\item
Cor.\ref{05112017} stating the conservation of operatorial Noether charges relative to $u$ on-shell;
\item
Thm.\ref{08280333} under the hypothesis that $\mf{a}(\ze)$ is differentiable, states that 
if $\mf{a}(\eta)=\mf{a}(\ze)$, then $\phi_{\eta}$ is a symmetry of $\mr{vp}(\mf{a}(\ze))$,
with $\mf{a}(\eta):\mf{Z}^{O}\to\mc{M}(W,\mc{Z}_{\R})$, such that $\mf{a}(\eta)(u)\coloneqq\mf{A}(u,\eta)$,
for any $\eta\in A$ and $u\in\mf{Z}^{O}$; 
\item
Thm.\ref{09142108} provides a sufficient condition ensuring the \emph{differentiability} of $\mf{a}(\ze)$ and establishes its 
differential map;
\item
Cor.\ref{10011525} establishes a characterization of $\mr{vp}(\mf{a}(\ze))$ in terms of the solution set $\mr{sEL}_{W}(\mh{L})$ 
of the functional differential equation defined in Def.\ref{11011831};
\item
Thm.\ref{10301959} states that $\mr{sEL}_{W}(\mh{L})$ and then $\mr{vp}(\mf{a}(\ze))$ is a subset of $\mr{EL}_{W}(\mh{L})$;
\item
Cor.\ref{11031537} establishes a sufficient condition ensuring that the invariance $\mf{a}(\eta)=\mf{a}(\ze)$ would imply that 
$\phi_{\eta}$ is a symmetry of the solution set $\mr{sEL}_{W}(\mh{L})$.
\end{itemize}
Next we provide a coherent outline of the above listed contributions.
\par
Section \ref{09030751}. In Def.\ref{08310159} we define the normed spaces of fields involved in particular the space $\mf{Z}$
a subspace of the normed space $\mc{C}_{c}^{\infty}(W,Y_{0})$, then we construct an open set $\mf{Z}^{O}$ of $\mf{Z}$. 
Next we fix a general action $\phi:A\to\mc{F}(\mf{Z},\mf{Z})$ preserving $\mf{Z}^{O}$ and a map $\psi:A\to\mc{F}(\mf{Z},\mf{T})$ 
satisfying $\psi_{\eta}\circ\imath_{\mf{Z}^{O}}^{\mf{Z}}=\partial^{\xi}\circ\phi_{\eta}\circ\imath_{\mf{Z}^{O}}^{\mf{Z}}$ for every 
$\eta\in A$. Fix $p\in\Z_{>0}\cup\{+\infty\}$. Then we let the Lagrangian $\mh{L}$ be such that 
\begin{equation*}
\mh{L}\in\mc{C}^{p}(O,\mc{Z}),
\end{equation*}
assume $\phi_{\ze}=\mr{Id}_{\mf{Z}}$, and that the following maps are $C^{p}$-functions: the map $\Upphi(u)$ for every $u\in\mf{Z}$ 
and the map $\widehat{\chi}$, both defined on $A\times U$ and induced, via the chart $\xi$, by $\phi$ and respectively by 
$\omega$ see Def.\ref{12141626} and Def.\ref{10051149}, 
\begin{equation*}
\Upphi(u)\in\mc{C}^{p}(A\times U,Y_{0}),
\end{equation*}
and $\imath_{U}^{\R^{n}}\circ\widehat{\chi}\in\mc{C}^{p}(A\times U,\R^{n})$.
As a result the map $\Uppsi(u)$, induced by the action $\psi$, is such that 
\begin{equation*}
\Uppsi(u)\in\mc{C}^{p-1}(A\times U,\mf{Y}).
\end{equation*}
\par
As we shall see the regularity of the maps $\mh{L}$, $\Upphi(u)$, $\Uppsi(u)$ and $\widehat{\chi}$ will be a starting point 
for establishing the differentiability of sections of the following map whose \textbf{Def.\ref{08280414}} employs
\cite[Thm.3.57]{sil09261621}, a result in our previous work
\begin{equation*}
\mf{A}:\mf{Z}^{O}\times A\to\mc{M}(W,\mc{Z}_{\R}),
\end{equation*}
where we recall that $\mc{M}(W,\mc{Z}_{\R})=\mc{L}_{s}(\mc{H}(W),\mc{Z}_{\R})$ is the locally convex space of vector-measures 
defined on $W$ and valued in $\mc{Z}_{\R}$ endowed by the topology of simple convergence.
\par
Next for any $u\in\mf{Z}^{O}$ let $\mc{A}(u):A\to\mc{M}(W,\mc{Z}_{\R})$ be such that $\mc{A}(u)(\eta)\coloneqq\mf{A}(u,\eta)$
for every $\eta\in A$. Instead for any $\eta\in A$ let $\mf{a}(\eta):\mf{Z}^{O}\to\mc{M}(W,\mc{Z}_{\R})$ be such that 
$\mf{a}(\eta)(u)\coloneqq\mf{A}(u,\eta)$ for every $u\in\mf{Z}^{O}$.
For every $u\in\mf{Z}^{O}$ and $\eta\in A$, and for every $f\in\mc{H}(W)$ we have
\begin{equation}
\label{10011549}
\mc{A}(u)(\eta)(f)
=
\int
(\xi^{-1})^{\ast}(f)(x)G_{u}(x,\eta)d\lambda_{U}(x),
\end{equation}
where
\begin{equation*}
\begin{aligned}
G:\mf{Z}^{O}&\to\mc{F}(A\times U,\mc{Z}),
\\
u&\mapsto G_{u}\coloneqq F_{u}\cdot\widehat{J},
\\
F:\mf{Z}^{O}&\to\mc{F}(A\times U,\mc{Z}),
\\
u&\mapsto F_{u}\coloneqq\mh{L}\circ\bigl(\Upphi(u)\boxtimes\Uppsi(u)\bigr)\circ(\beta\boxtimes\widehat{\chi}),
\end{aligned}
\end{equation*}
where $\widehat{J}:A\times U\to\R$ is such that $\widehat{J}(\eta,\cdot)$ is the Jacobian of $\widehat{\chi}(\eta,\cdot)$ for 
every $\eta\in A$. 
\par
We call $\mc{A}(u)(\eta)$ 
\emph{the action functional associated with the Lagrangian $\mh{L}$, the field $u$ and the perturbation paramenter $\eta$} 
and whose action over the function space $\mc{H}(W)$ is given by \eqref{10011549}.
\par
Now the regularity of the map $G_{u}$ and the general result Cor.\ref{01031245b} in Appendix on the commutation 
between differentiation and integration with respect to a general measure, permits in 
\textbf{Thm.\ref{07222143}} to prove that $\mc{A}(u)$ is differentiable and to display its differential map. 
After having introduced in Def.\ref{07230430} the operatorial Euler-Lagrange equation relative to $\mh{L}$ 
and the set $\mr{EL}_{W}(\mh{L})$ of those $u\in\mf{Z}^{O}$ satisfying this equation, we construct in 
\textbf{Def.\ref{09302310}} the main concept of this paper 
namely \textbf{the universal operatorial Noether current relative to $u$} as a 
$\mc{L}_{b}(T,\mc{Z}_{\R})$-valued $(p-1)$-times continuously differentiable vector field of $M$ on $U$ namely 
\begin{equation*}
\mh{N}[u]
\in\mc{C}^{p-1}(U,\mc{L}_{b}(T,\mc{Z}_{\R}))\otimes_{\mc{C}^{\infty}(U)}\Gamma(U,TM).
\end{equation*}
The divergence on a chart of a smooth locally convex vector valued vector field provided in \cite[Cor.3.48]{sil09261621}
permits to exhibit in \textbf{Cor.\ref{12131752}} the divergence of $\mh{N}[u]$ in the special case $p=+\infty$,
where the divergence of a general $q$-times continuously differentiable locally convex 
vector valued vector field is established in \cite[Cor.3.41]{sil09261621} where $q\in\Z_{>0}\cup\{+\infty\}$.
Now in \textbf{Thm.\ref{08070358}} we specialize the result in Thm.\ref{07222143} by obtaining the 
differential of $\mc{A}(u)$ at $\eta=\ze$ and on-shell namely by taking $u\in\mr{EL}_{W}(\mh{L})$, as a consequence we 
establish in \textbf{Cor.\ref{08110158}} the first main result of this paper namely 
\emph{the conservation of the universal operatorial Noether current under the condition that the differential map of $\mc{A}(u)$
evaluated in $\ze$ is $\ze$ namely $\mc{A}(u)^{(1)}(\ze)=\ze$} and in \textbf{Cor.\ref{12142339}} the second main result of the 
paper we state that \emph{under the hypothesis $p=\infty$, $\mc{A}(u)^{(1)}(\ze)=\ze$ implies that the divergence of the 
universal operatorial Noether current is $\ze$}. As an application of Cor.\ref{08110158} we obtain in 
\textbf{Cor.\ref{05112017}} the conservation of operatorial Noether charges relative to $u$. 
\par
Now we want to emphasize what follows.
Three are the preeminent original features that permit the introduction of other new components and grant the novel 
contributions outlined in this Introduction. 
\begin{enumerate}
\item The construction of the map $\mf{A}$ with values in the locally convex linear space of 
\textbf{$\mc{Z}_{\R}$-valued measures on $W$} and defined on $\mf{Z}^{O}\times A$ with the property that $\mc{A}(u)$ and under 
additional conditions also $\mf{a}(\eta)$ are differentiable. Here it is worth recalling that $\mc{Z}$, is a general 
quasi-complete Hausdorff locally convex space, while $T$ and $Y_{0}$ are possibly infinite dimensional normed spaces, 
with the property that $\mc{A}(u)$ and under additional conditions also $\mf{a}(\eta)$ are differentiable. 
In the construction we employ a result in our recent work \cite[Thm.3.57]{sil09261621} which associate a $\mc{Z}_{\R}$-valued 
measure to every $\mc{Z}_{\R}$-valued locally integrable compactly supported form of $M$ defined on $W$.
\item
The construction of the $\mc{L}_{b}(T,\mc{Z}_{\R})$-valued vector field $\mh{N}[u]$ that, as a consequence of our main results
Cor.\ref{08110158} and Cor.\ref{12142339},
\textbf{naturally encodes in a unique object all the possible conservation laws associated with $\mh{L}$ and $\phi$, 
each one of them obtained evaluating $\mh{N}[u]$ over an element of $T$.}
\label{10041037}
\item
The construction of the \textbf{divergence} of the operator valued vector field $\mh{N}[u]$.
Here we employ the divergence of a locally convex space valued $C^{p}$-vector field which is constructed and respectively 
displayed in a special feature in the following results of our recent work: \cite[Cor.3.41 and Cor.3.48]{sil09261621}.
\end{enumerate}
Concerning item \eqref{10041037} we can be more specific. Indeed in Rmk.\ref{11051131} we reobtain as special cases of our 
setting the classical results for instance those in \cite{bog2}, by letting $\mc{Z}=\R$, $\mr{dim}Y_{0}<\infty$, $T=\R^{s}$ with 
$s\in\Z_{>0}$ and by evaluating $\mh{N}[u]$ over each element of the standard basis of $\R^{s}$. Notice that \cite{bog2} has been 
for decades the standard reference for the Lagrangian approach to the quantum field theories where the classical fields after 
quantization become operator-valued maps. In this approach given a Lagrangian say $L$, the Nother currents corresponding to the 
symmetries of the action functional defined by $L$, and their associated conserved field-depending charges are of paramount 
importance. Indeed when the classical fields are replaced by their quantized version, these charges represent the physical 
observables of the quantum theory relative to $L$.
\par
Section \ref{11050951} analyzes symmetries of $\mr{vp}(\mf{a}(\ze))$. More precisely under conditions ensuring the 
differentiability of the map $\mf{a}(\ze)$ and for any differentiable map $g:\mf{Z}^{O}\to\mc{M}(W,\mc{Z}_{\R})$ by letting 
$\mr{vp}(g)$ be the set of $u\in\mf{Z}^{0}$ for which $g^{(1)}(u)=\ze$ where $g^{(1)}$ is the differential map of $g$, we search 
for functions mapping $\mr{vp}(\mf{a}(\ze))$ into itself.
Now under the hypothesis that $\mf{a}(\ze)$ is differentiable, 
\textbf{Thm.\ref{08280333}} states that an invariance of the action $\mf{a}$ induces a symmetry of $\mr{vp}(\mf{a}(\ze))$, 
more exactly if $\mf{a}(\eta)=\mf{a}(\ze)$, then $\phi_{\eta}$ is a symmetry of $\mr{vp}(\mf{a}(\ze))$.
\par
Section \ref{09171942}. We have seen that Thm.\ref{08280333} requires the differentiability of $\mf{a}(\ze)$.
Now \textbf{Thm.\ref{09142108}} provides a sufficient condition ensuring the differentiability of 
$\mf{a}(\ze)$ and establishes its differential map; \textbf{Cor.\ref{10011525}} displays a characterization of 
$\mr{vp}(\mf{a}(\ze))$ in terms of the solution set $\mr{sEL}_{W}(\mh{L})$ of a functional differential equation 
\textbf{Def.\ref{11011831}}; while \textbf{Thm.\ref{10301959}} establishes that $\mr{sEL}_{W}(\mh{L})$ and then 
$\mr{vp}(\mf{a}(\ze))$ is a subset of $\mr{EL}_{W}(\mh{L})$.
As a result of Cor.\ref{10011525}, Thm.\ref{09142108} and Thm.\ref{08280333} we obtain in \textbf{Cor.\ref{11031537}} 
a sufficient condition ensuring that the invariance $\mf{a}(\eta)=\mf{a}(\ze)$ would imply that $\phi_{\eta}$ is a symmetry of the 
solution set $\mr{sEL}_{W}(\mh{L})$.
Finally Cor.\ref{10042025} summarizes the consequences of the main results of the paper under the hypothesis that 
$\mf{a}$ is a constant map.
\par
In the final section \ref{08300401} we apply our framework to a normed space valued generalization of the Dirac equation.
\section{Universal operatorial Noether current}
\label{09030751}
\begin{definition}
Let $\mf{d}:\mc{Y}\to\mh{Y}$ be the isometry such that $\Pr_{0}^{Y}\circ\mf{d}=\Pr_{1}^{\mc{E}}$ and 
$\Pr_{\nu}^{Y}\circ\mf{d}=\Pr_{\nu}^{\mc{B}}\circ\Pr_{2}^{\mc{E}}$ for any $\nu\in[1,n]\cap\Z$.
Furthermore define 
\begin{equation*}
\begin{aligned}
\mf{e}&\coloneqq\mf{d}^{-1}
\\
\mf{d}_{O}&\coloneqq\mf{d}\circ\imath_{O}^{\mc{Y}}
\\
\mf{e}_{O}&\coloneqq\mf{e}\circ\imath_{\mf{d}(O)}^{\mh{Y}},
\end{aligned}
\end{equation*}
and if $\mh{L}:O\to\mc{Z}$, then define 
\begin{equation*}
\mh{L}_{\mf{e}}\coloneqq\mh{L}\circ(\mf{e}_{O})_{\natural}.
\end{equation*}
\end{definition}
We have $(\Pr_{0}^{Y}\circ\mf{d})(a,b)=a$ and $(\Pr_{\nu}^{Y}\circ\mf{d})(a,b)=b_{\nu}$ for every $a\in Y_{0}$ and $b\in\mf{Y}$. 
In the following lemma we employ to the special case of the category of sets and maps the operator $(\cdot)^{\dagger}$ defined in 
the general case in \cite[p.12]{sil}. 
\begin{lemma} 
\label{10181658}
Let $C\subseteq O$ and $\mh{L}:O\to\mc{Z}$ be differentiable on $C$, then 
\begin{equation}
\label{10181658a}
\begin{aligned}
\mh{L}^{(1)}
&=
(\Pr_{1}^{\mc{E}})^{\dagger}\circ\frac{\partial\mh{L}}{\partial Y_{0}}
+
(\Pr_{2}^{\mc{E}})^{\dagger}\circ\frac{\partial\mh{L}}{\partial\mf{Y}}
\\
&=
\sum_{i\in[0,n]\cap\Z}(\Pr_{i}^{Y}\circ\mf{d})^{\dagger}\circ\frac{\partial\mh{L}_{\mf{e}}}{\partial Y_{i}}\circ(\mf{d}_{O})_{\natural}
\end{aligned}
\end{equation}
in particular
\begin{equation}
\label{10181658b}
\frac{\partial\mh{L}}{\partial\mf{Y}}
=
\sum_{\nu\in[1,n]\cap\Z}(\Pr_{\nu}^{\mc{B}})^{\dagger}\circ\frac{\partial\mh{L}_{\mf{e}}}{\partial Y_{\nu}}\circ(\mf{d}_{O})_{\natural}.
\end{equation}
\end{lemma}
\begin{proof}
The 1th equality in \eqref{10181658a} follows directly since Cor.\ref{02211942} applied to $\mh{L}$. 
Next $\mh{L}_{\mf{e}}$ is differentiable since Prp.\ref{813}, Prp.\ref{05231704} and Thm.\ref{821}
but $\mh{L}=\mh{L}_{\mf{e}}\circ(\mf{d}_{O})_{\natural}$ so, since Prp.\ref{813}, Prp.\ref{05231704} and Thm.\ref{821} we obtain 
for any $s\in\mc{Y}$ that $\mh{L}^{(1)}(t)(s)=\mh{L}_{\mf{e}}^{(1)}(\mf{d}(t))(\mf{d}(s))$, then the 2th equality in 
\eqref{10181658a} follows since Cor.\ref{02211942} applied to $\mh{L}_{\mf{e}}$. Next by Cor.\ref{02211942} applied to $\mh{L}$ we
have that $\frac{\partial\mh{L}}{\partial\mf{Y}}=(\mf{i}_{2}^{\mc{E}})^{\dagger}\circ\mh{L}^{(1)}$ so, \eqref{10181658b} follows by 
\eqref{10181658a}.
\end{proof}
Under the hypothesis of Lemma \ref{10181658} explicitly it states that for every $t\in C$ and for every $a\in Y_{0}$ and 
$b\in\mf{Y}$ we have 
\begin{equation*}
\begin{aligned}
\mh{L}^{(1)}(t)(a,b)
&=
\frac{\partial\mh{L}}{\partial Y_{0}}(t)(a)
+
\frac{\partial\mh{L}}{\partial\mf{Y}}(t)(b)
\\
&=
\frac{\partial\mh{L}_{\mf{e}}}{\partial Y_{0}}(\mf{d}(t))(a)
+
\sum_{\nu\in[1,n]\cap\Z}
\frac{\partial\mh{L}_{\mf{e}}}{\partial Y_{\nu}}(\mf{d}(t))(b_{\nu}),
\end{aligned}
\end{equation*}
and 
\begin{equation*}
\frac{\partial\mh{L}}{\partial\mf{Y}}(t)(b)
=
\sum_{\nu\in[1,n]\cap\Z}
\frac{\partial\mh{L}_{\mf{e}}}{\partial Y_{\nu}}(\mf{d}(t))(b_{\nu}).
\end{equation*}
\begin{definition}
Define 
\begin{equation*}
\mc{C}_{c,1}^{\infty}(M,E)\coloneqq\{f\in\mc{C}_{c}^{\infty}(M,E)\,\vert\,(\forall\,(V,\rho)\in\mr{Crt}(M))
(\forall\nu\in[1,n]\cap\Z)(\mr{supp}(\partial_{\nu}^{\rho}f)\in\mr{Cmp}(V)\}.
\end{equation*}
\end{definition}
\begin{definition}
\label{08310159}
Let 
\begin{enumerate}
\item
$\mf{Z}_{U}^{\emptyset}$ be the $\K_{1}$-normed space whose underlying linear space is $\mc{C}_{c}(U,Y_{0})$ and whose norm is 
\begin{equation*}
q_{U}^{\emptyset}(f)\coloneqq\sup_{x\in U}\|f(x)\|_{Y_{0}};
\end{equation*}
\item
$\mf{Z}^{\emptyset}$ be the $\K_{1}$-normed space whose underlying linear space is $\mc{C}_{c}(W,Y_{0})$ and whose norm is 
\begin{equation*}
q^{\emptyset}(g)\coloneqq\sup_{x\in W}\|g(x)\|_{Y_{0}};
\end{equation*}
\item
$\mf{Z}_{U}$ be the $\K_{1}$-normed space whose underlying linear space is $\mc{C}_{c,1}^{\infty}(U,Y_{0})$ and whose norm is 
\begin{equation*}
\widehat{q}_{U}(f)\coloneqq\max_{\nu\in[0,n]\cap\Z}\sup_{x\in U}\|(\partial_{\nu}f)(x)\|_{Y_{0}};
\end{equation*}
\item
$\mf{Z}$ be the $\K_{1}$-normed space whose underlying linear space is $\mc{C}_{c,1}^{\infty}(W,Y_{0})$ and whose norm is 
\begin{equation*}
\widehat{q}(g)\coloneqq\max_{\nu\in[0,n]\cap\Z}\sup_{x\in W}\|(\partial_{\nu}^{\xi}g)(x)\|_{Y_{0}};
\end{equation*}
\item
$\mf{T}$ be the $\K_{1}$-normed space whose underlying linear space is $\mc{C}_{c}^{\infty}(W,\mf{Y})$ and whose norm is 
\begin{equation*}
\widehat{t}(l)\coloneqq\max_{\mu\in[1,n]\cap\Z}\sup_{x\in W}\|(\Pr_{\mu}^{\mc{B}}\circ l)(x)\|_{Y_{\mu}}.
\end{equation*}
\end{enumerate}
\end{definition}
\begin{definition}
For any $i\in[0,n]\cap\Z$ we let $\widehat{\partial_{i}^{\xi}}:\mf{Z}\to\mf{Z}^{\emptyset}$, $g\mapsto\partial_{i}^{\xi}g$
and $\widehat{\partial}_{i}:\mf{Z}_{U}\to\mf{Z}_{U}^{\emptyset}$, $f\mapsto\partial_{i}f$.
\end{definition}
\begin{definition}
For any $y\in W$ and $x\in U$ define $\mf{e}_{y}^{\mf{Z}}:\mf{Z}\to\mc{Y}$, $\mf{e}_{x}^{\mf{Z}_{U}}:\mf{Z}_{U}\to\mc{Y}$ such that 
\begin{equation*}
\begin{aligned}
\mf{e}_{y}^{\mf{Z}}
&\coloneqq
\mf{d}^{-1}\circ\prod_{\nu\in[0,n]\cap\Z}^{\boxtimes}\mr{ev}_{y}^{\mf{Z}^{\emptyset}}\circ\widehat{\partial_{\nu}^{\xi}},
\\
\mf{e}_{x}^{\mf{Z}_{U}}
&\coloneqq
\mf{d}^{-1}\circ\prod_{\nu\in[0,n]\cap\Z}^{\boxtimes}\mr{ev}_{x}^{\mf{Z}_{U}^{\emptyset}}\circ\widehat{\partial}_{\nu};
\end{aligned}
\end{equation*}
moreover define the sets $\mf{Z}^{O}\subseteq\mf{Z}$ and $\mf{Z}_{U}^{O}\subseteq\mf{Z}_{U}$ 
such that 
\begin{equation*}
\begin{aligned}
\mf{Z}^{O}&\coloneqq\bigcap_{y\in W}\overset{-1}{\overbrace{\mf{e}_{y}^{\mf{Z}}}}(O),
\\
\mf{Z}_{U}^{O}&\coloneqq\bigcap_{x\in U}\overset{-1}{\overbrace{\mf{e}_{x}^{\mf{Z}_{U}}}}(O).
\end{aligned}
\end{equation*}
\end{definition}
\begin{lemma}
\label{09171929}
$\mf{Z}^{O}$ is an open set of $\mf{Z}$ and $\mf{Z}_{U}^{O}$ is an open set of $\mf{Z}_{U}$.
\end{lemma}
\begin{proof}
For any $y\in W$ and $x\in U$ the maps $\mr{ev}_{y}^{\mf{Z}^{\emptyset}}$ and $\mr{ev}_{x}^{\mf{Z}_{U}^{\emptyset}}$ are linear bounded 
operators so continuous, and $\widehat{\partial_{\nu}^{\xi}}$ and $\widehat{\partial}_{\nu}$ are linear and bounded operators so 
continuous, therefore the maps $\mf{e}_{y}^{\mf{Z}}$ and $\mf{e}_{x}^{\mf{Z}_{U}}$ are continuous since \cite[I.25 Prp.1]{BourGT} and 
since $\mf{d}^{-1}$ is an isometry thus, the statement follows since \cite[I.9 Thm.1]{BourGT} and the fact that $O$ is an open 
set of $\mc{Y}$.
\end{proof}
\begin{definition}
If $\eta\in A$ and $V\subseteq\R^{n}$, define $\mf{o}_{\eta}^{V}:V\to A$, $x\mapsto\eta$.
\end{definition}
We convein to let $\mf{o}_{\eta}$ denote $\mf{o}_{\eta}^{U}$.
\begin{lemma}
\label{08080256} 
Let $E$ be a Hausdorff locally convex space, $(\eta,x)\in A\times U$, $f:A\times U\to E$ be differentiable at $(\eta,x)$ and 
let $k\in[1,n]\cap\Z$. Then
\begin{equation*}
\frac{\partial f}{\partial e_{k}^{2}[n]}\circ(\mf{o}_{\eta}\boxtimes\mr{Id}_{U})(x)
=
\frac{\partial}{\partial e_{k}[n]}\bigl(f\circ(\mf{o}_{\eta}\boxtimes\mr{Id}_{U})\bigr)(x).
\end{equation*}
\end{lemma}
\begin{proof}
Let $t\in\R^{n}$ thus,
\begin{equation}
\label{08080256a}
\begin{aligned}
\left(f\circ(\mf{o}_{\eta}\boxtimes\mr{Id}_{U})\right)^{(1)}(x)(t)
&=
f^{(1)}((\mf{o}_{\eta}\boxtimes\mr{Id}_{U})(x))((\imath_{A}^{T}\circ\mf{o}_{\eta})^{(1)}(x)t,\mr{Id}_{U\to\R^{n}}^{(1)}(x)t)
\\
&=
f^{(1)}((\mf{o}_{\eta}\boxtimes\mr{Id}_{U})(x))(\ze,t),
\end{aligned}
\end{equation}
where the 1th equality follows since Thm.\ref{821} and Prp.\ref{815}, the 2th one since Prp.\ref{812} and Prp.\ref{813}.
Next by the 1th equality in \eqref{08092146} we have
\begin{equation}
\label{08080256b}
\frac{\partial f}{\partial e_{k}^{2}[n]}\circ(\mf{o}_{\eta}\boxtimes\mr{Id}_{U})(x)
=
f^{(1)}((\mf{o}_{\eta}\boxtimes\mr{Id}_{U})(x))(\ze,e_{k}[n]);
\end{equation}
while by Prp.\ref{08091051} we have 
\begin{equation}
\label{08080256c}
\frac{\partial}{\partial e_{k}[n]}\left(f\circ(\mf{o}_{\eta}\boxtimes\mr{Id}_{U})\right)(x)
=
\left(f\circ(\mf{o}_{\eta}\boxtimes\mr{Id}_{U})\right)^{(1)}(x)(e_{k}[n]).
\end{equation}
The statement then follows since (\ref{08080256a},\ref{08080256b},\ref{08080256c}).
\end{proof}
\begin{definition}
For any $(V,\gamma)\in\mr{Crt}(M)$ such that $V=W$ define 
\begin{equation*}
\partial^{\gamma}:\mf{Z}\to\mf{T},
\end{equation*}
such that for every $g\in\mf{Z}$ and $i\in[1,n]\cap\Z$ we have 
\begin{equation*}
\Pr_{i}^{\mc{B}}\circ(\partial^{\gamma}g)\coloneqq\partial_{i}^{\gamma}g.
\end{equation*}
Next define  
\begin{equation*}
\partial:\mf{Z}_{U}\to\mf{T}_{U},
\end{equation*}
such that for every $f\in\mf{Z}_{U}$ and $i\in[1,n]\cap\Z$
\begin{equation*}
\Pr_{i}^{\mc{B}}\circ(\partial f)=\frac{\partial}{\partial e_{i}}f.
\end{equation*}
\end{definition}
The above definition is well-set indeed the property of product of sets ensures that $\partial^{\xi}g$ and $\partial f$ 
respectively are well-defined maps on $W$ and on $U$ respectively and valued in $\mf{Y}$, while Cor.\ref{8126} ensures that 
$\partial^{\xi}g$ and $\partial f$ are smooth. 
\begin{definition}
\label{10051149}
Let $i\in[1,n]\cap\Z$, define 
\begin{equation*}
\begin{cases}
\chi:A\to\mc{C}^{\infty}(U,U),
\\
\widehat{\chi}:A\times U\to U,
\\
\chi_{i}:A\to\mc{C}^{\infty}(U,\R),
\\
\widehat{\chi}_{i}:A\times U\to\R,
\\
M:A\to\mc{C}^{\infty}(U,\mathbb{M}_{n}(\R)),
\\
\widehat{M}:A\times U\to\mathbb{M}_{n}(\R),
\\
J:A\to\mc{C}^{\infty}(U,\R),
\\
\widehat{J}:A\times U\to\R,
\\
\sigma:A\to\mc{C}^{\infty}(U,U),
\\
\widehat{\sigma}:A\times U\to U,
\\
\upzeta:A\to\mc{C}^{\infty}(W,W),
\\
\widehat{\upzeta}:A\times W\to W,
\\
\widehat{\omega}:A\times W\to U;
\end{cases}
\end{equation*}
such that for every $\eta\in A$, $x\in U$, $y\in W$, and $\mu,\nu\in[1,n]\cap\Z$ we have 
\begin{equation*}
\begin{aligned}
\chi(\eta)
&\coloneqq
\xi\circ(\omega_{\natural}^{\eta})^{-1},
\\
\widehat{\chi}(\eta,x)
&\coloneqq
\chi(\eta)(x),
\\
\chi_{i}(\eta)
&\coloneqq
\Pr_{i}\circ\imath_{U}^{\R^{n}}\circ\xi\circ(\omega_{\natural}^{\eta})^{-1},
\\
\widehat{\chi}_{i}
&\coloneqq
\Pr_{i}\circ\imath_{U}^{\R^{n}}\circ\widehat{\chi},
\\
M(\eta)(x)_{\mu,\nu}
&\coloneqq
\frac{\partial\chi_{\mu}(\eta)}{\partial e_{\nu}[n]}(x),
\\
\widehat{M}(\eta,x)_{\mu,\nu}
&\coloneqq
\frac{\partial\widehat{\chi}_{\mu}}{\partial e_{\nu}^{2}[n]}(\eta,x),
\\
J_{\eta}
&\coloneqq
\mr{det}\circ(M(\eta)),
\\
\widehat{J}
&\coloneqq
\mr{det}\circ\widehat{M},
\\
\sigma(\eta)
&\coloneqq
\omega_{\natural}^{\eta}\circ\xi^{-1},
\\
\widehat{\sigma}(\eta,x)
&\coloneqq
\sigma(\eta)(x),
\\
\upzeta(\eta)
&\coloneqq
\xi^{-1}\circ\omega_{\natural}^{\eta},
\\
\widehat{\upzeta}(\eta,y)
&\coloneqq
\upzeta(\eta)(y),
\\
\widehat{\omega}(\eta,y)
&\coloneqq
\omega^{\eta}(y).
\end{aligned}
\end{equation*}
\end{definition}
Notice that for any $\eta\in A$, $J_{\eta}$ is the Jacobian of $(\omega_{\natural}^{\eta})^{-1}$ with respect to the charts 
$(U,\mr{Id}_{U})$ and $(W,\xi)$ and also the Jacobian of $\chi(\eta)$ with respect to $(U,\mr{Id}_{U})$. Since positively 
oriented charts are consistently oriented we obtain that  $J_{\eta}$ is a strictly positive map. 
\begin{assumption}
For the remaining of the paper we let $\phi:A\to\mc{F}(\mf{Z},\mf{Z})$ such that 
$\phi_{\eta}(\mf{Z}^{O})\subseteq\mf{Z}^{O}$ for every $\eta\in A$ and $\psi:A\to\mc{F}(\mf{Z},\mf{T})$ such that 
\begin{equation}
\label{08161354}
(\forall\eta\in A)(\psi_{\eta}\circ\imath_{\mf{Z}^{O}}^{\mf{Z}}=\partial^{\xi}\circ\phi_{\eta}\circ\imath_{\mf{Z}^{O}}^{\mf{Z}}).
\end{equation}
\end{assumption}
\begin{definition}
\label{12141626}
Let $\eta\in A$, $u\in\mf{Z}$ define 
\begin{equation*}
\begin{aligned}
u^{\xi}
&\coloneqq
u\circ\xi^{-1}
\\
\phi_{\eta}^{\xi}(u)
&\coloneqq
\phi_{\eta}(u)\circ\xi^{-1};
\end{aligned}
\end{equation*}
furthermore let $i\in[1,n]\cap\Z$ and define 
\begin{equation*}
\begin{aligned}
\Upphi:
\mf{Z}
&\to
\mc{F}(A\times U,Y_{0}),\,u\mapsto((\eta,x)\mapsto\phi_{\eta}^{\xi}(u)(x))
\\
\psi_{\eta}^{\xi}(u)  
&\coloneqq
\psi_{\eta}(u)\circ\xi^{-1}
\\
\Uppsi:
\mf{Z}
&\to
\mc{F}(A\times U,\mf{Y}),\,u\mapsto((\eta,x)\mapsto\psi_{\eta}^{\xi}(u)(x))
\\
\Uppsi_{i}:
\mf{Z}
&\to
\mc{F}(A\times U,X),\,\Uppsi_{i}(u)\coloneqq\Pr_{i}^{\mc{B}}\circ(\Uppsi(u)).
\end{aligned}
\end{equation*}
\end{definition}
\begin{convention}
Given any two maps $f:K\to L$ and $g:N\to P$ with $f(K)\subseteq N$, in order to simplify the notation, sometime and by abuse of
language we let $g\circ f$ denote $g\circ\imath_{f(K)}^{N}\circ f_{\natural}$.
\end{convention}
\begin{remark}
\label{10161240}
Let $u\in\mf{Z}^{O}$, $\eta\in A$ and $y\in W$ then, $(\phi_{\eta}(u)\boxtimes\psi_{\eta}(u))(y)=\mf{e}_{y}^{\mf{Z}}(\phi_{\eta}(u))$
but $\phi_{\eta}(u)\in\mf{Z}^{O}$ then by definition of $\mf{Z}^{O}$ we obtain $(\phi_{\eta}(u)\boxtimes\psi_{\eta}(u))(y)\in O$.
In particular since for every $x\in U$ we have 
$\bigl(\Upphi(u)\boxtimes\Uppsi(u)\bigr)(\eta,x)=(\phi_{\eta}(u)\boxtimes\psi_{\eta}(u))(\xi^{-1}(x))$, then 
$\bigl(\Upphi(u)\boxtimes\Uppsi(u)\bigr)(\eta,x)\in O$. 
\end{remark}
This fact permits to provide the following:
\begin{definition}
\label{09030748}
Define $\beta:A\times U\to A,\,(\eta,x)\mapsto\eta$ moreover for any $\mh{L}:O\to\mc{Z}$ define 
\begin{equation*}
\begin{aligned}
F:\mf{Z}^{O}&\to\mc{F}(A\times U,\mc{Z}),
\\
u&\mapsto F_{u}\coloneqq\mh{L}\circ\bigl(\Upphi(u)\boxtimes\Uppsi(u)\bigr)\circ(\beta\boxtimes\widehat{\chi}),
\\
G:\mf{Z}^{O}&\to\mc{F}(A\times U,\mc{Z}),
\\
u&\mapsto G_{u}\coloneqq F_{u}\cdot\widehat{J}.
\end{aligned}
\end{equation*}
\end{definition}
\begin{definition}
\label{05101630}
We call $\lr{\mh{L},\phi}{\omega}$ an appropriate transformation system on $X$ iff
\begin{enumerate}
\item
\begin{equation*}
\mh{L}\in\mc{C}^{p}(O,\mc{Z}),
\end{equation*}
\label{09261410a}
\item
\begin{equation*}
\phi_{\ze}=\mr{Id}_{\mf{Z}};
\end{equation*}
\item
\begin{equation*}
\Upphi(\mf{Z})\subseteq\mc{C}^{p}(A\times U,Y_{0});
\end{equation*}
\label{051016303}
\item
\begin{equation*}
\imath_{U}^{\R^{n}}\circ\widehat{\chi}\in\mc{C}^{p}(A\times U,\R^{n}).
\end{equation*}
\label{06291612}
\end{enumerate}
\end{definition}
\begin{remark}
\label{051016304}
Def.\ref{05101630}\eqref{051016303} and \eqref{08161354} imply that 
\begin{equation*}
\Uppsi(\mf{Z})\subseteq\mc{C}^{p-1}(A\times U,\mf{Y}).
\end{equation*}
\end{remark}
\begin{assumption}
\label{08101151a}
For the remaining of the present subsection \ref{09030751} we let $\lr{\mh{L},\phi}{\omega}$ be an appropriate 
transformation system on $X$. 
\end{assumption}
\begin{remark}
\label{06132050}
Let $u\in\mf{Z}^{O}$ so, $\phi_{\ze}^{\omega}(u)=u^{\xi}$ and $\psi_{\ze}^{\omega}(u)=\partial u^{\xi}$ next let $\eta\in A$, then 
\begin{equation*}
\begin{aligned}
(\forall i\in[1,n]\cap\Z)(\forall x\in U)
(\Uppsi_{i}(u)(\eta,x)=\frac{\partial(\Upphi(u))}{\partial e_{i}^{2}[n]
}(\eta,x))
&\Leftrightarrow
\\
\psi_{\eta}(u)\circ\xi^{-1}=\partial(\phi_{\eta}(u)\circ\xi^{-1})
&\Leftrightarrow
\\
\psi_{\eta}(u)\circ\xi^{-1}=\partial^{\xi}(\phi_{\eta}(u))\circ\xi^{-1}
&\Leftrightarrow
\\
\psi_{\eta}(u)=\partial^{\xi}(\phi_{\eta}(u));
\end{aligned}
\end{equation*}
where the second equivalence follows since we have for any chart $(W,\gamma)$ of $M$ whose domain is $W$ and such that 
$\gamma(W)=U$ that 
\begin{equation}
\label{08060435}
(\forall u\in\mf{Z})((\partial^{\gamma}u)\circ\gamma^{-1}=\partial(u\circ\gamma^{-1})).
\end{equation}
\end{remark}
\begin{remark}
\label{09132217}
Define $\delta:A\times W\to A$, $(\eta,y)\mapsto\eta$.
Let $\eta\in A$ and $u\in\mf{Z}$. Since \eqref{08060435} we have 
\begin{equation*}
\partial^{\omega_{\natural}^{\eta}}(u\circ\upzeta(\eta))
=
\partial^{\xi}(u)\circ\upzeta(\eta),
\end{equation*}
therefore we have for every $y\in W$ and every $u\in\mf{Z}^{O}$ 
\begin{equation*}
\left[\left(\Upphi(u)\boxtimes\Uppsi(u)\right)\circ(\delta\boxtimes\widehat{\omega})\right](\eta,y)
=
\left[\left(\phi_{\eta}(u)\circ\upzeta(\eta)\right)
\boxtimes
\partial^{\omega_{\natural}^{\eta}}\left(\phi_{\eta}(u)\circ\upzeta(\eta)\right)\right](y).
\end{equation*}
\end{remark}
\begin{definition}
\label{05101649}
If $u:W\to Y_{0}$, $v:W\to\mf{Y}$, and $w:U\to Y_{0}$, $z:U\to\mf{Y}$ such that $(u\boxtimes v)(W)\subseteq O$ and 
$(w\boxtimes z)(U)\subseteq O$, then define $\mf{L}(u,v)\coloneqq\mh{L}\circ(u\boxtimes v)$ and 
$\mr{L}(w,z)\coloneqq\mh{L}\circ(w\boxtimes z)$.
\end{definition}
Let us provide the following general
\begin{lemma}
\label{10161358}
Let $S$ and $Z$ be locally compact spaces, let $\mu\in\mc{M}(S,\mc{Z}_{\R})$ and $L:S\to Z$ be such that for every 
$\gamma\in\mc{L}(\mc{Z}_{\R},\R)$ we have that $L$ is $\gamma\circ\mu$-measurable and for every $g\in\mc{H}(Z)$ the map 
$g\circ L$ is essentially $\mu$-integrable. Then the map $L(\mu):\mc{H}(Z)\to\mc{Z}_{\R}$ defined by 
\begin{equation}
\label{10161406}
(\forall g\in\mc{H}(Z))(L(\mu)(g)\coloneqq\mu(g\circ L)). 
\end{equation}
is such that $L(\mu)\in\mc{M}(Z,\mc{Z}_{\R})$. 
\end{lemma}
\begin{proof}
The map $L(\mu)$ as in \eqref{10161406} is well-defined and it is linear since $g\circ L$ is essentially $\mu$-integrable 
for every $g\in\mc{H}(Z)$. Let $\gamma\in\mc{L}(\mc{Z}_{\R},\R)$ thus, $g\circ L$ is essentially $\gamma\circ\mu$-integrable 
for every $g\in\mc{H}(Z)$ and $L$ is $\gamma\circ\mu$-measurable so, $L$ is $\gamma\circ\mu$-proper \cite[Def.1 V.63]{IntBourb} 
therefore we have 
\begin{equation}
\label{10161351}
\begin{aligned}
L(\gamma\circ\mu):\mc{H}(Z)\ni g&\mapsto\int(g\circ L)\,d(\gamma\circ\mu);
\\
L(\gamma\circ\mu)&\in\mc{M}(Z,\R).
\end{aligned}
\end{equation}
Next by \eqref{10161406} and the 1th relation in \eqref{10161351} we have $\gamma\circ(L(\mu))=L(\gamma\circ\mu)$, then we 
conclude that $L(\mu)\in\mc{M}(Z,\mc{Z}_{\R})$ by the 2th relation in \eqref{10161351} and since \cite[VI.19 Rmk.3]{IntBourb}.
\end{proof}
Now in order to provide the next fundamental definition we need a result in \cite{sil09261621} to which we refer
the reader for the notation there employed. That said we have that \cite[Thm.3.57]{sil09261621} and Lemma \ref{10161358} 
legitimate the following
\begin{definition}
\label{08280414}
Define $\mf{A}:\mf{Z}^{O}\times A\to\mc{M}(W,\mc{Z}_{\R})$ such that for every $(u,\eta)\in\mf{Z}^{O}\times A$
\begin{equation*}
\mf{A}(u,\eta)\doteqdot
\upzeta(\eta)(\mf{m}_{\mf{L}(\phi_{\eta}(u),\psi_{\eta}(u))\otimes\wedge_{i=1}^{n}d\xi_{i}}).
\end{equation*}
For every $u\in\mf{Z}^{O}$, define $\mc{A}(u):A\to\mc{M}(W,\mc{Z}_{\R})$, such that $\mc{A}(u)(\eta)\coloneqq\mf{A}(u,\eta)$ for 
every $\eta\in A$ and define $\mf{a}(\eta):\mf{Z}^{O}\to\mc{M}(W,\mc{Z}_{\R})$, such that $\mf{a}(\eta)(u)\coloneqq\mf{A}(u,\eta)$ 
for every $u\in\mf{Z}^{O}$.
\end{definition}
The above definition is well-set. Indeed Rmk.\ref{10161240} and Def.\ref{05101630}\eqref{09261410a} ensure that 
$\mf{L}(\phi_{\eta}(u),\psi_{\eta}(u))$ is well-defined, while it belongs to $\mc{C}^{p-1}(W,\mc{Z})$ namely to 
$\mc{C}^{p-1}(W,\mc{Z}_{\R})$ by Def.\ref{05101630} and Rmk.\ref{051016304} thus, 
$\mf{L}(\phi_{\eta}(u),\psi_{\eta}(u))\otimes\wedge_{i=1}^{n}d\xi_{i}\in\mr{Alt}_{loc}^{n}(W;\mc{Z}_{\R},\lambda)$ 
since \cite[(1.6)]{sil09261621} with 
$\mr{Alt}_{loc}^{n}(W;\mc{Z}_{\R},\lambda)=\mf{L}_{loc}^{1}(W,\mc{Z}_{\R},\lambda)\otimes_{\mc{A}(U)}\mr{Alt}^{n}(W,M)$
see \cite[Def.1.14]{sil09261621} for $\mf{L}_{loc}^{1}(W,\mc{Z}_{\R},\lambda)$,
then $\mf{m}_{\mf{L}(\phi_{\eta}(u),\psi_{\eta}(u))\otimes\wedge_{i=1}^{n}d\xi_{i}}\in\mc{M}(W,\mc{Z}_{\R})$ since \cite[Thm.3.57]{sil09261621} 
and the fact that $\mc{Z}_{\R}$ is quasi-complete, finally 
$\upzeta(\eta)(\mf{m}_{\mf{L}(\phi_{\eta}(u),\psi_{\eta}(u))\otimes\wedge_{i=1}^{n}d\xi_{i}})\in\mc{M}(W,\mc{Z}_{\R})$
since Lemma \ref{10161358}. 
\begin{convention}
In the present paper we convein to simplify the notation by letting $\int fT$ denote $\int f\otimes T$ for every 
$f\in\mf{L}_{c}^{1}(W,\mc{Z}_{\R},\lambda)$ (see \cite[Def.1.16]{sil09261621}) and $T\in\mr{Alt}^{n}(W,M)$.
\end{convention}
As a consequence of the above convention we have for every $f\in\mc{H}(W)$ that 
\begin{equation}
\label{08280414a}
\mf{A}(u,\eta)(f)=\int(\upzeta(\eta))^{\ast}(f)\mf{L}(\phi_{\eta}(u),\psi_{\eta}(u))\,\wedge_{i=1}^{n}d\xi_{i}.
\end{equation}  
\begin{lemma}
\label{06121742}
Let $u\in\mf{Z}^{O}$ and $\eta\in A$ thus, for every $f\in\mc{H}(W)$ we have 
\begin{equation*}
\mc{A}(u)(\eta)(f)  
=
\int
(\xi^{-1})^{\ast}(f)
\mr{L}\left(\phi_{\eta}^{\xi}(u)\circ\chi(\eta),\psi_{\eta}^{\xi}(u)\circ\chi(\eta)\right)
J_{\eta}d\lambda_{U}.
\end{equation*}
\end{lemma}
\begin{proof}
For every $\phi\in(\mc{Z}_{\R})^{\prime}$ we have 
\begin{equation*}
\begin{aligned}
\phi(\mc{A}(u)(\eta)(f))
&=
\int
(\upzeta(\eta))^{\ast}(f)\phi_{\ast}(\mf{L}(\phi_{\eta}(u),\psi_{\eta}(u)))\,\wedge_{i=1}^{n}d\xi_{i}
\\
&=
\int  
((\omega_{\natural}^{\eta})^{-1})^{\ast}
\left(
\upzeta(\eta))^{\ast}(f)
\phi_{\ast}(\mf{L}(\phi_{\eta}(u),\psi_{\eta}(u)))\,\wedge_{i=1}^{n}d\xi_{i}\right)
\\
&=
\int  
(\xi^{-1})^{\ast}(f)
\bigl(\phi_{\ast}(\mf{L}(\phi_{\eta}(u),\psi_{\eta}(u)))\circ(\omega_{\natural}^{\eta})^{-1}\bigr)J_{\eta}\,\wedge_{i=1}^{n}
d(\Pr_{i}\circ\imath_{U}^{\R^{n}})
\\
&=
\int
(\xi^{-1})^{\ast}(f)
\bigl(\phi_{\ast}(\mf{L}(\phi_{\eta}(u),\psi_{\eta}(u)))\circ\xi^{-1}\circ\xi\circ(\omega_{\natural}^{\eta})^{-1}\bigr)J_{\eta}
\,\wedge_{i=1}^{n}d(\Pr_{i}\circ\imath_{U}^{\R^{n}})
\\
&=
\phi\left(\int
(\xi^{-1})^{\ast}(f)
\bigl(\mf{L}(\phi_{\eta}(u),\psi_{\eta}(u))\circ\xi^{-1}\circ\xi\circ(\omega_{\natural}^{\eta})^{-1}\bigr)J_{\eta}
\,\wedge_{i=1}^{n}d(\Pr_{i}\circ\imath_{U}^{\R^{n}})\right);
\end{aligned}
\end{equation*}
where $\phi_{\ast}$ is defined in \cite[Def.2.39]{sil09261621}.
Here the 1th equality follows by \eqref{08280414a} and since \cite[Prp.3.55]{sil09261621}, with $\phi_{\times}$ defined in 
\cite[Def.2.41]{sil09261621}, the 2th one follows since \cite[Prp.16.6(d)]{lee} and by $\mr{sgn}(J_{\eta})=1$, while the 3th one 
by \cite[Prp. 14.20]{lee} both results generalized: the first to compactly supported continuous forms, the second one to 
continuous forms. The 5th equality follows since \cite[Prp.3.55]{sil09261621}. 
Then, the statement follows since $(\mc{Z}_{\R})^{\prime}$ separates the points of $\mc{Z}_{\R}$, being $\mc{Z}_{\R}$ of Hausdorff.
\end{proof}
\begin{proposition}
\label{08161352}
Let $\eta\in A$. If there exists a set $V\subset U$ such that $V$ is locally $\lambda_{U}$-negligible and 
$(\forall x\in\complement_{U}V)(\widehat{J}(\eta,x)=1)$, then 
\begin{equation*}
\begin{aligned}
\mr{L}\left(\phi_{\eta}^{\xi}(u),\partial(\phi_{\eta}^{\xi}(u))\right)\circ\chi(\eta)\circ\imath_{\complement_{U}V}^{U}
&=
\mr{L}\bigl(\phi_{\ze}^{\xi}(u),\partial(\phi_{\ze}^{\xi}(u))\bigr)\circ\imath_{\complement_{U}V}^{U}
\\
&\Rightarrow
\\
\mc{A}(u)(\eta)
&=
\mc{A}(u)(\ze).
\end{aligned}
\end{equation*}
\end{proposition}
\begin{proof}
Since Lemma \ref{06121742}, \eqref{08161354}, \eqref{08060435} and the fact that the integral of a locally 
$\lambda_{U}$-negligible map is zero.
\end{proof}
\begin{lemma}
\label{07181052}
Let $u\in\mf{Z}^{O}$ and $\eta\in A$ thus, for every $f\in\mc{H}(W)$ we have 
\begin{equation*}
\mc{A}(u)(\eta)(f)
=
\int
(\xi^{-1})^{\ast}(f)(x)G_{u}(x,\eta)d\lambda_{U}(x).
\end{equation*}
\end{lemma}
\begin{proof}
Since Lemma \ref{06121742}.
\end{proof}
\begin{definition}
\label{11022206}
If $E,F,G$ are topological linear spaces over $\K$, then we let 
\begin{equation*}
\mf{t}_{E,F,G}:\mc{L}_{b}(F,G)\times\mc{L}_{b}(E,F)\to\mc{L}_{b}(E,G), (b,a)\mapsto b\circ a;
\end{equation*}
and 
\begin{equation*}
\mf{s}_{E,F}:\mc{L}_{b}(E,F)\times E\to F, (S,e)\mapsto S(e).
\end{equation*}
\end{definition}
\begin{lemma}
\label{10221256}
If $N,P$ are normed spaces and $F$ is a locally convex space all over $\K$, then $\mf{s}_{N,F}$ and 
$\mf{t}_{N,P,F}$ are continuous and bilinear id est $\mf{s}_{N,F}\in\mc{L}(\mc{L}_{b}(N,F),N;F)$ and 
$\mf{t}_{N,P,F}\in\mc{L}(\mc{L}_{b}(P,F),\mc{L}_{b}(N,P);\mc{L}_{b}(N,F))$.
\end{lemma}
\begin{remark}
\label{10231117}
Let $h\in N$ and let $k_{h}:\mc{L}_{b}(N,F)\to N$, $f\mapsto h$, then 
$\mr{ev}_{h}^{\mc{L}_{b}(N,F)}=\mf{s}_{N,F}\circ(\mr{Id}_{\mc{L}_{b}(N,F)}\boxtimes k_{h})$ so the continuity of $\mr{ev}_{h}^{\mc{L}_{b}(N,F)}$ 
follows since that of $\mf{s}_{N,F}$ stated in Lemma \ref{10221256}.
\end{remark}
\begin{proof}[Proof of Lemma \ref{10221256}]
Let $\Gamma_{F}$ be a fundamental set of seminorms of $F$, $x\in N$ such that $x\neq\ze$, $f\in\mc{L}(N,F)$ and $p\in\Gamma_{F}$, 
then 
\begin{equation*}
p(\mf{s}_{N,F}(f,x))=\|x\|(p\circ f)(x\|x\|^{-1}))\leq\|x\|\cdot\overset{1}{p}(f),
\end{equation*}
so, $\mf{s}_{N,F}$ is continuous being bilinear. Next let $g\in\mc{L}_{b}(N,P)$ and $f\in\mc{L}_{b}(P,F)$ then, 
\begin{equation*}
\begin{aligned}
\overset{1}{p}(\mf{t}_{N,P,F}(f,g))
&=
\sup_{y\in B_{1}^{N}(\ze)}p(f(g(y)))
\\
&=
\sup_{y\in B_{1}^{N}(\ze)}
p(\mf{s}_{P,F}(f,g(y)))
\\
&\leq
\sup_{y\in B_{1}^{N}(\ze)}
\|g(y)\|\cdot\overset{1}{p}(f)
\\
&=
\|g\|\cdot
\overset{1}{p}(f),
\end{aligned}
\end{equation*}
so, $\mf{t}_{N,P,F}$ is continuous being bilinear.
\end{proof}
\begin{definition}
\label{10231103}
Let $E$ be a topological linear space over $\K$, $\mr{s}_{E}:\K\times E\to E$ be the product by scalars in $E$
and $P$ be a normed space over $\K$. Define $\mf{e}_{P,E}:E\times\mc{L}_{b}(P,\K)\to\mc{L}_{b}(P,E)$ such that 
\begin{equation*}
(\forall z\in E)(\forall b\in\mc{L}_{b}(P,\K))(\mf{e}_{P,E}(z,b)\coloneqq\mr{s}_{E}\circ(b\boxtimes c_{z})),
\end{equation*}
where $c_{z}:P\to E$, $h\mapsto z$.
\end{definition}
Since the locally convex spaces fixed in Assumption \ref{12152203} are over different fields, we convein to simplify the 
forthcoming notation by adopting the following
\begin{convention}
If $E$ is a topological linear space over $\R$, $F$ is a topological linear space over $\K_{a}\in\{\R,\K\}$ and 
$G$ is a topological linear space over $\K_{b}\in\{\R,\K\}$, then we let $\mf{t}_{E,F,G}$ denote $\mf{t}_{E,F_{\R},G_{\R}}$.
Furthermore let $\mf{g}_{i}$ denote $\mf{s}_{(Y_{i})_{\R},\mc{Z}_{\R}}$ for every $i\in[1,n]\cap\Z$ and let $\mf{h}$ denote 
$\mf{s}_{\mf{Y}_{\R},\mc{Z}_{\R}}$ in addition for any $h\in T$ we convein to let $\mr{ev}_{h}$ denote $\mr{ev}_{h}^{\mc{L}_{b}(T,\mc{Z}_{\R})}$.
If $F$ is a topological space over $\K$ and $P$ is a normed space over $\R$, then we let $\mf{e}_{P,F}$ denote $\mf{e}_{P,F_{\R}}$.
\end{convention}
\begin{lemma}
\label{10221931}
Let $P$ be a normed space and $E$ be a linear topological space both over $\K$.
Then, $\mf{e}_{P,E}$ is a well defined bilinear map. Furthermore if $E$ is locally convex, then $\mf{e}_{P,E}$ is continuous namely
$\mf{e}_{P,E}\in\mc{L}(E,\mc{L}_{b}(P,\K);\mc{L}_{b}(P,E))$.
\end{lemma}
\begin{proof}
$\mf{e}_{P,E}(z,b)$ is trivially linear moreover it is continuous since $\mr{s}_{E}$, $b$ and $c_{z}$ are continuous maps namely
$\mf{e}_{P,E}(z,b)\in\mc{L}(P,E)$ so, $\mf{e}_{P,E}$ is well-defined and trivially bilinear. Next assume that $E$ is locally convex
and let $q\in\Gamma_{E}$ where $\Gamma_{E}$ is a fundamental set of seminorms of $E$, then
\begin{equation*}
\begin{aligned}
\overset{1}{q}(\mf{e}_{P,E}(z,b))
&=
\sup_{h\in B_{1}^{P}(\ze)}
q(b(h)z)
\\
&=
\sup_{h\in B_{1}^{P}(\ze)}
|b(h)|\cdot q(z)
\\
&=
\|b\|\cdot q(z),
\end{aligned}
\end{equation*}
then $\mf{e}_{P,E}$ is continuous being bilinear.
\end{proof}
\begin{remark}
\label{10231439} 
Let $E$ be a topological linear space over $\K$, $P$ be a normed space over $\K$, $B$ a set, $f:B\to E$ and
$g:B\to\mc{L}_{b}(P,\K)$, then for every $h\in P$ we have 
\begin{equation*}
\mr{ev}_{h}^{\mc{L}_{b}(P,E)}\circ\mf{e}_{P,E}\circ(f\boxtimes g)=\mr{s}_{E}\circ((\mr{ev}_{h}^{\mc{L}_{b}(P,\K)}\circ g)\boxtimes f).
\end{equation*}
\end{remark}
\begin{convention} 
Let $E$ be a topological linear space over $\K$. For every set $B$, and every map $r:B\to\K$ and $t:B\to E$ 
we let $r\cdot t$ and $t\cdot r$ denote $\mr{s}_{E}\circ(r\boxtimes t)$.
\end{convention}
\begin{lemma}
\label{07220225}
$\beta$ is differentiable and $\beta^{1}(\eta,x)(h,t)=h$ for every $(\eta,x)\in A\times U$ and $(h,t)\in T\times\R^{n}$.
\end{lemma}
\begin{proof}
Trivially $D_{1}(\beta)(\eta,x)=\mr{Id}_{T}$ and $D_{2}(\beta)(\eta,x)=\ze$ so the statement follows since Cor.\ref{02211942}.
\end{proof}
\begin{lemma}
\label{07220226}
Let $u\in\mf{Z}^{O}$, then $F_{u}\in\mc{C}^{p-1}(A\times U,\mc{Z})$,
in particular $D_{1}(F_{u})\in\mc{C}^{p-2}(A\times U,\mc{L}_{b}(T,\mc{Z}_{\R}))$ and it is such that 
\begin{multline*}
D_{1}(F_{u})=
\mf{t}_{T,Y_{0},\mc{Z}}
\circ\left\{
\bigl[
\frac{\partial\mh{L}}{\partial Y_{0}}
\circ
(\Upphi(u)
\boxtimes
\Uppsi(u))
\bigr]
\boxtimes
D_{1}(\Upphi(u))
\right\}
\circ(\beta\boxtimes\widehat{\chi})
+
\\
\mf{t}_{T,\mf{Y},\mc{Z}}
\circ\left\{
\bigl[
\frac{\partial\mh{L}}{\partial\mf{Y}}
\circ
(\Upphi(u)
\boxtimes
\Uppsi(u))
\bigr]
\boxtimes
D_{1}(\Uppsi(u))
\right\}
\circ(\beta\boxtimes\widehat{\chi})
+
\\
\sum_{i=1}^{n}
\mf{e}_{T,\mc{Z}}
\circ
\left\{
\bigl[
\frac{\partial}{\partial e_{i}^{2}[n]}
\bigl(
\mh{L}
\circ
(\Upphi(u)
\boxtimes
\Uppsi(u))
\bigr)
\circ(\beta\boxtimes\widehat{\chi})
\bigr]
\boxtimes
D_{1}(\widehat{\chi}_{i})
\right\}.
\end{multline*}
\end{lemma}
\begin{remark}
\label{07220226b}
Let $u\in\mf{Z}^{O}$ thus, for every $h\in T$ we have by Rmk.\ref{10231439} 
\begin{multline*}
\mr{ev}_{h}\circ D_{1}(F_{u})=
\mf{g}_{0}
\circ\left\{
\bigl[
\frac{\partial\mh{L}}{\partial Y_{0}}
\circ
(\Upphi(u)
\boxtimes
\Uppsi(u))
\bigr]
\boxtimes
\bigl[
\mr{ev}_{h}^{\mc{L}_{b}(T,(Y_{0})_{\R})}\circ(D_{1}(\Upphi(u)))
\bigr]
\right\}
\circ(\beta\boxtimes\widehat{\chi})
+
\\
\mf{h}
\circ\left\{
\bigl[
\frac{\partial\mh{L}}{\partial\mf{Y}}
\circ
(\Upphi(u)
\boxtimes
\Uppsi(u))
\bigr]
\boxtimes
\bigl[
\mr{ev}_{h}^{\mc{L}_{b}(T,\mf{Y}_{\R})}\circ(D_{1}(\Uppsi(u)))
\bigr]
\right\}
\circ(\beta\boxtimes\widehat{\chi})
+
\\
\sum_{i=1}^{n}
\bigl[
\frac{\partial}{\partial e_{i}^{2}[n]}
\bigl(
\mh{L}
\circ
(\Upphi(u)
\boxtimes
\Uppsi(u))
\bigr)
\circ(\beta\boxtimes\widehat{\chi})
\bigr]
\cdot
(\mr{ev}_{h}^{\mc{L}_{b}(T,\R)}\circ D_{1}(\widehat{\chi}_{i}));
\end{multline*}
equivalently for every $(\eta,x)\in A\times U$ we have by \eqref{08030349}, since Lemma \ref{08080256} and by 
\begin{equation*}
\beta\boxtimes\widehat{\chi}
=
(\mf{o}_{\eta}\boxtimes\mr{Id}_{U})\circ\widehat{\chi},
\end{equation*}
that 
\begin{multline*}
D_{1}(F_{u})(\eta,x)(h)=
\frac{\partial\mh{L}}{\partial Y_{0}}
\bigl(\phi_{\eta}^{\xi}(u)(\widehat{\chi}(\eta,x)),\psi_{\eta}^{\xi}(u)(\widehat{\chi}(\eta,x))\bigr)
\bigl(D_{1}(\Upphi(u))(\eta,\widehat{\chi}(\eta,x))(h)\bigr)
+
\\
\frac{\partial\mh{L}}{\partial\mf{Y}}
\bigl(\phi_{\eta}^{\xi}(u)(\widehat{\chi}(\eta,x)),\psi_{\eta}^{\xi}(u)(\widehat{\chi}(\eta,x))\bigr)
\bigl(D_{1}(\Uppsi(u))(\eta,\widehat{\chi}(\eta,x))(h)\bigr)
+
\\
\sum_{i=1}^{n}\frac{\partial\mr{L}(\phi_{\eta}^{\xi}(u),\psi_{\eta}^{\xi}(u))}{\partial e_{i}[n]}
(\widehat{\chi}(\eta,x))(D_{1}(\widehat{\chi}_{i})(\eta,x)(h)).
\end{multline*}
\end{remark}
\begin{proof}[Proof of Lemma \ref{07220226}]
$F_{u}\in\mc{C}^{p-1}(A\times U,\mc{Z})$ follows by Def.\ref{05101630} and since Thm.\ref{81210}, then
$D_{1}(F_{u})\in\mc{C}^{p-2}(A\times U,\mc{L}_{b}(T,\mc{Z}_{\R}))$ follows since Cor.\ref{04221226}.
Next let $(\eta,x)\in A\times U$, $(h,t)\in T\times\R^{n}$, $u\in\mf{Z}^{O}$ and in the present proof let 
\begin{equation*}
\begin{aligned}
\gamma
&\coloneqq
\bigl(((\Upphi(u)\boxtimes\Uppsi(u))\circ(\beta\boxtimes\widehat{\chi})\bigr)(\eta,x)
\\
\delta
&\coloneqq
(\beta\boxtimes\widehat{\chi})(\eta,x)
\\
\theta
&\coloneqq
(\beta^{(1)}(\eta,x)(h,t),
(\widehat{\chi})^{(1)}(\eta,x)(h,t))
\\
\rho
&\coloneqq
(\widehat{\chi})^{(1)}(\eta,x)(h,\ze).
\end{aligned}
\end{equation*}
Thus, by Thm.\ref{821} and Prp.\ref{815} we have that
\begin{equation}
\label{08030309a}
(F_{u})^{(1)}(\eta,x)(h,t)=\mh{L}^{(1)}(\gamma)
\bigl((\Upphi(u))^{(1)}(\delta)(\theta),(\Uppsi(u))^{(1)}(\delta)(\theta)\bigr).
\end{equation}
Next by Cor.\ref{02211942} we have 
\begin{equation}
\label{08030309b}
D_{1}(F_{u})(\eta,x)(h)=(F_{u})^{(1)}(\eta,x)(h,\ze),
\end{equation}
therefore we obtain
\begin{equation}
\label{07210844}
\begin{aligned}
D_{1}(F_{u})(\eta,x)(h)
&=
\mh{L}^{(1)}(\gamma)\bigl((\Upphi(u))^{(1)}(\delta)(h,\rho),(\Uppsi(u))^{(1)}(\delta)(h,\rho)\bigr)
\\
&=
\frac{\partial\mh{L}}{\partial Y_{0}}(\gamma)
\bigl((\Upphi(u))^{(1)}(\delta)(h,\rho)\bigr)
+
\frac{\partial\mh{L}}{\partial\mf{Y}}(\gamma)
\bigl((\Uppsi(u))^{(1)}(\delta)(h,\rho)\bigr)
\\
&=
\frac{\partial\mh{L}}{\partial Y_{0}}(\gamma)
\bigl((D_{1}\Upphi(u))(\delta)(h)+(D_{2}\Upphi(u))(\delta)(\rho)\bigr)
\\
&+
\frac{\partial\mh{L}}{\partial\mf{Y}}(\gamma)
\bigl((D_{1}\Uppsi(u))(\delta)(h)+(D_{2}\Uppsi(u))(\delta)(\rho)\bigr)
\\
&=
\frac{\partial\mh{L}}{\partial Y_{0}}(\gamma)
\bigl((D_{1}\Upphi(u))(\delta)(h)\bigr)
+
\frac{\partial\mh{L}}{\partial\mf{Y}}(\gamma)
\bigl((D_{1}\Uppsi(u))(\delta)(h)\bigr)
\\
&
+
\frac{\partial\mh{L}}{\partial Y_{0}}(\gamma)
\bigl((D_{2}\Upphi(u))(\delta)(\rho)\bigr)
+
\frac{\partial\mh{L}}{\partial\mf{Y}}(\gamma)
\bigl((D_{2}\Uppsi(u))(\delta)(\rho)\bigr).
\end{aligned}
\end{equation}
Here the 1th equality follows by \eqref{08030309a}, \eqref{08030309b} and Lemma \ref{07220225}, the 2th one by Cor.\ref{02211942}
applied to $\mh{L}$, the 3th one since Cor.\ref{02211942} applied to $\Upphi(x)$ and $\Uppsi(x)$. Next we have 
\begin{equation}
\label{08030349}
\mr{L}(\phi_{\eta}^{\xi}(u),\psi_{\eta}^{\xi}(u))=
\mh{L}\circ(\Upphi(u)\boxtimes\Uppsi(u))\circ(\mf{o}_{\eta}\boxtimes\mr{Id}_{U}),
\end{equation}
therefore we obtain for every $y\in U$ and $r\in\R^{n}$ 
\begin{equation}
\label{07210833}
\begin{aligned}
\mr{L}(\phi_{\eta}^{\xi}(u),\psi_{\eta}^{\xi}(u))^{(1)}(y)(r)
=
\bigl(\mh{L}\circ(\Upphi(u)\boxtimes\Uppsi(u))\bigr)^{(1)}(\eta,y)(\ze,r);
\end{aligned}
\end{equation}
where the equality follows by \eqref{08030349}, Thm.\ref{821}, Prp.\ref{815} and the fact that 
$(\imath_{A}^{T}\circ\mf{o}_{\eta})^{(1)}=\ze$, while $\mr{Id}_{U\to\R^{n}}^{(1)}$ is the constant map equal to $\mr{Id}_{\R^{n}}$.
Therefore for every $y\in U$ and $r\in\R^{n}$ we have 
\begin{equation*}
\begin{aligned}
\mr{L}(\phi_{\eta}^{\xi}(u),\psi_{\eta}^{\xi}(u))^{(1)}(y)(r)
&=
\mh{L}^{(1)}\left(\phi_{\eta}^{\xi}(u)(y),\psi_{\eta}^{\xi}(u)(y)\right)
\bigl((\Upphi(u))^{(1)}(\eta,y)(\ze,r),(\Uppsi(u))^{(1)}(\eta,y)(\ze,r)\bigr)
\\
&=
\frac{\partial\mh{L}}{\partial Y_{0}}
\left(\phi_{\eta}^{\xi}(u)(y),\psi_{\eta}^{\xi}(u)(y)\right)
\bigl((\Upphi(u))^{1}(\eta,y)(\ze,r)\bigr)
\\
&+
\frac{\partial\mh{L}}{\partial\mf{Y}}
\left(\phi_{\eta}^{\xi}(u)(y),\psi_{\eta}^{\xi}(u)(y)\right)
\bigl((\Uppsi(u))^{1}(\eta,y)(\ze,r)\bigr)
\\
&=
\frac{\partial\mh{L}}{\partial Y_{0}}
\left(\phi_{\eta}^{\xi}(u)(y),\psi_{\eta}^{\xi}(u)(y)\right)
\bigl(D_{2}(\Upphi(u))(\eta,y)(r)\bigr)
\\
&+
\frac{\partial\mh{L}}{\partial\mf{Y}}
\left(\phi_{\eta}^{\xi}(u)(y),\psi_{\eta}^{\xi}(u)(y)\right)
\bigl(D_{2}(\Uppsi(u))(\eta,y)(r)\bigr),
\end{aligned}
\end{equation*}
where the 1th equality follows since \eqref{07210833}, Thm.\ref{821} and Prp.\ref{815}; the 2th one since Cor.\ref{02211942}
applied to $\mh{L}$; the 3th one since Cor.\ref{02211942} applied to $\Upphi(u)$ and $\Uppsi(u)$. Therefore 
\begin{equation}
\label{07210845}
\frac{\partial\mh{L}}{\partial Y_{0}}(\gamma)
\bigl((D_{2}\Upphi(u))(\delta)(\rho)\bigr)
+
\frac{\partial\mh{L}}{\partial\mf{Y}}(\gamma)
\bigl((D_{2}\Uppsi(u))(\delta)(\rho)\bigr)
=
\mr{L}(\phi_{\eta}^{\xi}(u),\psi_{\eta}^{\xi}(u))^{(1)}(\widehat{\chi}(\eta,x))(\rho).
\end{equation}
Now \eqref{07210844} and \eqref{07210845} yield 
\begin{equation}
\label{07210849}
\begin{aligned}
D_{1}(F_{u})(\eta,x)(h)
&=
\frac{\partial\mh{L}}{\partial Y_{0}}(\gamma)
\bigl((D_{1}\Upphi(u))(\delta)(h)\bigr)
+
\frac{\partial\mh{L}}{\partial\mf{Y}}(\gamma)
\bigl((D_{1}\Uppsi(u))(\delta)(h)\bigr)
\\
&+
\mr{L}(\phi_{\eta}^{\xi}(u),\psi_{\eta}^{\xi}(u))^{(1)}(\widehat{\chi}(\eta,x))(\rho).
\end{aligned}
\end{equation}
The statement then follows by Cor.\ref{02211942} applied to $L(\phi_{\eta}^{\xi}(u),\psi_{\eta}^{\xi}(u))$, Prp.\ref{815} applied 
to $\widehat{\chi}$ and Cor.\ref{02211942} applied to $\widehat{\chi}_{i}$ for any $i\in[1,n]\cap\Z$.
\end{proof}
\begin{definition}
\label{10221817}
Let $u\in\mf{Z}^{O}$, then for any $k\in[1,4]\cap\Z$ define $\Omega_{k}[u]\in\mc{C}^{p-2}(A\times U,\mc{L}_{b}(T,\mc{Z}_{\R}))$ such 
that 
\begin{equation*}
\begin{aligned}
\Omega_{1}[u]
&\coloneqq
\mf{t}_{T,Y_{0},\mc{Z}}
\circ\bigl\{
\bigl[
\frac{\partial\mh{L}}{\partial Y_{0}}
\circ
(\Upphi(u)
\boxtimes
\Uppsi(u))
\bigr]
\boxtimes
D_{1}(\Upphi(u))
\bigr\};
\\
\Omega_{2}[u]
&\coloneqq
\sum_{i=1}^{n}
\mf{t}_{T,Y_{i},\mc{Z}}
\circ\bigl\{
\bigl[
\frac{\partial\mh{L}_{\mf{e}}}{\partial Y_{i}}
\circ
(\mf{d}_{O})_{\natural}
\circ
(\Upphi(u)
\boxtimes
\Uppsi(u))
\bigr]
\boxtimes
\bigl[\frac{\partial}{\partial e_{i}^{2}[n]}(D_{1}(\Upphi(u)))\bigr]
\bigr\};
\\
\Omega_{3}[u]
&\coloneqq
\sum_{i=1}^{n}
\mf{e}_{T,\mc{Z}}
\circ
\bigl\{
\bigl[
\frac{\partial}{\partial e_{i}^{2}[n]}
\bigl(
\mh{L}
\circ
(\Upphi(u)
\boxtimes
\Uppsi(u))
\bigr)
\bigr]
\boxtimes
[D_{1}(\widehat{\chi}_{i})\circ(\beta\boxtimes\widehat{\sigma})]
\bigr\};
\\
\Omega_{4}[u]
&\coloneqq
\mf{e}_{T,\mc{Z}}
\circ
\bigl\{F_{u}\boxtimes(D_{1}\widehat{J})\big\}.
\end{aligned}
\end{equation*}
Furthermore define $\mh{B}[u]:A\to\mc{M}(W,\mc{L}_{b}(T,\mc{Z}_{\R}))$ and $\mh{E}[u]:A\to\mc{M}(W,\mc{L}_{b}(T,\mc{Z}_{\R}))$ 
such that for every $\eta\in A$ and $f\in\mc{H}(W)$ we have 
\begin{equation*}
\begin{aligned}
\mh{B}[u](\eta)(f)
&\coloneqq
\int((\chi(\eta)\circ\xi)^{-1})^{\ast}(f)\cdot
\left[\sum_{i=1}^{3}\Omega_{i}[u]\circ(\mf{o}_{\eta}\boxtimes\mr{Id}_{U})\right]\,d\lambda_{U};
\\
\mh{E}[u](\eta)(f)
&\coloneqq
\int(\xi^{-1})^{\ast}(f)\cdot\left[\Omega_{4}[u]\circ(\mf{o}_{\eta}\boxtimes\mr{Id}_{U})\right]\,d\lambda_{U}.
\end{aligned}
\end{equation*}
\end{definition}
\begin{proposition}
\label{10221814}
Def.\ref{10221817} is well-set more specifically for any $u\in\mf{Z}^{O}$, $\eta\in A$ and $k\in[1,4]\cap\Z$, we have 
\begin{enumerate}
\item
$\Omega_{k}[u]\in\mc{C}^{p-2}(A\times U,\mc{L}_{b}(T,\mc{Z}_{\R}))$;
\label{10221814a}
\item
$\mh{B}[u](\eta)\in\mc{M}(W,\mc{L}_{b}(T,\mc{Z}_{\R}))$ and $\mh{E}[u](\eta)\in\mc{M}(W,\mc{L}_{b}(T,\mc{Z}_{\R}))$.
\label{10221814b}
\end{enumerate}
\end{proposition}
\begin{proof}
Item \eqref{10221814a} follows by Def.\ref{05101630}, Rmk.\ref{051016304} and since Lemma \ref{10221256}, Lemma \ref{10221931}, 
Thm.\ref{81210}, Cor.\ref{04221226}, Prp.\ref{08091052}, Cor.\ref{8126} and Prp.\ref{01101646}. Next since Lemma \ref{12311713} 
we deduce that $\mc{L}_{b}(T,\mc{Z}_{\R})$ is a Hausdorff locally convex quasi-complete space therefore item \eqref{10221814b} 
follows by the proven item \eqref{10221814a}, since Cor.\ref{06162059} and Cor.\ref{08111037}.
\end{proof}
\begin{definition}
Let $u\in\mf{Z}^{O}$ and $h\in T$, then for any $k\in[1,4]\cap\Z$ define $\Omega_{k}[u,h]\in\mc{C}^{p-2}(A\times U,\mc{Z})$ such 
that 
\begin{equation*}
\Omega_{k}[u,h]
\coloneqq
\mr{ev}_{h}^{\mc{L}_{b}(T,\mc{Z}_{\R})}\circ\Omega_{k}[u].
\end{equation*}
Furthermore for any $\eta\in A$ and $f\in\mc{H}(W)$, define $\mh{B}[u,\eta,f,h]\in\mc{Z}$ and $\mh{E}[u,\eta,f,h]\in\mc{Z}$ 
such that 
\begin{equation*}
\begin{aligned}
\mh{B}[u,\eta,f,h]
&\coloneqq
\mr{ev}_{h}^{\mc{L}_{b}(T,\mc{Z}_{\R})}(\mh{B}[u](\eta)(f)),
\\
\mh{E}[u,\eta,f,h]
&\coloneqq
\mr{ev}_{h}^{\mc{L}_{b}(T,\mc{Z}_{\R})}(\mh{E}(\eta)(f)).
\end{aligned}
\end{equation*}
\end{definition}
\begin{remark}
Let $u\in\mf{Z}^{O}$, $\eta\in A$, $f\in\mc{H}(W)$ and $h\in T$. Since $\mr{ev}_{h}\in\mc{L}(\mc{L}_{b}(T,\mc{Z}_{\R}),\mc{Z}_{\R})$
by Rmk.\ref{10231117} we obtain 
\begin{equation*}
\mh{B}[u,\eta,f,h]=
\int((\chi(\eta)\circ\xi)^{-1})^{\ast}(f)\cdot\left[\sum_{i=1}^{3}\Omega_{i}[u,h]\circ(\mf{o}_{\eta}\boxtimes\mr{Id}_{U})\right]
\,d\lambda_{U},
\end{equation*}
and
\begin{equation*}
\mh{E}[u,\eta,f,h]
=
\int(\xi^{-1})^{\ast}(f)\cdot\left[\Omega_{4}[u,h]\circ(\mf{o}_{\eta}\boxtimes\mr{Id}_{U})\right]\,d\lambda_{U}.
\end{equation*}
Furthermore by Rmk.\ref{10231439} we obtain
\begin{equation*}
\begin{aligned}
\Omega_{1}[u,h]
&=
\mf{g}_{0}
\circ\bigl\{
\bigl(
\frac{\partial\mh{L}}{\partial Y_{0}}
\circ
(\Upphi(u)
\boxtimes
\Uppsi(u))
\bigr)
\boxtimes
\bigl(
\mr{ev}_{h}^{\mc{L}_{b}(T,(Y_{0})_{\R})}\circ(D_{1}(\Upphi(u)))
\bigr)
\bigr\}
\\
\Omega_{2}[u,h]
&=
\sum_{i=1}^{n}
\mf{g}_{i}\circ\bigl\{
\bigl(
\frac{\partial\mh{L}_{\mf{e}}}{\partial Y_{i}}
\circ
(\mf{d}_{O})_{\natural}
\circ
(\Upphi(u)
\boxtimes
\Uppsi(u))
\bigr)
\boxtimes
\bigl(
\mr{ev}_{h}^{\mc{L}_{b}(T,(Y_{i})_{\R})}\circ\bigl(\frac{\partial}{\partial e_{i}^{2}[n]}(D_{1}(\Upphi(u)))\bigr)
\bigr)
\bigr\};
\\
\Omega_{3}[u,h]
&=
\sum_{i=1}^{n}
\bigl(
\frac{\partial}{\partial e_{i}^{2}[n]}
\bigl(
\mh{L}
\circ
(\Upphi(u)
\boxtimes
\Uppsi(u))
\bigr)
\bigr)
\cdot
(\mr{ev}_{h}^{\mc{L}_{b}(T,\R)}\circ D_{1}(\widehat{\chi}_{i})
\circ(\beta\boxtimes\widehat{\sigma})),
\\
\Omega_{4}[u,h]
&=
F_{u}\cdot(\mr{ev}_{h}^{\mc{L}_{b}(T,\R)}\circ(D_{1}\widehat{J})).
\end{aligned}
\end{equation*}
\end{remark}
\begin{remark}
Let $u\in\mf{Z}^{O}$, $\eta\in A$, $h\in T$, $f\in\mc{H}(W)$, then by \eqref{08030349} and since Lemma \ref{08080256} we have
\begin{multline*}
\mh{B}[u,\eta,f,h]
=
\\
\int((\chi(\eta)\circ\xi)^{-1})^{\ast}(f)(x)
\bigl[
\frac{\partial\mh{L}}{\partial Y_{0}}(\phi_{\eta}^{\xi}(u)(x),\psi_{\eta}^{\xi}(u)(x))
\bigl((D_{1}\Upphi(u))(\eta,x)(h)\bigr)
\\
+
\sum_{i=1}^{n}\bigl(\frac{\partial\mh{L}_{\mf{e}}}{\partial Y_{i}}\circ(\mf{d}_{O})_{\natural}\bigr)
(\phi_{\eta}^{\xi}(u)(x),\psi_{\eta}^{\xi}(u)(x))
\bigl(\frac{\partial}{\partial e_{i}^{2}[n]}\bigl(D_{1}\Upphi(u)\bigr)(\eta,x)(h)\bigr)
\bigr]\,d\lambda_{U}(x)
\\
+
\int((\chi(\eta)\circ\xi)^{-1})^{\ast}(f)(x)
\sum_{i=1}^{n}\frac{\partial}{\partial e_{i}[n]}\left(\mr{L}(\phi_{\eta}^{\xi}(u),\psi_{\eta}^{\xi}(u))\right)(x)
\cdot
(D_{1}\widehat{\chi}_{i})\left(\eta,(\chi(\eta))^{-1}(x)\right)(h)
\,d\lambda_{U}(x).
\end{multline*}
\end{remark}
\begin{lemma}
\label{07221116}
Let $u\in\mf{Z}^{O}$, $\eta\in A$, $f\in\mc{H}(W)$, then
\begin{equation}
\label{07221116a}
\int((\xi^{-1})^{\ast}f)\cdot\bigl(\{(D_{1}F_{u})\cdot\widehat{J}\}\circ(\mf{o}_{\eta}\boxtimes\mr{Id}_{U})\bigr)\,d\lambda_{U}
\in\mc{L}_{b}(T,\mc{Z}_{\R}).
\end{equation}
Furthermore \eqref{07221116b}$\Leftrightarrow$\eqref{07221116c} and each of them holds true, where
\begin{equation}
\label{07221116b}
\int((\xi^{-1})^{\ast}f)\cdot\bigl(\{(D_{1}F_{u})\cdot\widehat{J}\}\circ(\mf{o}_{\eta}\boxtimes\mr{Id}_{U})\bigr)\,d\lambda_{U}
=\mh{B}[u,\eta,f],
\end{equation}
and 
\begin{equation}
\label{07221116c}
(\forall h\in T)\bigl(\int((\xi^{-1})^{\ast}f)\cdot\bigl(
\{[\mr{ev}_{h}^{\mc{L}_{b}(T,\mc{Z}_{\R})}\circ(D_{1}F_{u})]\cdot\widehat{J}\}\circ(\mf{o}_{\eta}\boxtimes\mr{Id}_{U})\bigr)\,
d\lambda_{U}=\mh{B}[u,\eta,f,h]\bigr).
\end{equation}
\end{lemma}
\begin{proof}
Since Lemma \ref{07220226}, Prp.\ref{01101646} and the fact that $\widehat{J}\in\mc{C}^{p-1}(A\times U,\R)$ we have 
\begin{equation*}
(D_{1}F_{u})\cdot\widehat{J}\in\mc{C}^{p-2}(A\times U,\mc{L}_{b}(T,\mc{Z}_{\R})).
\end{equation*}
Next $\mc{L}_{b}(T,\mc{Z}_{\R})$ is a Hausdorff locally convex quasi-complete space since Lemma \ref{12311713} therefore, 
\eqref{07221116a} follows since \cite[III.38 Cor.2]{IntBourb}.
Next $\mr{ev}_{h}\in\mc{L}(\mc{L}_{b}(T,\mc{Z}_{\R}),\mc{Z}_{\R})$ since Rmk.\ref{10231117} so, 
\eqref{07221116b}$\Rightarrow$\eqref{07221116c}; while \eqref{07221116a}, Prp.\ref{10221814}\eqref{10221814b} and 
$\mr{ev}_{h}\in\mc{L}(\mc{L}_{b}(T,\mc{Z}_{\R}),\mc{Z}_{\R})$ imply \eqref{07221116c}$\Rightarrow$\eqref{07221116b} so, we have 
proven that \eqref{07221116b}$\Leftrightarrow$\eqref{07221116c}. In order to conclude the proof it sufficies to prove that 
\eqref{07221116c} holds true. To this end note that for any $e\in A\times U$, $h\in T$ and $i\in[1,n]\cap\Z$ we have 
\begin{equation}
\label{08031059b}
\begin{aligned}
\Pr_{i}^{\mc{B}}
\bigl(D_{1}(\Uppsi(u))(e)(h)\bigr)
&=
\Pr_{i}^{\mc{B}}
\bigl((\Uppsi(u))^{(1)}(e)(h,\ze)\bigr)
\\
&=
(\Uppsi_{i}(u))^{(1)}(e)(h,\ze)
\\
&=
D_{1}(\Uppsi_{i}(u))(e)(h)
\\
&=
D_{1}
\bigl(\frac{\partial(\Upphi(u))}{\partial e_{i}^{2}[n]}\bigr)(e)(h)
\\
&=
\frac{\partial(D_{1}\Upphi(u))}{\partial e_{i}^{2}[n]}(e)(h),
\end{aligned}
\end{equation}
where the 1th and 3th equality follow by Cor.\ref{02211942}, the 2th one by Prp.\ref{815}, the 4th one by Rmk.\ref{06132050}
and \eqref{08161354}, the 5th one by Cor.\ref{06300307}.
Then \eqref{07221116c} follows since the definition of weak integral, since Cor.\ref{09261537} applied for $F=\R$ and to the 
diffeomorphism $\chi(\eta)$, since Lemma \ref{07220226}, \eqref{10181658b}, and \eqref{08031059b}.
\end{proof}
In the following remark we employ to the special case of the category of topological linear spaces over $\K$ and continuous 
linear maps, the operator $(\cdot)_{\star}$ defined in the general case in \cite[p.12]{sil}. 
\begin{remark}
\label{09152034}
Let $N$ be a normed space over $\K$, $B$ be an open set of $N$, $C\subseteq B$ and $f:B\to\mc{M}(W,\mc{Z}_{\R})$ in 
addition let $r:C\to\R_{>0}$ be such that $x+B_{r(x)}^{N}(\ze)\subseteq B$ for every $x\in C$ which exists since $B$ is open in $N$.
Now by recalling that the topology on $\mc{M}(W,\mc{Z}_{\R})$ is the topology of simple convergence, we have that $f$ is 
differentiable on $C$ iff we prove that there exists a map $\mf{f}$ such that 
\begin{enumerate}
\item
$\mf{f}:C\to\mc{L}(N,\mc{M}(W,\mc{Z}_{\R}))$;
\label{09152034a}
\item
\begin{equation*}
(\forall x\in C)(\forall l\in\mc{H}(W))
\bigl(\lim_{h\to\ze,\,h\in\complement_{B_{r(x)}^{N}(\ze)}\{\ze\}}
\frac{f(x+h)(l)-f(x)(l)-\mf{f}(x)(h)(l)}{\|h\|_{N}}=\ze\bigr).
\end{equation*}
\label{09152034b}
\end{enumerate}
By Thm.\ref{821} we deduce that item \eqref{09152034b} is equivalent to the following one:
\begin{itemize}
\item
For every $l\in\mc{H}(W)$ the map $\mr{ev}_{l}^{\mc{M}(W,\mc{Z}_{\R})}\circ f:B\to\mc{Z}_{\R}$, $x\mapsto f(x)(l)$ is differentiable on 
$C$ with differential map equal to $(\mr{ev}_{l}^{\mc{M}(W,\mc{Z}_{\R})})_{\star}\circ\mf{f}:C\to\mc{L}(N,\mc{Z}_{\R})$, 
$x\mapsto(h\mapsto \mf{f}(x)(h)(l))$. 
\end{itemize}
Well-set since $\mr{ev}_{l}^{\mc{M}(W,\mc{Z}_{\R})}\in\mc{L}(\mc{M}(W,\mc{Z}_{\R}),\mc{Z}_{\R})$. 
If items (\ref{09152034a},\ref{09152034b}) are fulfilled, then $\mf{f}$ is the differential map of $f$.
Next given $S:N\to\mc{M}(W,\mc{Z}_{\R})$ linear we notice that 
\begin{equation}
\label{09152119}
S\in\mc{L}(N,\mc{M}(W,\mc{Z}_{\R}))
\Leftrightarrow
(\forall l\in\mc{H}(W))(\forall q\in\Gamma)(\exists\,c\in\R_{>0})(\forall h\in N)(q(S(h)(l))\leq c\|h\|_{N}).
\end{equation}
Finally Cor.\ref{06162059} might be employed to check if a map belongs to $\mc{M}(W,\mc{Z}_{\R})$.
\end{remark}
To this end let us state the following: 
\begin{proposition}
\label{09030733} 
For every $u\in\mf{Z}^{O}$ and $\eta\in A$ we have 
\begin{equation*}
\begin{aligned}
G_{u}&\in\mc{C}^{p-1}(A\times U,\mc{Z});
\\
(D_{1}G_{u})&\in\mc{C}^{p-2}(A\times U,\mc{L}_{b}(T,\mc{Z}_{\R})).
\end{aligned}
\end{equation*}
\end{proposition}
\begin{proof}
The 1th relation in the statement follows since Prp.\ref{01101646}, Lemma \ref{07220226} specifically 
$F_{u}\in\mc{C}^{p-1}(A\times U,\mc{Z})$ and the fact that $\widehat{J}\in\mc{C}^{p-1}(A\times U,\R)$. The 2th relation follows by 
the 1th one and since Cor.\ref{04221226}.
\end{proof}
Notice that if $x\in U$, $\eta\in A$ and $h\in T$, then since Cor.\ref{02211942} we have 
\begin{equation}
\label{07222205a}
(D_{1}G_{u})(\eta,x)(h)=(G_{u})^{(1)}(\eta,x)(h,\ze);
\end{equation}
while 
\begin{equation}
\label{07222205b}
\begin{aligned}
(G_{u})^{(1)}(\eta,x)(h,\ze)
&=
\\
(F_{u})^{(1)}(\eta,x)(h,\ze)\cdot\widehat{J}(\eta,x)
+
F_{u}(\eta,x)\cdot\widehat{J}^{(1)}(\eta,x)(h,\ze)
&=
\\
(D_{1}F_{u})(\eta,x)(h)\cdot\widehat{J}(\eta,x)
+
F_{u}(\eta,x)\cdot(D_{1}\widehat{J})(\eta,x)(h),
\end{aligned}
\end{equation}
where the 1th equality follows since Prp.\ref{03241344}, the 2th one by Cor.\ref{02211942}. 
Now we can state the following 
\begin{theorem}
[\textbf{Differential of $\mc{A}(u)$}]
\label{07222143}
Let $u\in\mf{Z}^{O}$, then $\mc{A}(u)$ is differentiable and for every $\eta\in A$, $h\in T$ and $f\in\mc{H}(W)$ we have 
\begin{equation*}
\mc{A}(u)^{(1)}(\eta)(h)(f)=\mh{B}[u,\eta,f,h]+\mh{E}[u,\eta,f,h].
\end{equation*}
\end{theorem}
\begin{proof}
Let $u\in\mf{Z}^{O}$, $\eta\in A$, $h\in T$ and $f\in\mc{H}(W)$. Since Rmk.\ref{09152034} $\mc{A}(u)$ is differentiable if we 
prove that 
\begin{enumerate}
\item
$(\forall\eta\in A)(\exists\,\mf{c}(u,\eta):T\to\mc{M}(W,\mc{Z}_{\R}))$ such that for every $f\in\mc{H}(W)$ the map 
\begin{equation*}
T\ni h\mapsto\mf{c}(u,\eta)(h)(f)\in\mc{Z},
\end{equation*}
is the differential at $\eta$ of the map 
\begin{equation*}
\mf{r}_{f}^{u}:A\ni a\mapsto\mc{A}(u)(a)(f)\in\mc{Z},
\end{equation*}
namely 
\begin{equation}
\label{09301538}
(\forall h\in T)(\mf{c}(u,\eta)(h)(f)=(\mf{r}_{f}^{u})^{(1)}(\eta)(h));
\end{equation}
\label{08040402a}
\item
$(\forall\eta\in A)(\mf{c}(u,\eta)\in\mc{L}(T,\mc{M}(W,\mc{Z}_{\R})))$;
\label{08040402b}
\end{enumerate}
If the above items (\ref{08040402a},\ref{08040402b}) are fulfilled, then $\mf{c}(u,\eta)$ is the differential of 
$\mc{A}(u)$ at $\eta$ for every $\eta\in A$. Next since Lemma \ref{06121742} we have 
\begin{equation}
\label{09301432}
\mf{r}_{f}^{u}(\eta)=\int(\xi^{-1})^{\ast}(f)\cdot\bigl(G_{u}\circ(\mf{o}_{\eta}\boxtimes\mr{Id}_{U})\bigr)d\lambda_{U},
\end{equation}
next $(\xi^{-1})^{\ast}(f)\in\mc{H}(U)$ since Cor.\ref{08111037}, then since Prp.\ref{09030733} and Cor.\ref{09292048} we obtain
\begin{equation}
\label{09301433}
\begin{aligned}
(\xi^{-1})^{\ast}(f)\cdot\bigl(G_{u}&\circ(\mf{o}_{\eta}\boxtimes\mr{Id}_{U})\bigr)\in\mc{H}\bigl(U,\xi(\mr{supp}(f));\mc{Z}\bigr),
\\
(\xi^{-1})^{\ast}(f)\cdot\bigl((D_{1}G_{u})&\circ(\mf{o}_{\eta}\boxtimes\mr{Id}_{U})\bigr)
\in\mc{H}\bigl(U,\xi(\mr{supp}(f));\mc{L}_{b}(T,\mc{Z}_{\R})\bigr),
\end{aligned}
\end{equation}
therefore by \eqref{09301432}, \eqref{09301433} and Cor.\ref{01031245b} we deduce that 
\begin{equation}
\label{09301513}
(\mf{r}_{f}^{u})^{(1)}(\eta)=\int(\xi^{-1})^{\ast}(f)\cdot\bigl((D_{1}G_{u})\circ(\mf{o}_{\eta}\boxtimes\mr{Id}_{U})\bigr)\,d\lambda_{U},
\end{equation}
well-set since we recall that $\mc{L}_{b}(T,\mc{Z}_{\R})$ is a Hausdorff quasi-complete locally convex space since 
Lemma \ref{12311713}; then by Rmk.\ref{08190319} we have 
\begin{equation}
\label{07230420}
(\mf{r}_{f}^{u})^{(1)}(\eta)(h)=
\int(\xi^{-1})^{\ast}(f)\cdot\bigl(\mr{ev}_{h}\circ(D_{1}G_{u})\circ(\mf{o}_{\eta}\boxtimes\mr{Id}_{U})\bigr)\,d\lambda_{U}.
\end{equation}
Next assume the position 
\begin{equation}
\label{08040908}
\mf{c}(u,\eta)\coloneqq\mc{A}(u)^{(1)}(\eta),
\end{equation}
where $\mc{A}(u)^{(1)}(\eta):T\to\mc{F}(\mc{H}(W),\mc{Z})$, $h\mapsto(f\mapsto\mc{A}(u)^{(1)}(\eta)(h)(f))$
as given in the statement of the present theorem. 
Since \eqref{07230420}, \eqref{07222205a}, \eqref{07222205b} and Lemma \ref{07221116} we obtain \eqref{09301538}, 
then the claimed item \eqref{08040402a} follows if we prove that 
$\bigl(f\mapsto\mf{c}(u,\eta)(h)(f)\bigr)\in\mc{M}(W,\mc{Z}_{\R})$
but this follows by \eqref{09301538}; \eqref{07230420}; the second relation in Prp.\ref{09030733}; by Rmk.\ref{10231117}
applied to $\mr{ev}_{h}$; by Cor.\ref{06162059}. Item \eqref{08040402a} has been proven.
Next in order to conclude the proof we claim to prove item \eqref{08040402b} which by \eqref{09152119} is equivalent 
to the following one 
\begin{equation}
\label{07230414}
(\forall\eta\in A)(\forall f\in\mc{H}(W))(\forall q\in\Gamma)(\exists\,c\in\R_{>0})(\forall h\in T)
(q(\mc{A}(u)^{(1)}(\eta)(h)(f))\leq c\|h\|).
\end{equation}
Since \eqref{07230420}, the already proven item \eqref{08040402a} and \eqref{08040908} we obtain by \cite[VI.5, Prp.6]{IntBourb}
that 
\begin{equation*}
\begin{aligned}
q(\mc{A}(u)^{1}(\eta)(h)(f))
&\leq
\int|(\xi^{-1})^{\ast}(f)|(x)q\bigl((D_{1}G_{u})(\eta,x)(h)\bigr)\,d\lambda_{U}(x)
\\
&\leq
\int|(\xi^{-1})^{\ast}(f)|(x)\|h\|\overset{1}{q}\bigl((D_{1}G_{u})(\eta,x)\bigr)\,d\lambda_{U}(x)
\\
&=
\int|(\xi^{-1})^{\ast}(f)|(x)\|h\|\overset{1}{q}\bigl((D_{1}G_{u})(\eta,x)\bigr)\,d\lambda_{\xi(\mr{supp}(f))}(x)
\\
&\leq
\|h\|\sup_{x\in\xi(\mr{supp}(f))}\overset{1}{q}\bigl((D_{1}G_{u})(\eta,x)\bigr)\int|(\xi^{-1})^{\ast}(f)|(y)\,d\lambda_{U}(y).
\end{aligned}
\end{equation*}
Now the locally convex topology of $\mc{L}_{b}(T,\mc{Z}_{\R})$ is generated by the set of 
seminorms $\widehat{\Gamma}\coloneqq\{\overset{1}{q}\,\vert\,q\in\Gamma\}$ then continuous maps on $\mc{L}_{b}(T,\mc{Z}_{\R})$, 
therefore since the second relation in Prp.\ref{09030733} and since continuous functions map compacts into compacts we obtain 
\begin{equation*}
(\forall\eta\in A)(\sup_{x\in\xi(\mr{supp}(f))}\overset{1}{q}\bigl((D_{1}G_{u})(\eta,x)\bigr)<\infty).
\end{equation*}
Claim \eqref{07230414} and the statement then follow by taking 
\begin{equation*}
c=\sup_{x\in\xi(\mr{supp}(f))}\overset{1}{q}\bigl((D_{1}G_{u})(\eta,x)\bigr)\int|(\xi^{-1})^{\ast}(f)|(y)\,d\lambda_{U}(y).
\end{equation*}
\end{proof}
\begin{definition}
[Euler-Lagrange equation]
\label{07230430}
Define $\mr{EL}_{W}(\mh{L})$ the set of those $u\in\mf{Z}^{O}$ such that 
\begin{equation*}
\frac{\partial\mh{L}}{\partial Y_{0}}\circ(u\boxtimes\partial^{\xi}u)\circ\xi^{-1}
=
\sum_{i=1}^{n}\frac{\partial}{\partial e_{i}[n]}\bigl(
\frac{\partial\mh{L}_{\mf{e}}}{\partial Y_{i}}\circ(\mf{d}_{O})_{\natural}\circ(u\boxtimes\partial^{\xi}u)\circ\xi^{-1}
\bigr).
\end{equation*}
\end{definition}
\begin{lemma}
\label{08060533}
$u\in\mr{EL}_{W}(\mh{L})$ iff 
\begin{equation*}
\left(\frac{\partial\mh{L}}{\partial Y_{0}}\circ(\Upphi(u)\boxtimes\Uppsi(u))\right)\circ(\mf{o}_{\ze}\boxtimes\mr{Id}_{U})
=
\sum_{i=1}^{n}\frac{\partial}{\partial e_{i}^{2}[n]}\bigl(
\frac{\partial\mh{L}_{\mf{e}}}{\partial Y_{i}}\circ(\mf{d}_{O})_{\natural}\circ(\Upphi(u)\boxtimes\Uppsi(u))
\bigr)\circ(\mf{o}_{\ze}\boxtimes\mr{Id}_{U}).
\end{equation*}
\end{lemma}
\begin{proof}
Since \eqref{08060435} we have $u\in\mr{EL}_{W}(\mh{L})$ iff
\begin{equation*}
\frac{\partial\mh{L}}{\partial Y_{0}}\circ(\phi_{\ze}^{\xi}(u)\boxtimes\psi_{\ze}^{\xi}(u))
=
\sum_{i=1}^{n}\frac{\partial}{\partial e_{i}[n]}\bigl(
\frac{\partial\mh{L}_{\mf{e}}}{\partial Y_{i}}\circ(\mf{d}_{O})_{\natural}\circ(\phi_{\ze}^{\xi}(u)\boxtimes\psi_{\ze}^{\xi}(u))
\bigr),
\end{equation*}
then the statement follows since \eqref{08161354} and Lemma \ref{08080256}. 
\end{proof}
\begin{proposition}
\label{08060629a}
Let $P$ and $Q$ locally convex spaces on $\R$, $\upphi\in\mc{L}(P,Q)$, $R$ be a family of normed spaces on $\R$
defined on $[1,m]\cap\Z$ with $m\in\Z_{>0}$, $B$ be an open set of $\prod_{i=1}^{m}R_{i}$, $C\subseteq B$, $f:B\to P$ be a 
differentiable map on $C$. Thus for every $k\in[1,m]\cap\Z$ and every $v\in R_{k}$ we have 
\begin{equation*}
\mr{ev}_{v}^{\mc{L}_{b}(R_{k},P)}\circ D_{k}(\upphi\circ f)=\upphi\circ\mr{ev}_{v}^{\mc{L}_{b}(R_{k},P)}\circ(D_{k}f).
\end{equation*}
\end{proposition}
\begin{proof}
Let $y\in C$ then
\begin{equation*}
\begin{aligned}
(\mr{ev}_{v}^{\mc{L}_{b}(R_{k},P)}\circ D_{k}(\upphi\circ f))(y)
&=
D_{k}(\upphi\circ f)(y)(v)
\\
&=
(\upphi\circ f)^{(1)}(y)(\mf{i}_{k}v)
\\
&=
\upphi(f^{(1)}(y)(\mf{i}_{k}v))
\\
&=
\upphi((D_{k}f)(y)(v))
\\
&=
(\upphi\circ\mr{ev}_{v}^{\mc{L}_{b}(R_{k},P)}\circ(D_{k}f))(y).
\end{aligned}
\end{equation*}
where the 2th equality follows since Cor.\ref{02211942}; the 3th one by Thm.\ref{821} and Prp.\ref{813}; the 4th one by
Cor.\ref{02211942}.
\end{proof}
\begin{corollary}
\label{08060629b}
Under the hypothesis of Prp.\ref{08060629a} assume in addition that $R_{i}=\R$ for every $i\in[1,N]\cap\Z$, thus
\begin{equation*}
\frac{\partial}{\partial e_{k}}(\upphi\circ f)=\upphi\circ\frac{\partial f}{\partial e_{k}}.
\end{equation*}
\end{corollary}
\begin{proof}
By Prp.\ref{08060629a} and Prp.\ref{08091052p}.
\end{proof}
In what follows we let $\{\partial_{i}^{U}\}_{i=1}^{n}$ denote the basis of vector fields on $U$ associated with the chart 
$(U,\mr{Id}_{U})$. Next we introduce the main concept of the paper namely
\begin{definition}
[\textbf{Universal Operatorial Noether Current}]
\label{09302310}
Let $u\in\mf{Z}^{O}$, $i\in[1,n]\cap\Z$ and $k\in[1,2]\cap\Z$, define 
$\mr{N}_{k}^{i}[u]\in\mc{C}^{p-1}(A\times U,\mc{L}_{b}(T,\mc{Z}_{\R}))$
such that 
\begin{equation*}
\mr{N}_{1}^{i}[u]
\coloneqq
\mf{t}_{T,Y_{i},\mc{Z}}
\circ\bigl\{
\bigl[
\frac{\partial\mh{L}_{\mf{e}}}{\partial Y_{i}}
\circ
(\mf{d}_{O})_{\natural}
\circ
(\Upphi(u)
\boxtimes
\Uppsi(u))
\bigr]
\boxtimes
\bigl[D_{1}(\Upphi(u))\bigr]\bigr\};
\end{equation*}
and
\begin{equation*}
\mr{N}_{2}^{i}[u]
\coloneqq
\mf{e}_{T,\mc{Z}}
\circ
\bigl\{
\bigl[
\mh{L}
\circ
(\Upphi(u)
\boxtimes
\Uppsi(u))
\bigr]
\boxtimes
[D_{1}(\widehat{\chi}_{i})]
\bigr\}.
\end{equation*}
Next define the map $\mh{N}^{i}[u]\in\mc{C}^{p-1}(U,\mc{L}_{b}(T,\mc{Z}_{\R}))$ such that 
\begin{equation*}
\mh{N}^{i}[u]
\coloneqq
\sum_{k=1}^{2}
\mr{N}_{k}^{i}[u]
\circ(\mf{o}_{\ze}\boxtimes\mr{Id}_{U}).
\end{equation*}
Now we define \textbf{the universal operatorial Noether current relative to $u$}, 
or simply operatorial Noether current relative to $u$, the element 
\begin{equation*}
\mh{N}[u]\in\mc{C}^{p-1}(U,\mc{L}_{b}(T,\mc{Z}_{\R}))\otimes_{\mc{C}^{\infty}(U)}\Gamma(U,TM),
\end{equation*}
such that 
\begin{equation*}
\mh{N}[u]\coloneqq\sum_{i=1}^{n}\mh{N}^{i}[u]\otimes\partial_{i}^{U}.
\end{equation*}
If in addition $h\in T$, then define $\mh{N}^{i}[u](h)\in\mc{C}^{p-1}(U,\mc{Z})$ such that 
\begin{equation*}
\mh{N}^{i}[u](h)
\coloneqq
\mr{ev}_{h}^{\mc{L}_{b}(T,\mc{Z}_{\R})}\circ\mh{N}^{i}[u].
\end{equation*}
We call $\mh{N}^{i}[u]$ the $i$-th component of the operatorial Noether current relative to $u$, while 
call $\mh{N}^{i}[u](h)$ the $h$-evaluation of the $i$-th component of the operatorial Noether current relative to $u$.
\end{definition}
\begin{proposition}
\label{10282047}
Def.\ref{09302310} is well-set more specifically for any $u\in\mf{Z}^{O}$, $i\in[1,n]\cap\Z$ and $k\in[1,2]\cap\Z$, we have 
$\mr{N}_{k}^{i}[u]\in\mc{C}^{p-1}(A\times U,\mc{L}_{b}(T,\mc{Z}_{\R}))$.
\end{proposition}
\begin{proof}
$\mr{N}_{k}^{i}[u]\in\mc{C}^{p-1}(A\times U,\mc{L}_{b}(T,\mc{Z}_{\R}))$ follows by Def.\ref{05101630}, Rmk.\ref{051016304} and since 
Lemma \ref{10221256}, Lemma \ref{10221931}, Thm.\ref{81210}, Cor.\ref{04221226}, Prp.\ref{08091052}, Cor.\ref{8126} and 
Prp.\ref{01101646}. 
\end{proof}
\begin{remark}
Let $u\in\mf{Z}^{O}$, $h\in T$ and $i\in[1,n]\cap\Z$, then by Rmk.\ref{10231439} we obtain
\begin{multline*}
\mh{N}^{i}[u](h)
=
\Bigr\{
\mf{g}_{i}
\circ
\left(
\bigl(
\frac{\partial\mh{L}_{\mf{e}}}{\partial Y_{i}}\circ(\mf{d}_{O})_{\natural}\circ(\Upphi(u)\boxtimes\Uppsi(u))
\bigr)
\boxtimes
\bigl(
\mr{ev}_{h}^{\mc{L}_{b}(T,(Y_{i})_{\R})}\circ(D_{1}(\Upphi(u)))
\bigr)
\right)
\\
+
\bigl(\mh{L}\circ(\Upphi(u)\boxtimes\Uppsi(u))\bigr)
\cdot
\bigl(\mr{ev}_{h}^{\mc{L}_{b}(T,\R)}\circ(D_{1}\widehat{\chi}_{i})\bigr)
\Bigr\}
\circ(\mf{o}_{\ze}\boxtimes\mr{Id}_{U}).
\end{multline*}
\end{remark}
In the following Cor.\ref{12131752} we exhibit the divergence of $\mh{N}[u]$ as a consequence of \cite[Thm.3.47]{sil09261621}.
Here we recall some notation needed to decript the result below established. 
Given a Hausdorff locally convex space $H$ over $\K$, an open set $V$ of $M$ and a volume form on $V$, 
the divergence operator we mean is a map $\mr{div}_{\omega,p,H_{\R}}$ constructed in \cite[Cor.3.41]{sil09261621} 
defined on $\mc{C}^{p}(V,H_{\R})\otimes_{\mc{C}^{\infty}(V)}\Gamma(V,TM)$ and valued in $\mc{C}^{p-1}(V,\R\widehat{\otimes}H_{\R})$.
If $G$ and $H$ are locally convex spaces over $\K$, then $\widehat{G}$ is the Hausdorff completion of $G$, $\mf{j}_{G}$ is the 
canonical map of $G$ into $\widehat{G}$. If $T\in\mc{L}(G,H)$, then we let $\widehat{T}\in\mc{L}(\widehat{G},\widehat{H})$ denote
the unique continuous extension of $T$, namely the unique map $\widehat{T}$ such that 
$\widehat{T}\circ\mf{j}_{G}=\mf{j}_{H}\circ T$. We let $G\otimes H$ be the projective topological tensor product of $G$ and $H$, 
while $G\widehat{\otimes}H$ stands for $\widehat{G\otimes H}$, finally the map $\gamma_{G}\in\mc{L}(\K\otimes G,G)$ is such that 
$\gamma_{G}(\lambda\otimes x)=\lambda x$, for every $(\lambda,x)\in\K\times G$.
\begin{corollary}
[\textbf{Divergence of $\mh{N}[u]$}]
\label{12131752}
Assume $p=+\infty$, let $u\in\mf{Z}^{O}$ and set $G=\mc{L}_{b}(T,\mc{Z}_{\R})$, then
\begin{equation*}
((\widehat{\gamma_{G}})_{\ast}
\circ\mr{div}_{v_{g[U]},\infty,G})(\mh{N}[u])
=
\\
\mf{j}_{G}\bigl(\sum_{i=1}^{n}\frac{\partial\mh{N}^{i}[u]}{\partial e_{i}[n]}\bigr),
\end{equation*}
where $g[U]$ is the euclidean metric on $U$ and $v_{g[U]}$ be the volume form on $U$ induced by $g[U]$.
\end{corollary}
\begin{proof}
Since \cite[Cor.3.48]{sil09261621} applied to $\uppsi=\mr{Id}_{U}$.
\end{proof}
By developping $D_{1}(\Upphi(u))$ in Lemma \ref{07220332} we will display in Prp.\ref{08110009} the component $\mh{N}^{i}[u](h)$ 
in a different form. 
\begin{definition}
Define 
\begin{equation*}
\begin{aligned}
H:A\times U&\to\mc{L}(T,\mathbb{M}_{n}(\R));
\\
(\eta,x)&\mapsto\left(h\mapsto
\bigl((\mu,\nu)\mapsto\frac{\partial(D_{1}\widehat{\chi}_{\mu})}{\partial e_{\nu}^{2}[n]}(\eta,x)(h)\bigr)\right).
\end{aligned}
\end{equation*}
\end{definition}
\begin{proposition}
\label{07232347}
For every $h\in T$ we have 
\begin{equation*}
\mr{ev}_{h}^{\mc{L}_{b}(T,\R)}\circ(D_{1}\widehat{J})
=
\widehat{J}
\cdot
\left[\mr{Tr}\circ\mf{p}\circ\left((\mr{Inv}\circ\widehat{M})\boxtimes(\mr{ev}_{h}^{\mc{L}_{b}(T,\mathbb{M}_{n}(\R))}\circ H)\right)
\right];
\end{equation*}
in particular 
\begin{equation*}
(\mr{ev}_{h}^{\mc{L}_{b}(T,\R)}\circ (D_{1}\widehat{J}))\circ(\mf{o}_{\ze}\boxtimes\mr{Id}_{U})
=
\sum_{i=1}^{n}
\left(\mr{ev}_{h}^{\mc{L}_{b}(T,\R)}\circ\frac{\partial\bigl(D_{1}\widehat{\chi}_{i}\bigr)}{\partial e_{i}^{2}[n]}\right)
\circ(\mf{o}_{\ze}\boxtimes\mr{Id}_{U}).
\end{equation*}
\end{proposition}
Def.\ref{05101630}\eqref{06291612} implies $\widehat{J}\in\mc{C}^{1}(A\times U,\R)$ so, the statement of Prp.\ref{07232347}
is well-set.
\begin{proof}[Proof of Prp.\ref{07232347}]
It is well-known that for any $a\in\mathbb{GL}_{n}(\R)$ and $b\in\mathbb{M}_{n}(\R)$ 
we have 
\begin{equation}
\label{09250752}
\mr{det}^{(1)}(a)(b)=\mr{det}(a)\mr{Tr}\left(\mf{p}(\mr{Inv}(a),b)\right).
\end{equation}
Next if $Q$ is an open set of a real normed space $N$ and $f:Q\to\mathbb{M}_{n}(\R)$ is a map differentiable at $\eta\in Q$, 
then since Thm.\ref{821} we obtain $(\mr{det}\circ f)^{(1)}(\eta)(h)=\mr{det}^{(1)}(f(\eta))(f^{(1)}(\eta)(h))$ 
for every $h\in N$, and if in addition $f(\eta)\in\mathbb{GL}_{n}(\R)$, then since \eqref{09250752} we conclude
\begin{equation}
\label{09250732a}
(\mr{det}\circ f)^{(1)}(\eta)(h)
=
(\mr{det}\circ f)(\eta)\mr{Tr}\left(\mf{p}\bigl(\mr{Inv}(f(\eta)),f^{(1)}(\eta)(h)\bigr)\right).
\end{equation}
Now let $x\in U$ thus, under the positions $Q=A$, $N=T$ and $f=\widehat{M}(\cdot,x)$ we have 
$(\mr{det}\circ f)(\eta)=\widehat{J}(\eta,x)$ and then 
\begin{equation}
\label{09250732b}
(\mr{det}\circ f)^{(1)}(\eta)(h)=(D_{1}\widehat{J})(\eta,x)(h)=(\mr{ev}_{h}^{\mc{L}_{b}(T,\R)}\circ(D_{1}\widehat{J}))(\eta,x),
\end{equation}
while $f^{(1)}(\eta)(h)=(D_{1}\widehat{M})(\eta,x)(h)$ then since Prp.\ref{815} and Cor.\ref{06300307} we obtain 
\begin{equation}
\label{09250732c}
f^{(1)}(\eta)(h)_{\mu,\nu}=\frac{\partial(D_{1}\widehat{\chi}_{\mu})}{\partial e_{\nu}^{2}[n]}(\eta,x)(h).
\end{equation}
The 1th equality in the statement then follows since \eqref{09250732a}, \eqref{09250732b} and \eqref{09250732c}; while the 2th
equality follows by the 1th one, and by $\widehat{J}(\ze,x)=1$ and $(\mr{Inv}\circ\widehat{M})(\ze,x)=\un_{n}$, indeed 
$\widehat{M}(\ze,x)=\un_{n}$ since $\chi(\ze)=\mr{Id}_{U}$.
\end{proof}
Now we are able to establish the following
\begin{theorem}
[\textbf{Differential of $\mc{A}(u)$ at $\ze$ and on-shell}]
\label{08070358}
Let $u\in\mr{EL}_{W}(\mh{L})$ thus for every $h\in T$ and $f\in\mc{H}(W)$ we have 
\begin{equation}
\label{08070358a}
\mc{A}(u)^{(1)}(\ze)(h)(f)
=
\mr{ev}_{h}^{\mc{L}_{b}(T,\mc{Z}_{\R})}
\left(\int
(\xi^{-1})^{\ast}(f)
\sum_{i=1}^{n}\frac{\partial\mh{N}^{i}[u]}{\partial e_{i}[n]}
d\lambda_{U}\right).
\end{equation}
Equivalently
\begin{equation}
\label{08070358b}
\mc{A}(u)^{(1)}(\ze)(h)(f)
=
\int
(\xi^{-1})^{\ast}(f)
\sum_{i=1}^{n}\frac{\partial\mh{N}^{i}[u](h)}{\partial e_{i}[n]}
d\lambda_{U}.
\end{equation}
\end{theorem}
\begin{proof}
We have
\begin{multline}
\label{07240315}
\begin{aligned}
\mf{g}_{0}\circ
\left(
\bigl(
\frac{\partial\mh{L}}{\partial Y_{0}}
\circ
(\Upphi(u)
\boxtimes
\Uppsi(u))
\bigr)
\boxtimes
\bigl(
\mr{ev}_{h}^{\mc{L}_{b}(T,(Y_{0})_{\R})}\circ(D_{1}(\Upphi(u)))
\bigr)
\right)
\circ
(\mf{o}_{\ze}\boxtimes\mr{Id}_{U})
+
\\
\sum_{i=1}^{n}
\mf{g}_{i}\circ
\left(
\bigl(
\frac{\partial\mh{L}_{\mf{e}}}{\partial Y_{i}}
\circ(\mf{d}_{O})_{\natural}
\circ
(\Upphi(u)
\boxtimes
\Uppsi(u))
\bigr)
\boxtimes
\bigl(
\mr{ev}_{h}^{\mc{L}_{b}(T,(Y_{0})_{\R})}\circ\bigl(\frac{\partial}{\partial e_{i}^{2}[n]}(D_{1}(\Upphi(u)))\bigr)
\bigr)
\right)
\circ
(\mf{o}_{\ze}\boxtimes\mr{Id}_{U})
&=
\\
\sum_{i=1}^{n}
\mf{g}_{i}\circ
\left(
\frac{\partial}{\partial e_{i}^{2}[n]}
\bigl(
\frac{\partial\mh{L}_{\mf{e}}}{\partial Y_{i}}\circ(\mf{d}_{O})_{\natural}\circ(\Upphi(u)\boxtimes\Uppsi(u))
\bigr)
\boxtimes
\bigl(
\mr{ev}_{h}^{\mc{L}_{b}(T,(Y_{0})_{\R})}\circ(D_{1}(\Upphi(u)))
\bigr)
\right)
\circ
(\mf{o}_{\ze}\boxtimes\mr{Id}_{U})
+
\\
\mf{g}_{i}\circ
\left(
\bigl(
\frac{\partial\mh{L}_{\mf{e}}}{\partial Y_{i}}
\circ(\mf{d}_{O})_{\natural}
\circ
(\Upphi(u)
\boxtimes
\Uppsi(u))
\bigr)
\boxtimes
\bigl(
\mr{ev}_{h}^{\mc{L}_{b}(T,(Y_{0})_{\R})}\circ\bigl(\frac{\partial}{\partial e_{i}^{2}[n]}(D_{1}(\Upphi(u)))\bigr)
\bigr)
\right)
\circ
(\mf{o}_{\ze}\boxtimes\mr{Id}_{U})
&=
\\
\sum_{i=1}^{n}
\mf{g}_{i}\circ
\left(
\frac{\partial}{\partial e_{i}^{2}[n]}\bigl(
\frac{\partial\mh{L}_{\mf{e}}}{\partial Y_{i}}\circ(\mf{d}_{O})_{\natural}\circ(\Upphi(u)\boxtimes\Uppsi(u))
\bigr)
\boxtimes
\bigl(
\mr{ev}_{h}^{\mc{L}_{b}(T,(Y_{0})_{\R})}\circ(D_{1}(\Upphi(u)))
\bigr)
\right)
\circ
(\mf{o}_{\ze}\boxtimes\mr{Id}_{U})
+
\\
\mf{g}_{i}\circ
\left(
\bigl(
\frac{\partial\mh{L}_{\mf{e}}}{\partial Y_{i}}
\circ(\mf{d}_{O})_{\natural}
\circ
(\Upphi(u)
\boxtimes
\Uppsi(u))
\bigr)
\boxtimes
\bigl(\frac{\partial}{\partial e_{i}^{2}[n]}
\bigl(\mr{ev}_{h}^{\mc{L}_{b}(T,(Y_{0})_{\R})}\circ(D_{1}(\Upphi(u)))
\bigr)
\bigr)
\right)
\circ
(\mf{o}_{\ze}\boxtimes\mr{Id}_{U})
&=
\\
\sum_{i=1}^{n}
\frac{\partial}{\partial e_{i}^{2}[n]}
\left[
\mf{g}_{i}
\circ
\left(
\bigl(
\frac{\partial\mh{L}_{\mf{e}}}{\partial Y_{i}}\circ(\mf{d}_{O})_{\natural}\circ(\Upphi(u)\boxtimes\Uppsi(u))
\bigr)
\boxtimes
\bigl(
\mr{ev}_{h}^{\mc{L}_{b}(T,(Y_{0})_{\R})}\circ(D_{1}(\Upphi(u)))
\bigr)
\right)
\right]
\circ
(\mf{o}_{\ze}\boxtimes\mr{Id}_{U})
\end{aligned}
\end{multline}
where the 1th equality follows by Lemma \ref{08060533} and the bilinearity of $\mf{g}_{r}$ for $r\in[0,n]\cap\Z$,
the 2th one by Cor.\ref{08060629b}, the 3th one since Prp.\ref{03241344} and Prp.\ref{08091051}.
Next 
\begin{multline}
\label{07240317}
\begin{aligned}
(F_{u}\cdot(\mr{ev}_{h}^{\mc{L}_{b}(T,\R)}\circ(D_{1}\widehat{J})))
\circ
(\mf{o}_{\ze}\boxtimes\mr{Id}_{U})
+
\\
\left[
\sum_{i=1}^{n}
\bigl(
\frac{\partial}{\partial e_{i}^{2}[n]}
\bigl(
\mh{L}
\circ
(\Upphi(u)
\boxtimes
\Uppsi(u))
\bigr)
\bigr)
\cdot
(\mr{ev}_{h}^{\mc{L}_{b}(T,\R)}\circ D_{1}(\widehat{\chi}_{i}))
\right]
\circ
(\mf{o}_{\ze}\boxtimes\mr{Id}_{U})
&=
\\
\bigl[
\bigl(\mh{L}\circ(\Upphi(u)\boxtimes\Uppsi(u))\bigr)
\cdot
\sum_{i=1}^{n}
\frac{\partial}{\partial e_{i}^{2}[n]}
\bigl(\mr{ev}_{h}^{\mc{L}_{b}(T,\R)}\circ(D_{1}\widehat{\chi}_{i})\bigr)
+
\\
\sum_{i=1}^{n}
\frac{\partial}{\partial e_{i}^{2}[n]}
\bigl(
\mh{L}
\circ
(\Upphi(u)
\boxtimes
\Uppsi(u))
\bigr)
\cdot
(\mr{ev}_{h}^{\mc{L}_{b}(T,\R)}\circ D_{1}(\widehat{\chi}_{i}))
\bigr]
\circ
(\mf{o}_{\ze}\boxtimes\mr{Id}_{U})
&=
\\
\sum_{i=1}^{n}
\bigl[
\bigl(\mh{L}\circ(\Upphi(u)\boxtimes\Uppsi(u))\bigr)
\cdot
\frac{\partial}{\partial e_{i}^{2}[n]}
\bigl(\mr{ev}_{h}^{\mc{L}_{b}(T,\R)}\circ(D_{1}\widehat{\chi}_{i})\bigr)
+
\\
\frac{\partial}{\partial e_{i}^{2}[n]}
\bigl(
\mh{L}
\circ
(\Upphi(u)
\boxtimes
\Uppsi(u))
\bigr)
\cdot
(\mr{ev}_{h}^{\mc{L}_{b}(T,\R)}\circ D_{1}(\widehat{\chi}_{i}))
\bigr]
\circ
(\mf{o}_{\ze}\boxtimes\mr{Id}_{U})
&=
\\
\sum_{i=1}^{n}
\frac{\partial}{\partial e_{i}^{2}[n]}
\bigl[
\bigl(\mh{L}\circ(\Upphi(u)\boxtimes\Uppsi(u))\bigr)
\cdot
\bigl(\mr{ev}_{h}^{\mc{L}_{b}(T,\R)}\circ(D_{1}\widehat{\chi}_{i})\bigr)
\bigr]
\circ
(\mf{o}_{\ze}\boxtimes\mr{Id}_{U});
\end{aligned}
\end{multline}
where the 1th equality follows since Prp.\ref{07232347} and Cor.\ref{08060629b};
the 3th one since Prp.\ref{03241344} and Prp.\ref{08091051}. Statement \eqref{08070358b}
then follows since Thm.\ref{07222143}, \eqref{07240315}, \eqref{07240317} and Lemma \ref{08080256}.
Next by definition $\mh{N}^{i}[u]\in\mc{C}^{p-1}(U,\mc{L}_{b}(T,\mc{Z}_{\R}))$ so, 
\begin{equation*}
\frac{\partial\mh{N}^{i}[u]}{\partial e_{i}[n]}\in\mc{C}^{p-2}(U,\mc{L}_{b}(T,\mc{Z}_{\R})),
\end{equation*}
since Prp.\ref{08091052} and Cor.\ref{04221226} therefore the operator 
\begin{equation*}
f\mapsto
\int
(\xi^{-1})^{\ast}(f)
\sum_{i=1}^{n}\frac{\partial\mh{N}^{i}[u]}{\partial e_{i}[n]}
d\lambda_{U},
\end{equation*}
belongs to $\mc{M}(W,\mc{L}_{b}(T,\mc{Z}_{\R}))$ since Lemma \ref{12311713}, Cor.\ref{06162059} and Cor.\ref{08111037},
in particular the integral belongs to $\mc{L}_{b}(T,\mc{Z}_{\R})$ thus, \eqref{08070358b} is equivalent to \eqref{08070358a}
which then for what above proven holds true.
\end{proof}
\begin{remark}
\label{10031813}
Let $\mf{k}:U\to\R$ be a locally $\lambda_{U}$-integrable map. Thus
\begin{equation*}
(\forall f\in\mc{H}(W))\left(\int(\xi^{-1})^{\ast}(f)\cdot\mf{k}\,d\lambda_{U}=0\right)
\Rightarrow
\mf{k}=\ze.
\end{equation*}
Indeed $\xi^{-1}$ is a homeomorphism being a diffeomorphism so, by Cor.\ref{08111037} $(\xi^{-1})^{\ast}\mc{H}(W)=\mc{H}(U)$ 
while by definition $\mc{C}_{c}^{\infty}(U)\subset\mc{H}(U)$ 
therefore, the above implication follows since the fundamental lemma of calculus of variations \cite[p.16-17]{gh}.
\end{remark}
Next we can establish the following conservation law:
\begin{corollary}
[\textbf{Conservation of the operatorial Noether current}]
\label{08110158}
Let $u\in\mr{EL}_{W}(\mh{L})$ thus 
\begin{equation*}
\mc{A}(u)^{(1)}(\ze)=\ze
\Rightarrow
\sum_{i=1}^{n}\frac{\partial\mh{N}^{i}[u]}{\partial e_{i}[n]}=\ze.
\end{equation*}
\end{corollary}
Equivalently we can express the above result in terms of conservation for every $h\in T$ of the $h$-evaluation of the 
operatorial Noether current relative to $u$: 
\begin{equation}
\label{11051310}
\mc{A}(u)^{(1)}(\ze)=\ze
\Rightarrow
(\forall h\in T)\left(\sum_{i=1}^{n}\frac{\partial\mh{N}^{i}[u](h)}{\partial e_{i}[n]}=\ze\right).
\end{equation}
\begin{remark}
\label{08281125}
Let $u\in\mf{Z}^{O}$, if $\mc{A}(u)$ is a constant map then, $\mc{A}(u)^{(1)}(\eta)=\ze$ for every $\eta\in A$.
Next notice that the following four statements are equivalent: $\mf{a}$ is constant; 
$(\forall\eta\in A)(\mf{a}(\eta)=\mf{a}(\ze))$; 
$(\forall\eta\in A)(\forall u\in\mf{Z}^{O})(\mf{A}(u,\eta)=\mf{A}(u,\ze))$;
$(\forall u\in\mf{Z}^{O})(\mc{A}(u)\text{ is constant})$. Thus if $\mf{a}$ is constant, 
then $\mc{A}(u)^{(1)}(\eta)=\ze$ for every $\eta\in A$ and every $u\in\mf{Z}^{O}$.
\end{remark}
\begin{proof}[Proof of Corollary \ref{08110158}]
Assume that $\mc{A}(u)^{(1)}(\ze)=\ze$ thus, by Thm.\ref{08070358} we have 
\begin{equation*}
(\forall f\in\mc{H}(W))
(\forall h\in T)
\left(\int
(\xi^{-1})^{\ast}(f)
\cdot
\sum_{i=1}^{n}\frac{\partial\mh{N}^{i}[u](h)}{\partial e_{i}[n]}d\lambda_{U}=\ze
\right),
\end{equation*}
and then 
\begin{equation*}
(\forall f\in\mc{H}(W))
(\forall\gamma\in\mc{L}(\mc{Z}_{\R},\R))
(\forall h\in T)
\left(\int
(\xi^{-1})^{\ast}(f)
\cdot
\left(\gamma\circ\sum_{i=1}^{n}\frac{\partial\mh{N}^{i}[u](h)}{\partial e_{i}[n]}\right)d\lambda_{U}=\ze
\right),
\end{equation*}
therefore since Rmk.\ref{10031813} applied to the map 
\begin{equation*}
\gamma\circ\sum_{i=1}^{n}\frac{\partial\mh{N}^{i}[u](h)}{\partial e_{i}[n]},
\end{equation*} 
we obtain
\begin{equation*}
(\forall\gamma\in\mc{L}(\mc{Z}_{\R},\R))
(\forall h\in T)
\left(\gamma\circ\sum_{i=1}^{n}\frac{\partial\mh{N}^{i}[u](h)}{\partial e_{i}[n]}=\ze\right),
\end{equation*}
then the statement follows since $\mc{L}(\mc{Z}_{\R},\R)$ separates the points of $\mc{Z}_{\R}$ being $\mc{Z}_{\R}$ of Hausdorff.
\end{proof}
Now we are in the position of exhibiting the main result of the paper namely:
\begin{corollary}
[\textbf{Conservation of the operatorial Noether current in terms of divergence}]
\label{12142339}
Assume $p=+\infty$, let $u\in\mr{EL}_{W}(\mh{L})$ and set $G=\mc{L}_{b}(T,\mc{Z}_{\R})$, then
\begin{equation*}
\mc{A}(u)^{(1)}(\ze)=\ze
\Rightarrow
\mr{div}_{v_{g[U]},\infty,G}(\mh{N}[u])=\ze.
\end{equation*}
\end{corollary}
\begin{proof}
$\gamma_{G}$ is a linear and continuous isomorphism by \cite[Lemma 3.33]{sil09261621} then it is so also $\widehat{\gamma_{G}}$ 
and so, $(\widehat{\gamma_{G}})_{\ast}$ is a $\mc{C}^{\infty}(U)$-module isomorphism since \cite[Lemma 2.40]{sil09261621}.
Then the statement follows since Cor.\ref{08110158} and Cor.\ref{12131752}.
\end{proof}
As direct consequence of the conservation law stated in Cor.\ref{08110158} we obtain the following conservation of 
operatorial charges.
In what follows for any open subset $V$ of $\R^{n-1}$ we let $\{\partial_{i}^{V}\}_{i=1}^{n-1}$ denote the basis of vector 
fields on $V$ associated with the chart $(V,\mr{Id}_{V})$. 
\begin{corollary}
[\textbf{Conservation of operatorial Noether charges relative to $u$}]
\label{05112017}
Assume $n>1$, let $u\in\mr{EL}_{W}(\mh{L})$. For any $i\in[1,n]\cap\Z$ let $J_{i}$ be an open interval of $\R$ such that 
$J=\prod_{i=1}^{n}J_{i}\subseteq U$ which exist since the set of elementary sets constitutes a base of the product topology 
\cite[I.24]{BourGT}. Let $V=\prod_{i=2}^{n}J_{i}$ and $\lambda_{V}$ be the Lebesgue measure on $V$. 
Let for every $t\in J_{1}$ the map $c_{t}:V\to J_{1}$, $z\mapsto t$ and define
\begin{equation*}
Q_{V}(u):J_{1}\to\mc{L}_{b}(T,\mc{Z}_{\R}),
t\mapsto\int_{V}\mh{N}^{1}[u]\circ\imath_{J}^{U}\circ(c_{t}\boxtimes\mr{Id}_{V})\,d\lambda_{V}.
\end{equation*}
If $p=+\infty$, then the following conservation law is satisfied 
\begin{equation*}
\mc{A}(u)^{(1)}(\ze)=\ze\Rightarrow\frac{dQ_{V}(u)}{dt}=\ze.
\end{equation*}
\end{corollary}
We call $Q_{V}(u)$ as defined in Cor.\ref{05112017} the operatorial Noether charge relative to $u$ and to $V$.
\begin{proof}[Proof of Cor.\ref{05112017}.]
Let $t\in J_{1}$, we deduce since Prp.\ref{10282047}, and by Def.\ref{08310159} and Def.\ref{12141626} that 
\begin{equation}
\label{08190457}
\begin{aligned}
(\forall i\in[1,n]\cap\Z)(\mh{N}^{i}[u]\circ\imath_{J}^{U}\circ(c_{t}\boxtimes\mr{Id}_{V})\in
\mc{C}_{c}^{p-1}(V,\mc{L}_{b}(T,\mc{Z}_{\R}))),
\\
(\forall i,j\in[1,n]\cap\Z)\left(\frac{\partial\mh{N}^{i}[u]}{\partial e_{j}[n]}\circ\imath_{J}^{U}\circ
(c_{t}\boxtimes\mr{Id}_{V})\in\mc{C}_{c}^{p-2}(V,\mc{L}_{b}(T,\mc{Z}_{\R}))\right);
\end{aligned}
\end{equation}
so they are continuous and compactly supported maps then scalarly $\lambda_{V}$-integrable, therefore according to 
\cite[Def.1.16 and Cor.1.23(2)]{sil09261621} we can define 
\begin{equation*}
\mh{N}[u]_{t}\in\mf{L}_{c}^{1}(V,\mc{L}_{b}(T,\mc{Z}_{\R}),\lambda_{V})\otimes_{\mc{C}^{\infty}(V)}\Gamma(V,TV),
\end{equation*}
such that 
\begin{equation*}
\mh{N}[u]_{t}\coloneqq\sum_{i=2}^{n}\left(\mh{N}^{i}[u]\circ\imath_{J}^{U}\circ
(c_{t}\boxtimes\mr{Id}_{V})\right)\otimes\partial_{i-1}^{V}.
\end{equation*}
Next we deduce that the weak integral against $\lambda_{V}$ of the maps in \eqref{08190457} exists and belongs to 
$\mc{L}_{b}(T,\mc{Z}_{\R})$ since \cite[III.38 Cor.2]{IntBourb} and since $\mc{L}_{b}(T,\mc{Z}_{\R})$ is a Hausdorff locally convex 
quasi-complete space by Lemma \ref{12311713}; as a result $Q_{V}(u)$ is well-defined. Furthermore by \eqref{08190457} we deduce 
since Cor.\ref{01031245b} that 
\begin{equation*}
\frac{dQ_{V}(u)}{dt}(t)=
\int_{V}\frac{\partial\mh{N}^{1}[u]}{\partial e_{1}[n]}\circ\imath_{J}^{U}\circ(c_{t}\boxtimes\mr{Id}_{V})\,d\lambda_{V},
\end{equation*}
so, since Cor.\ref{08110158} we obtain that $\mc{A}(u)^{(1)}(\ze)=\ze$ implies 
\begin{equation}
\label{12141437}
\frac{dQ_{V}(u)}{dt}(t)
=
(-1)\int_{V}\sum_{k=2}^{n}\frac{\partial\mh{N}^{k}[u]}{\partial e_{k}[n]}\circ\imath_{J}^{U}
\circ(c_{t}\boxtimes\mr{Id}_{V})\,d\lambda_{V},
\end{equation}
where the integral in the right-hand side of the above equality belongs to $\mc{L}_{b}(T,\mc{Z}_{\R})$ by what above proven. 
Now assume $p=+\infty$ and let $\phi\in\mc{L}(\mc{L}_{b}(T,\mc{Z}_{\R}),\R)$, then 
\begin{equation}
\begin{aligned}
\label{12141435}
\phi\left(\int_{V}\sum_{k=2}^{n}\frac{\partial\mh{N}^{k}[u]}{\partial e_{k}[n]}\circ\imath_{J}^{U}
\circ(c_{t}\boxtimes\mr{Id}_{V})\,d\lambda_{V}\right)
&=
\int_{V}\sum_{k=2}^{n}\frac{\partial(\phi_{\ast}\circ\mh{N}^{k}[u])}{\partial e_{k}[n]}\circ\imath_{J}^{U}
\circ(c_{t}\boxtimes\mr{Id}_{V})\,d\lambda_{V}
\\
&=
\int_{V}\mr{div}_{V}\bigl((\mf{a}_{\Gamma(V,TV)}\circ\phi_{\times})(\mh{N}[u]_{t})\bigr)\,dv_{V},
\\
&=
\ze.
\end{aligned}
\end{equation}
Here $\phi_{\ast}$ and $\phi_{\times}$ are defined respectively in \cite[Def.2.39 and Def.2.41]{sil09261621}, 
$\mf{a}_{\Gamma(V,TV)}$ is a particular case of a map defined in \cite[Def.3.36]{sil09261621}, and it is the unique 
$\mc{C}^{\infty}(V)$-linear map $\mf{a}_{\Gamma(V,TV)}:\mc{C}^{\infty}(V)\otimes_{\mc{C}^{\infty}(V)}\Gamma(V,TV)\to\Gamma(V,TV)$ such that 
$\mf{a}_{\Gamma(V,TV)}(g\otimes T)=gT$ for every $(g,T)\in\mc{C}^{\infty}(V)\times\Gamma(V,TV)$, while $\mr{div}_{V}$ is the usual 
divergence associated with $v_{V}$ the volume form on $V$ induced by $g_{V}$ the Euclidean metric on $V$.
The 1th equality follows since the definition of weak integral and by \cite[(3.5)]{sil09261621}, the 2th equality follows by 
the definition of $\phi_{\times}$ and $\mf{a}_{\Gamma(V,TV)}$, and since the divergence formula in Cor.\ref{03290108},
the 3th equality follows by the fact that $\partial V=\emptyset$ the boundary of $V$ and since the divergence theorem 
Cor.\ref{03161712}.
Now $\mc{L}(\mc{L}_{b}(T,\mc{Z}_{\R}),\R)$ separates the points of $\mc{L}_{b}(T,\mc{Z}_{\R})$ being of Hausdorff then, we conclude 
that 
\eqref{12141435} implies
\begin{equation*}
\int_{V}\sum_{k=2}^{n}\frac{\partial\mh{N}^{k}[u]}{\partial e_{k}[n]}\circ\imath_{J}^{U}
\circ(c_{t}\boxtimes\mr{Id}_{V})\,d\lambda_{V}=\ze,
\end{equation*}
then the statement follows by \eqref{12141437}.
\end{proof}
We conclude this section by displaying a different form of the components of the operatorial Noether current. 
Let us start with the following
\begin{lemma}
\label{07220332}
Let $u\in\mf{Z}^{O}$ and $h\in T$ thus, 
\begin{multline*}
\mr{ev}_{h}^{\mc{L}_{b}(T,(Y_{0})_{\R})}\circ D_{1}\bigl(\Upphi(u)\bigr)\circ(\beta\boxtimes\widehat{\chi})
=
\\
\mr{ev}_{h}^{\mc{L}_{b}(T,(Y_{0})_{\R})}\circ D_{1}\bigl(\Upphi(u)\circ(\beta\boxtimes\widehat{\chi})\bigr)
-
\sum_{l=1}^{n}
\bigl(\frac{\partial\Upphi(u)}{\partial e_{l}^{2}[n]}\circ(\beta\boxtimes\widehat{\chi})\bigr)
\cdot
(\mr{ev}_{h}^{\mc{L}_{b}(T,\R)}\circ D_{1}\widehat{\chi}_{l}).
\end{multline*}
\end{lemma}
\begin{proof}
\begin{equation*}
\begin{aligned}
D_{1}\bigl(\Upphi(u)\circ(\beta\boxtimes\widehat{\chi})\bigr)(\eta,x)(h)
&=
\\
\bigl(\Upphi(u)\circ(\beta\boxtimes\widehat{\chi})\bigr)^{(1)}(\eta,x)(h,\ze)
&=
\\
(\Upphi(u))^{(1)}(\eta,\widehat{\chi}(\eta,x))(h,(\widehat{\chi})^{(1)}(\eta,x)(h,\ze))
&=
\\
D_{1}\bigl(\Upphi(u)\bigr)(\eta,\widehat{\chi}(\eta,x))(h)
+
D_{2}\bigl(\Upphi(u)\bigr)(\eta,\widehat{\chi}(\eta,x))((D_{1}\widehat{\chi})(\eta,x)h)
&=
\\
D_{1}\bigl(\Upphi(u)\bigr)(\eta,\widehat{\chi}(\eta,x))(h)
+
\bigl(\Upphi(u)\bigr)^{(1)}(\eta,\widehat{\chi}(\eta,x))(\ze,(D_{1}\widehat{\chi})(\eta,x)h);
\end{aligned}
\end{equation*}
where the 1th equality follows since Cor.\ref{02211942} applied to 
$\Upphi(u)\circ(\beta\boxtimes\widehat{\chi})$; the 2th one by Thm.\ref{821} and Prp.\ref{815}; the 3th one by
Cor.\ref{02211942} applied to $\Upphi(u)$ and to $\widehat{\chi}$; the 4th one by Cor.\ref{02211942} applied to
$\Upphi(u)$. The statement then follows since Cor.\ref{02211942} and Prp.\ref{815} applied to the map $\widehat{\chi}$,
and by the 1th equality in \eqref{08092146} applied to $\Upphi(u)$.
\end{proof}
\begin{definition}
We define $\lr{\mr{X}}{\Delta}$ to be a duplet associated with $\lr{\mh{L},\phi}{\omega}$ iff 
\begin{enumerate}
\item
$\mr{X}:[1,n]\cap\Z\to\mc{C}^{p-1}(A\times U,\mc{L}_{b}(T,\R))$ such that $\forall i\in[1,n]\cap\Z$ we have 
\begin{equation*}
\mr{X}^{i}=D_{1}(\widehat{\chi}_{i});
\end{equation*}
\label{051016307}
\item
$\Delta:\mf{Z}^{O}\to\mc{C}^{p-1}(A\times U,\mc{L}_{b}(T,(Y_{0})_{\R}))$ such that $\forall u\in\mf{Z}^{O}$ we have 
\begin{equation*}
\Delta(u)=D_{1}\bigl(\Upphi(u)\circ(\beta\boxtimes\widehat{\chi})\bigr).
\end{equation*}
\label{051016308}
\end{enumerate}
\end{definition}
\begin{definition}
Let $\lr{\mr{X}}{\Delta}$ be a duplet associted with $\lr{\mh{L},\phi}{\omega}$.
If $T=\R^{s}$ with $s\in\Z_{>0}$, then for every $r\in[1,s]\cap\Z$, then define $\mr{X}_{r}:[1,n]\cap\Z\to\mc{C}^{p-1}(U,\R)$ and
$\Delta_{r}:\mf{Z}^{O}\to\mc{C}^{p-1}(U,Y_{0})$ such that $\forall i\in[1,n]\cap\Z$ and $\forall u\in\mf{Z}^{O}$ we have 
\begin{equation*}
\begin{aligned}
\mr{X}_{r}^{i}
&\coloneqq
\mr{ev}_{e_{r}[s]}^{\mc{L}_{b}(T,\R)}\circ\mr{X}^{i}\circ(\mf{o}_{\ze}\boxtimes\mr{Id}_{U});
\\
\Delta_{r}(u)
&\coloneqq
\mr{ev}_{e_{r}[s]}^{\mc{L}_{b}(T,(Y_{0})_{\R})}
\circ(\Delta(u))\circ(\mf{o}_{\ze}\boxtimes\mr{Id}_{U}).
\end{aligned}
\end{equation*}
\end{definition}
In the following Lemma \ref{08112204} and Prp.\ref{08110009} we let $\lr{\mr{X}}{\Delta}$ be a duplet associated with 
$\lr{\mh{L},\phi}{\omega}$. 
\begin{lemma}
\label{08112204}
Let $u\in\mf{Z}^{O}$ and $h\in T$ thus, 
\begin{multline*}
\mr{ev}_{h}^{\mc{L}_{b}(T,(Y_{0})_{\R})}\circ D_{1}\bigl(\Upphi(u)\bigr)\circ(\mf{o}_{\ze}\boxtimes\mr{Id}_{U})
=
\\
\mr{ev}_{h}^{\mc{L}_{b}(T,(Y_{0})_{\R})}\circ\Delta(u)\circ(\mf{o}_{\ze}\boxtimes\mr{Id}_{U})
-
\sum_{l=1}^{n}(\partial_{l}u^{\xi})\cdot(\mr{ev}_{h}^{\mc{L}_{b}(T,\R)}\circ\mr{X}^{l}\circ(\mf{o}_{\ze}\boxtimes\mr{Id}_{U})).
\end{multline*}
\end{lemma}
\begin{proof}
We have 
\begin{equation*}
\mr{ev}_{h}^{\mc{L}_{b}(T,(Y_{0})_{\R})}\circ D_{1}\bigl(\Upphi(u)\bigr)\circ(\mf{o}_{\ze}\boxtimes\mr{Id}_{U})
=
\mr{ev}_{h}^{\mc{L}_{b}(T,(Y_{0})_{\R})}\circ D_{1}\bigl(\Upphi(u)\bigr)\circ(\beta\boxtimes\widehat{\chi})\circ(\mf{o}_{\ze}\boxtimes
\mr{Id}_{U}),
\end{equation*}
since 
\begin{equation*}
(\beta\boxtimes\widehat{\chi})\circ(\mf{o}_{\eta}\boxtimes\mr{Id}_{U})=\mf{o}_{\eta}\boxtimes\chi(\eta),
\end{equation*}
so $(\beta\boxtimes\widehat{\chi})\circ(\mf{o}_{\ze}\boxtimes\mr{Id}_{U})=\mf{o}_{\ze}\boxtimes\mr{Id}_{U}$.
Then the statement follows since Lemma \ref{07220332} and Lemma \ref{08080256}.
\end{proof}
\begin{proposition}
\label{08110009}
Let $u\in\mf{Z}^{O}$, $h\in T$ and $k\in[1,n]\cap\Z$, then
\begin{multline*}
\mh{N}^{k}[u](h)=
\\
\mf{g}_{k}
\circ
\left\{
\left(
\frac{\partial\mh{L}_{\mf{e}}}{\partial Y_{k}}\circ(\mf{d}_{O})_{\natural}\circ(u^{\xi}\boxtimes(\partial u^{\xi}))
\right)
\boxtimes
\left(
\mr{ev}_{h}^{\mc{L}_{b}(T,(Y_{0})_{\R})}\circ\Delta(u)\circ(\mf{o}_{\ze}\boxtimes\mr{Id}_{U})
-
\sum_{l=1}^{n}(\partial_{l}u^{\xi})\cdot(\mr{ev}_{h}^{\mc{L}_{b}(T,\R)}\circ\mr{X}^{l}\circ(\mf{o}_{\ze}\boxtimes\mr{Id}_{U}))
\right)
\right\}
\\
+
\bigl(\mh{L}\circ(u^{\xi}\boxtimes(\partial u^{\xi}))\bigr)
\cdot
\bigl(\mr{ev}_{h}^{\mc{L}_{b}(T,\R)}\circ\mr{X}^{k}\circ(\mf{o}_{\ze}\boxtimes\mr{Id}_{U})\bigr).
\end{multline*}
In particular if $T=\R^{s}$ with $s\in\Z_{>0}$, then for every $r\in[1,s]\cap\Z$ we obtain
\begin{equation}
\label{08110009b}
\mh{N}^{k}(u;e_{r}[s])=
\mf{g}_{k}
\circ
\left(
\left(
\frac{\partial\mh{L}_{\mf{e}}}{\partial Y_{k}}\circ(\mf{d}_{O})_{\natural}\circ(u^{\xi}\boxtimes(\partial u^{\xi}))
\right)
\boxtimes
\left(\Delta_{r}(u)-\sum_{l=1}^{n}(\partial_{l}u^{\xi})\cdot\mr{X}_{r}^{l}
\right)
\right)
+
\bigl(\mh{L}\circ(u^{\xi}\boxtimes(\partial u^{\xi}))\bigr)
\cdot\mr{X}_{r}^{k}.
\end{equation}
\end{proposition}
\begin{proof}
By Lemma \ref{08112204}.
\end{proof}
\begin{remark}
[Reobtaining the classical results as special cases.]
\label{11051131}
Cor.\ref{08110158} and Cor.\ref{05112017} with the help of the above result and under the following very special assumptions:
\begin{itemize}
\item
$\K=\R$,
\item
$\mc{Z}=\R$,
\item
$\K_{1}$ depending by the model considered,
\item
$\mr{dim}X<\infty$,
\item
$T=\R^{s}$ with $s\in\Z_{>0}$;
\end{itemize}
yield the classical results provided in the Bogoliubov text \cite{bog2}. More specifically if $u\in\mf{Z}^{O}$, $k\in[0,n]\cap\Z$
and $r\in[1,s]\cap\Z$, then the $e_{r}[s]$-evaluation $\mh{N}^{k}(u;e_{r}[s])$ of the $k$-th component of the operatorial Noether 
current relative to $u$ provided in \eqref{08110009b}, equals $(-1)$ times the current \cite[p.18 (2.5)]{bog2}.
Notice that in \cite[p.18 (2.5)]{bog2} the index $(n)$ replaces the index $r$, while in \cite[p.18 (2.5)]{bog2} the index $i$ 
runs in $[1,\mr{dim}(X)]$ because of the decomposition, relative to a basis of the finite dimensional $X$, 
of what in \eqref{08110009b} is the partial differential $\frac{\partial\mh{L}_{\mf{e}}}{\partial Y_{k}}$. 
Clearly the index $i\in[1,\mr{dim}(X)]$ could not be present in the more general \eqref{08110009b} since in Prp.\ref{08110009} 
we neither hypothized $X$ to be a cartesian product, nor in particular $X$ to be finite dimensional. Furthermore, if in addition 
to the above assumptions we assume that $u\in\mr{EL}_{W}(\mh{L})$ and $\mc{A}(u)^{(1)}(\ze)=\ze$, for instance when $\mf{a}$ is a 
constant map, then, we deduce since Cor.\ref{08110158} or more specifically \eqref{11051310} that 
\begin{equation}
\label{11061136a}
\sum_{k=1}^{n}\frac{\partial\mh{N}^{k}[u](e_{r}[s])}{\partial e_{k}[n]}=\ze,
\end{equation}
while since Cor.\ref{05112017} we obtain the following law of conservation 
\begin{equation}
\label{11061136b}
\frac{dQ_{r}(u)}{dt}=\ze,
\end{equation}
where $Q_{r}(u)$ equals the evaluation via $e_{r}[s]$ of the operatorial Noether charge $Q_{V}(u)$. 
Now \eqref{11061136a} equals \cite[p.18 (2.7)]{bog2} and \eqref{11061136b} equals \cite[p.19 (2.8)]{bog2} for a suitable $V$. 
\end{remark}
\section{Symmetries of $\mr{vp}(\mf{a}(\ze))$}
\label{11050951}
By recalling \cite[Thm.3.57]{sil09261621} and the fact that $\mc{Z}$ is quasi-complete, we provide the following 
\begin{definition}
\label{08280329}
Define $\upmu:\mf{Z}^{O}\to\mc{M}(W,\mc{Z}_{\R})$ such that for every $u\in\mf{Z}^{O}$
\begin{equation*}
\upmu(u)\doteqdot\mf{m}_{\mf{L}(u,\partial^{\xi}u)\otimes\wedge_{i=1}^{n}d\xi_{i}}.
\end{equation*}  
\end{definition}
In what follows we use Def.\ref{03011632}, while $\mf{a}(\eta)$ is defined in Def.\ref{08280414}.
\begin{proposition}
\label{08280330}
For every $\eta\in A$ we have 
\begin{equation*}
\mf{a}(\eta)=(\upzeta(\eta)^{\ast})^{\intercal}\circ\upmu\circ\imath_{\phi_{\eta}(\mf{Z}_{O})}^{\mf{Z}_{O}}
\circ(\phi_{\eta}\circ\imath_{\mf{Z}_{O}}^{\mf{Z}})_{\natural}.
\end{equation*}
\end{proposition}
\begin{proof}
By \eqref{08161354} and \eqref{08280414a}.
\end{proof}
\begin{lemma}
\label{08280331}
Let $N$ be a normed space on $\R$, and $P,Q$ be Hausdorff locally convex spaces on $\R$
moreover let $A$ and $B$ be open sets of $N$ such that $A\cap B\neq\emptyset$, $f:A\to N$ be a differentiable map at 
$x\in A\cap B$ and such that $f(A)\subseteq B$, $S:B\to P$ be a differentiable map at $x$ and $f(x)$, finally let 
$\mh{I}\in\mc{L}(P,Q)$. Then, 
\begin{equation*}
\mh{I}\circ S\circ\imath_{f(A)}^{B}\circ f_{\natural}\circ\imath_{A\cap B}^{A}=S\circ\imath_{A\cap B}^{B}
\Rightarrow
(S^{(1)}(f(x))=\ze\Rightarrow S^{(1)}(x)=\ze),
\end{equation*}
if in addition we assume that $f^{(1)}(x)$ admits a right inverse and that $\mh{I}$ admits a left inverse, then 
\begin{equation*}
\mh{I}\circ S\circ\imath_{f(A)}^{B}\circ f_{\natural}\circ\imath_{A\cap B}^{A}=S\circ\imath_{A\cap B}^{B}
\Rightarrow
(S^{(1)}(x)=\ze\Leftrightarrow S^{(1)}(f(x))=\ze).
\end{equation*}
\end{lemma}
\begin{proof}
Since Thm.\ref{821}, Prp.\ref{813} and Prp.\ref{05231704} we have 
\begin{equation*}
(\mh{I}\circ S\circ\imath_{f(A)}^{B}\circ f_{\natural}\circ\imath_{A\cap B}^{A})^{(1)}(x)
=
\mh{I}\circ S^{1}(f(x))\circ f^{(1)}(x),
\end{equation*}
while since Prp.\ref{05231704} we have $(S\circ\imath_{A\cap B}^{B})^{(1)}(x)=S^{(1)}(x)$ therefore 
$\mh{I}\circ S\circ\imath_{f(A)}^{B}\circ f_{\natural}\circ\imath_{A\cap B}^{A}=S\circ\imath_{A\cap B}^{B}$ implies 
$\mh{I}\circ S^{1}(f(x))\circ f^{(1)}(x)=S^{(1)}(x)$ and then the statement follows.
\end{proof}
By recalling that $\mf{Z}^{O}$ is an open set of the normed space $\mf{Z}$ according to Lemma \ref{09171929}, we can provide 
the following: 
\begin{definition}
\label{11031405}
If $g:\mf{Z}^{O}\to\mc{M}(W,\mc{Z}_{\R})$ is a differentiable map, then we define the set of solutions of the variational problem 
associated with $g$ the following set 
\begin{equation*}
\mr{vp}(g)\coloneqq\{u\in\mf{Z}^{O}\,\vert\,g^{(1)}(u)=\ze\}.
\end{equation*}
\end{definition}
In Thm.\ref{09142108} we shall provide a sufficient condition ensuring the differentiability of $\mf{a}(\ze)$. 
Next since Prp.\ref{08280330} we have $\mf{a}(\ze)=\upmu$ thus, we can provide the following 
\begin{definition}
Assume that $\upmu$ is differentiable, then $\lr{\mh{L},\phi}{\omega}$ is called perfect at $\eta$ iff $\eta\in A$, 
$\phi_{\eta}\circ\imath_{\mf{Z}^{O}}^{\mf{Z}}:\mf{Z}^{O}\to\mf{Z}$ is differentiable and its differential map 
$(\phi_{\eta}\circ\imath_{\mf{Z}^{O}}^{\mf{Z}})^{(1)}$ is such that $(\phi_{\eta}\circ\imath_{\mf{Z}^{O}}^{\mf{Z}})^{(1)}(u)$ 
admits a right inverse for every $u\in\mr{vp}(\mf{a}(\ze))$. 
Furthermore $\lr{\mh{L},\phi}{\omega}$ is called perfect iff it is perfect at $\eta$ for every 
$\eta\in A$.
\end{definition}
By calling a symmetry of a nonempty set any map defined on and valued in that set, we can state the following  
\begin{theorem}
[\textbf{An invariance of the action $\mf{a}$ induces a symmetry of $\mr{vp}(\mf{a}(\ze))$}]
\label{08280333}
Assume that $\upmu$ is differentiable and $\lr{\mh{L},\phi}{\omega}$ is perfect at $\eta$, then, 
\begin{equation*}
\mf{a}(\eta)=\mf{a}(\ze)\Rightarrow\phi_{\eta}(\mr{vp}(\mf{a}(\ze)))\subseteq\mr{vp}(\mf{a}(\ze)).
\end{equation*}
\end{theorem}
\begin{proof}
$(\upzeta(\eta)^{\ast})^{\intercal}$ is linear and continuous since Prp.\ref{06301252} moreover it is invertible with 
inverse $((\upzeta(\eta)^{-1})^{\ast})^{\intercal}$, then the statement follows since Prp.\ref{08280330} and Lemma \ref{08280331}
applied to $S=\mf{a}(\ze)=\upmu$.
\end{proof}
\begin{corollary}
\label{08281203}
Assume that $\mf{a}$ is constant. Thus, 
\begin{equation*}
(\forall u\in\mr{EL}_{W}(\mh{L}))
\left(\sum_{i=1}^{n}\frac{\partial\mh{N}^{i}[u]}{\partial e_{i}[n]}=\ze\right),
\end{equation*}
in addition under the hypothesis and notation of Cor.\ref{05112017} and if $p=+\infty$, then we have 
\begin{equation*}
(\forall u\in\mr{EL}_{W}(\mh{L}))
\left(\frac{dQ_{V}(u)}{dt}=\ze\right).
\end{equation*}
Furthermore if $\upmu$ is differentiable and $\lr{\mh{L},\phi}{\omega}$ is perfect at $\eta$, then
\begin{equation*}
\phi_{\eta}(\mr{vp}(\mf{a}(\ze)))\subseteq\mr{vp}(\mf{a}(\ze)).
\end{equation*}
\end{corollary}
\begin{proof}
Since Rmk.\ref{08281125}, Cor.\ref{08110158}, Cor.\ref{05112017}, and Thm.\ref{08280333}. 
\end{proof}
\section
[Differential of $\mf{a}(\ze)$ and characterization of $\mr{vp}(\mf{a}(\ze))$]
{Differential of $\mf{a}(\ze)$ and characterization of the set of solutions of the variational problem associated 
with $\mf{a}(\ze)$.}
\label{09171942}
In the previous section we stated Thm.\ref{08280333} and Cor.\ref{08281203} requiring the differentiability of $\mf{a}(\ze)$.
In the present section Thm.\ref{09142108} provides a sufficient condition ensuring the differentiability of $\mf{a}(\ze)$ 
and establishes its differential map; Cor.\ref{10011525} displays a characterization of $\mr{vp}(\mf{a}(\ze))$ in terms of the 
solution set $\mr{sEL}_{W}(\mh{L})$ of a functional differential equation Def.\ref{11011831}; while Thm.\ref{10301959} 
establishes that $\mr{sEL}_{W}(\mh{L})$ and then $\mr{vp}(\mf{a}(\ze))$ is a subset of $\mr{EL}_{W}(\mh{L})$.
As a result Cor.\ref{10011525}, Thm.\ref{09142108} and Thm.\ref{08280333} we can state in Cor.\ref{11031537} 
that the invariance $\mf{a}(\eta)=\mf{a}(\ze)$ implies that $\phi_{\eta}$ is a symmetry of the solution set $\mr{sEL}_{W}(\mh{L})$.
Finally Cor.\ref{10042025} summarizes the consequences of the main results of the paper under the hypothesis that 
$\mf{a}$ is a constant map.
\begin{convention}
Let $\mf{n}$ denote $\mf{s}_{\mc{Y}_{\R},\mc{Z}_{\R}}$ see Def.\ref{11022206}.
\end{convention}
\begin{definition}
\label{11030827}
For every $u\in\mf{Z}^{O}$ and $v\in\mf{Z}$ define $Q(u,v)\in\mc{C}^{p-1}(U,\mc{Z})$ 
such that 
\begin{equation*}
Q(u,v)
\coloneqq
\mf{n}\circ
\left[
\bigl(\mh{L}^{(1)}\circ(\Upphi(u)\boxtimes\Uppsi(u))\circ(\mf{o}_{\ze}\boxtimes\mr{Id}_{U})\bigr)
\boxtimes
\bigl((\Upphi(v)\boxtimes\Uppsi(v))\circ(\mf{o}_{\ze}\boxtimes\mr{Id}_{U})\bigr)
\right].
\end{equation*}
Furthermore define $\mh{Q}[u,v]\in\mc{M}(W,\mc{Z}_{\R})$ such that for every $f\in\mc{H}(W)$ we have 
\begin{equation*}
\mh{Q}[u,v](f)
\coloneqq
\int(\xi^{-1})^{\ast}(f)\cdot Q(u,v)\,d\lambda_{U}.
\end{equation*}
\end{definition}
Explicitly we have 
\begin{equation*}
\mh{Q}[u,v](f)
=
\int(\xi^{-1})^{\ast}(f)(x)\mh{L}^{(1)}(u^{\xi}(x),(\partial u^{\xi})(x))((v^{\xi}(x),(\partial v^{\xi})(x)))\,d\lambda_{U}(x).
\end{equation*}
\begin{proposition}
\label{11030824}
Definition \eqref{11030827} is well-set, specifically for every $u\in\mf{Z}^{O}$ and $v\in\mf{Z}$, we have 
\begin{enumerate}
\item
$Q(u,v)\in\mc{C}^{p-1}(U,\mc{Z})$; 
\label{11030824a}
\item
$\mh{Q}[u,v]\in\mc{M}(W,\mc{Z}_{\R})$. 
\label{11030824b}
\end{enumerate}
\end{proposition}
\begin{proof}
Item \eqref{11030824a} follows by Def.\ref{05101630}, Rmk.\ref{051016304} and since Lemma \ref{10221256}, 
Thm.\ref{81210}, Cor.\ref{8126} and Prp.\ref{01101646}. Next item \eqref{11030824b} follows 
by the proven item \eqref{11030824a}, since Cor.\ref{06162059} and Cor.\ref{08111037}.
\end{proof}
\begin{proposition}
\label{09251745b}
Let $u\in\mf{Z}^{O}$, then 
\begin{equation*}
(e\mapsto\mh{Q}[u,e])\in\mc{L}(\mf{Z}_{\R},\mc{M}(W,\mc{Z}_{\R})).
\end{equation*}
\end{proposition}
\begin{proof}
For every $e\in\mf{Z}$ and $f\in\mc{H}(W)$ we have 
\begin{equation}
\begin{aligned}
\label{09251803}
p(\mh{Q}[u,e](f))
&\leq
\int|(\xi^{-1})^{\ast}(f)|(x)
p\left(\mh{L}^{(1)}(u^{\xi}(x),(\partial u^{\xi})(x))((e^{\xi}(x),(\partial e^{\xi})(x)))\right)\,d\lambda_{U}(x)
\\
&
\leq
\int|(\xi^{-1})^{\ast}(f)|(x)
\overset{1}{p}\left(\mh{L}^{(1)}(u^{\xi}(x),(\partial u^{\xi})(x))\right)
\|(e^{\xi}(x),(\partial e^{\xi})(x))\|_{\mc{Y}}
\,d\lambda_{U}(x),
\\
&
=
\int|(\xi^{-1})^{\ast}(f)|(x)
\overset{1}{p}\left(\mh{L}^{(1)}(u^{\xi}(x),(\partial u^{\xi})(x))\right)
\|(e^{\xi}(x),(\partial e^{\xi})(x))\|_{\mc{Y}}
\,d\lambda_{\xi(\mr{supp}(f))}(x),  
\\
&\leq
C(f)K(p)\widehat{q}(e);
\end{aligned}
\end{equation}
where the first inequality follows since the general result \cite[VI.5, Prp.6]{IntBourb}, while
\begin{equation*}
C(f)\coloneqq
\int|(\xi^{-1})^{\ast}(f)|\,d\lambda_{U},
\end{equation*}
and 
\begin{equation*}
\begin{aligned}
K(p)
&\coloneqq
\sup_{x\in\xi(\mr{supp}(f))}\overset{1}{p}\left(\mh{L}^{(1)}(u^{\xi}(x),(\partial u^{\xi})(x))\right)
\\
&=
\sup_{x\in\xi(\mr{supp}(f))}
\bigl(\overset{1}{p}\circ\mh{L}^{1}\circ(\Upphi(u)\boxtimes\Uppsi(u))\circ(\mf{o}_{\ze}\boxtimes\mr{Id}_{U})\bigr)(x).
\end{aligned}
\end{equation*}
Next we claim that 
\begin{equation}
\label{09251805}
K(p)<\infty.
\end{equation}
To this end we note that 
\begin{equation*}
\mh{L}^{(1)}\circ(\Upphi(u)\boxtimes\Uppsi(u))\circ(\mf{o}_{\ze}\boxtimes\mr{Id}_{U})\in
\mc{C}^{p-1}(U,\mc{L}_{b}(\mc{Y}_{\R},\mc{Z}_{\R})),
\end{equation*}
in particular it is continuous, $\overset{1}{p}$ is continuous, while $\xi(\mr{supp}(f))$ is compact, therefore $K(p)<\infty$ 
since continuous functions map compacts into compacts and compacts in $\R$ are bounded. Therefore taking into account 
Prp.\ref{11030824}\eqref{11030824b}, we have according to \eqref{09152119} that \eqref{09251803} and \eqref{09251805} imply
the statement.
\end{proof}
Prp.\ref{09251745b} legitimates the following 
\begin{definition}
Define the map 
\begin{equation*}
\begin{aligned}
\mh{Q}:\mf{Z}^{O}&\to\mc{L}(\mf{Z}_{\R},\mc{M}(W,\mc{Z}_{\R}));
\\
u&\mapsto(e\mapsto\mh{Q}[u,e]).
\end{aligned}
\end{equation*}
\end{definition}
\begin{theorem}
{\textbf{Differential map of $\mf{a}(\ze)$}.}
\label{09142108}
Assume that $\Gamma_{\lr{\mc{Y},2}{\mc{Z}}}\circ\mh{L}^{(2)}$ is bounded, then $\mf{a}(\ze)$ is differentiable with differential map 
such that 
\begin{equation*}
\mf{a}(\ze)^{(1)}=\mh{Q}.
\end{equation*}
\end{theorem}
\begin{proof}
Let $u\in\mf{Z}^{O}$ and $f\in\mc{H}(W)$, then since Lemma \ref{06121742} we have 
\begin{equation*}
\mf{a}(\ze)(u)(f)=\int(\xi^{-1})^{\ast}(f)\cdot\left(\mh{L}\circ(u^{\xi}\boxtimes\partial u^{\xi})\right)\,d\lambda_{U}.
\end{equation*}
Next let $r>0$ be such that $u+B_{r}^{\mf{Z}}(\ze)\subseteq\mf{Z}^{O}$ which exists since $\mf{Z}^{O}$ is open in $\mf{Z}$.
Let $v\in B_{r}^{\mf{Z}}(\ze)$, then 
\begin{multline}
\label{09142037}
\mf{a}(\ze)(u+v)(f)-\mf{a}(\ze)(u)(f)-\mh{Q}[u,v](f)
=
\\
\int(\xi^{-1})^{\ast}(f)\cdot
\bigl[
\mh{L}\circ((u^{\xi}+v^{\xi})\boxtimes(\partial u^{\xi}+\partial v^{\xi}))
\\
-\mh{L}\circ(u^{\xi}\boxtimes\partial u^{\xi})
-\mf{n}\circ\left\{(\mh{L}^{(1)}\circ(u^{\xi}\boxtimes\partial u^{\xi}))\boxtimes(v^{\xi}\boxtimes\partial v^{\xi})\right\}
\bigr]
\,d\lambda_{U}.
\end{multline}
Let $p\in\Gamma$. By \eqref{09142037}, since the general result \cite[VI.5, Prp.6]{IntBourb}, 
since the Taylor expansion Thm.\ref{8143}, and finally since the hypothesis of boundedness we obtain 
\begin{equation*}
p\bigl(\mf{a}(\ze)(u+v)(f)-\mf{a}(\ze)(u)(f)-\mh{Q}[u,v](f)\bigr)\leq C(p)C(f)\widehat{q}(v)^{2},
\end{equation*}
where 
\begin{equation*}
\begin{aligned}
C(p)
&\coloneqq
\sup_{y\in\mc{Y}}\overset{2}{p}\left((\Gamma_{\lr{\mc{Y},2}{\mc{Z}}}\circ\mh{L}^{(2)})(y)\right)
\\
C(f)
&\coloneqq
\int|(\xi^{-1})^{\ast}(f)|\,d\lambda_{U};
\end{aligned}
\end{equation*}
therefore 
\begin{equation}
\label{09251745a}
\lim_{w\to\ze,\,w\in\complement_{B_{r}^{\mf{Z}}(\ze)}\{\ze\}}
\frac{p\bigl(\mf{a}(\ze)(u+w)(f)-\mf{a}(\ze)(u)(f)-\mh{Q}[u,w](f)\bigr)}{\widehat{q}(w)}=0.
\end{equation}
The statement then follows since \eqref{09251745a}, Prp.\ref{09251745b}, and Rmk.\ref{09152034}.
\end{proof}
\begin{definition}
\label{11011831}
Call \emph{special set of the solutions of the Euler-Lagrange equation} the set $\mr{sEL}_{W}(\mh{L})$ defined as 
the subset of those $u\in\mf{Z}^{O}$ satisfying
\begin{multline*}
(\forall v\in\mf{Z})
\bigl\{
\mf{g}_{0}
\circ
\left(
\bigl(
\frac{\partial\mh{L}}{\partial Y_{0}}\circ(u^{\xi}\boxtimes(\partial u^{\xi}))
\bigr)
\boxtimes
v^{\xi}
\right)
-
\mf{g}_{i}\circ\left(
\sum_{i=1}^{n}
\frac{\partial}{\partial e_{i}[n]}
\left[
\frac{\partial\mh{L}_{\mf{e}}}{\partial Y_{i}}\circ(\mf{d}_{O})_{\natural}\circ(u^{\xi}\boxtimes(\partial u^{\xi}))
\right]
\boxtimes
v^{\xi}
\right)
\\
=
(-1)
\sum_{i=1}^{n}
\frac{\partial}{\partial e_{i}[n]}\left[
\mf{g}_{i}\circ\left(  
\bigl(
\frac{\partial\mh{L}_{\mf{e}}}{\partial Y_{i}}\circ(\mf{d}_{O})_{\natural}\circ(u^{\xi}\boxtimes(\partial u^{\xi}))
\bigr)
\boxtimes 
v^{\xi}
\right)
\right]
\bigr\}.
\end{multline*}
\end{definition}
\begin{definition}
Let $Q\in\mr{Cmp}(W)$, define 
\begin{equation*}
\mf{Z}[Q]\coloneqq\{g\in\mf{Z}\,\vert\,(\forall(V,\rho)\in\mr{Crt}(W))
(\forall i\in[0,n]\cap\Z)(\mr{supp}(\partial_{i}^{\rho}g)\subseteq Q\}.
\end{equation*}
\end{definition}
\begin{lemma}
\label{10302024}
If $Q\in\mr{Cmp}(W)$, then $\mf{Z}[Q]$ is a $\mc{C}^{\infty}(W)$-module. 
\end{lemma}
\begin{proof}
Let $Q\in\mr{Cmp}(W)$, $g\in\mf{Z}[Q]$ and $f\in\mc{C}^{\infty}(W)$, clearly $fg\in\mc{C}^{\infty}(W,Y_{0})$ moreover 
$\mr{supp}(fg)\subseteq Q$ since Cor.\ref{09292048}, next for any $(V,\rho)\in\mr{Crt}(W)$ we have 
$\partial_{\nu}^{\rho}(fg)=f\partial_{\nu}^{\rho}(g)+g\partial_{\nu}^{\rho}(f)$ 
for any $\nu\in[1,n]\cap\Z$, then $\mr{supp}(\partial_{\nu}^{\rho}(fg))\subseteq Q$ since Prp. \ref{12031719}
and Cor.\ref{09292048} therefore, $fg\in\mf{Z}[Q]$ and the statement follows.
\end{proof}
Now we can establish the following key result that justifies the denomination provided in Def.\ref{11011831}.
\begin{theorem}
\label{10301959}
$\mr{sEL}_{W}(\mh{L})\subseteq\mr{EL}_{W}(\mh{L})$. 
\end{theorem}
\begin{proof}
Let $w\in\mf{Z}_{O}$ and $s\in\mf{Z}$. Only in the present proof we let 
\begin{equation*}
\Delta[w;s]
\coloneqq
\mf{g}_{0}
\circ
\left(
\bigl(
\frac{\partial\mh{L}}{\partial Y_{0}}\circ(w^{\xi}\boxtimes(\partial w^{\xi}))
\bigr)
\boxtimes
s^{\xi}
\right)
-
\mf{g}_{i}\circ\left(
\sum_{i=1}^{n}
\frac{\partial}{\partial e_{i}[n]}
\left[
\frac{\partial\mh{L}_{\mf{e}}}{\partial Y_{i}}\circ(\mf{d}_{O})_{\natural}\circ(w^{\xi}\boxtimes(\partial w^{\xi}))
\right]
\boxtimes
s^{\xi}
\right),
\end{equation*}
and 
\begin{equation*}
B[w;s]
\coloneqq
\sum_{i=1}^{n}
\frac{\partial}{\partial e_{i}[n]}\left[
\mf{g}_{i}\circ\left(  
\bigl(
\frac{\partial\mh{L}_{\mf{e}}}{\partial Y_{i}}\circ(\mf{d}_{O})_{\natural}\circ(w^{\xi}\boxtimes(\partial w^{\xi}))
\bigr)
\boxtimes 
s^{\xi}
\right)
\right].
\end{equation*}
Next let $K_{0},K\in\mr{Cmp}(U)$ and $V\in\mr{Op}(U)$ such that 
\begin{equation}
\label{11011820}
K_{0}\subset V\subset K,
\end{equation}
set $Q=\xi^{-1}(K)\in\mr{Cmp}(W)$ and let $v\in\mf{Z}[Q]$.
Thus, $v^{\xi}\in\mc{H}(U,K;Y_{0})$ since Cor.\ref{08111037} and then $v^{\xi}\circ\imath_{\partial K}^{U}=\ze$ by the principle of 
extension of identities therefore, since \cite[Cor.3.60 and Thm.3.47]{sil09261621} we obtain
\begin{equation}
\label{10311817} 
\int(B[w;v]\circ\imath_{K}^{U})\,d\lambda_{K}=\ze.
\end{equation}
Now let $u\in\mr{sEL}_{W}(\mh{L})$, then \eqref{10311817} implies 
\begin{equation}
\label{11011029}
\int(\Delta[u;v]\circ\imath_{K}^{U})\,d\lambda_{K}=\ze,
\end{equation}
but $\mf{Z}[Q]$ is a $\mc{C}^{\infty}(W)$-module since Lemma \ref{10302024}, while $\Delta[w;fs]=f^{\xi}\cdot\Delta[w;s]$ 
for every $f\in\mc{C}^{\infty}(W)$, where $f^{\xi}=f\circ\xi^{-1}$, then by \eqref{11011029} we obtain 
\begin{equation*}
(\forall f\in\mc{C}^{\infty}(W))
(\int[(f^{\xi}\cdot\Delta[u;v])\circ\imath_{K}^{U}]\,d\lambda_{K}=\ze),
\end{equation*}
in particular 
\begin{equation}
\label{11011749}
(\forall f\in\mc{C}^{\infty}(W))
(\int[(f^{\xi}\cdot\Delta[u;v])\circ\imath_{V}^{U}]\,d\lambda_{V}=\ze).
\end{equation}
Now $(\xi^{-1})^{\ast}\mc{C}^{\infty}(W)=\mc{C}^{\infty}(U)$, while $(\imath_{V}^{U})^{\ast}\mc{C}^{\infty}(U)=\mc{C}^{\infty}(V)$ indeed any 
smooth map on $V$ can be extended to a smooth map on $U$ by the extension lemma for smooth maps, see for instance 
\cite[Lemma 2.26]{lee} therefore, since the fundamental lemma of calculus of variations \cite[p.16-17]{gh} we deduce by 
\eqref{11011749} that 
\begin{equation}
\label{10311726}
\Delta[u;v]\circ\imath_{V}^{U}=\ze.
\end{equation}
Next let $\{O_{i}\}_{i=1}^{m}$ be a finite cover of $K_{0}$ made of open sets of $U$ such that $O_{i}\subset K$ for every 
$i\in[1,m]\cap\Z$. Then, $\{P_{i}=\xi^{-1}(O_{i})\}_{i=1}^{m}$ is a finite cover of 
$Q_{0}=\xi^{-1}(K_{0})$ compact of $W$ made of open sets of $W$ and such that $P_{i}\subset Q$ for every $i\in[1,m]\cap\Z$.
Now for any smooth manifold say $N$ there exists a partition of unity over $N$ made of smooth functions \cite[16.4.1]{die3} each 
one of them has the same support of its derivative with respect to any chart of $N$ \cite[16.4.1.4]{die3}. Therefore, in virtue 
of \cite[16.4.2]{die3} applied for $N=W$ there exists a real valued smooth map on $W$ say $h_{K_{0}}$ with the following 
properties:
\begin{itemize}
\item
$h_{K_{0}}\circ\imath_{Q_{0}}^{W}$ is the constant map equal to $1$;
\item
$\mr{supp}(h_{K_{0}})=\mr{supp}(\partial_{\nu}^{\rho}h_{K_{0}})\subseteq\bigcup_{i=1}^{m}P_{i}\subset Q$, for any 
$(\mr{Dom}(\rho),\rho)\in\mr{Crt}(W)$ and any $\nu\in[1,n]\cap\Z$.
\end{itemize}
Thus, for every $y\in Y_{0}$ the map $h_{K_{0}}^{y}:W\to Y_{0}$, $x\mapsto h_{K_{0}}(x)y$ is such that 
\begin{itemize}
\item
$h_{K_{0}}^{y}\circ\imath_{Q_{0}}^{W}$ is the constant map equal to $y$, equivalently 
$(h_{K_{0}}^{y})^{\xi}\circ\imath_{K_{0}}^{U}$ is the constant map equal to $y$;
\item
$h_{K_{0}}^{y}\in\mf{Z}[Q]$.
\end{itemize}
Now the facts that $h_{K_{0}}^{y}\in\mf{Z}[Q]$, $K_{0}\subset V$, and \eqref{10311726} imply
\begin{equation}
\label{10311742}
\Delta[u;h_{K_{0}}^{y}]\circ\imath_{K_{0}}^{U}=\ze.
\end{equation}
Next let $r:U\to\R_{>0}$ be such that $(\forall x\in U)(\ov{B}_{r(x)}(x)\subset U)$ which exists since $U$ is open in $\R^{n}$, 
then for any $x\in U$ we have $\ov{B}_{r(x)}(x)\in\mr{Cmp}(U)$, $\ov{B}_{3^{-1}r(x)}(x)\in\mr{Cmp}(U)$, $B_{2^{-1}r(x)}(x)\in\mr{Op}(U)$
and 
\begin{equation*}
\ov{B}_{3^{-1}r(x)}(x)\subset B_{2^{-1}r(x)}(x)\subset\ov{B}_{r(x)}(x), 
\end{equation*}
namely the position \eqref{11011820}, then we can apply \eqref{10311742} to conclude that 
\begin{equation}
\label{11011351}
(\forall x\in W)(\Delta[u;h_{\ov{B}_{3^{-1}r(x)}(x)}^{y}]\circ\imath_{\ov{B}_{3^{-1}r(x)}(x)}^{U}=\ze),
\end{equation}
but $(h_{\ov{B}_{3^{-1}r(x)}(x)}^{y})^{\xi}\circ\imath_{\ov{B}_{3^{-1}r(x)}(x)}^{U}$ is the constant map equal to $y$ so, by \eqref{11011351} we 
deduce that $u\in\mr{EL}_{W}(\mh{L})$ and the inclusion $\mr{sEL}_{W}(\mh{L})\subseteq\mr{EL}_{W}(\mh{L})$ follows. 
\end{proof}
The following result establishes a characterization of $\mr{vp}(\mf{a}(\ze))$ in terms of $\mr{sEL}_{W}(\mh{L})$.
\begin{corollary}
[\textbf{The set of solutions of the variational problem associated with $\mf{a}(\ze)$ equals $\mr{sEL}_{W}(\mh{L})$.}]
\label{10011525}
Assume that $\Gamma_{\lr{\mc{Y},2}{\mc{Z}}}\circ\mh{L}^{(2)}$ is bounded, then 
\begin{equation*}
\mr{vp}(\mf{a}(\ze))=\mr{sEL}_{W}(\mh{L}).
\end{equation*}
\end{corollary}
\begin{proof}
Let $u\in\mf{Z}^{O}$ and $w\in\mf{Z}$. Since Lemma \ref{10181658} we have 
\begin{multline}
\label{10011519}
\mf{n}
\circ
\left(
\bigl(\mh{L}^{(1)}\circ(u^{\xi}\boxtimes(\partial u^{\xi}))\bigr)
\boxtimes
(w^{\xi}\boxtimes(\partial w^{\xi}))
\right)
=
\mf{g}_{0}
\circ
\left(
\bigl(
\frac{\partial\mh{L}}{\partial Y_{0}}\circ(u^{\xi}\boxtimes(\partial u^{\xi}))
\bigr)
\boxtimes
w^{\xi}
\right)
\\
+
\sum_{i=1}^{n}
\mf{g}_{i}
\circ
\left(
\bigl(\frac{\partial\mh{L}_{\mf{e}}}{\partial Y_{i}}\circ(\mf{d}_{O})_{\natural}\circ(u^{\xi}\boxtimes(\partial u^{\xi}))\bigr)
\boxtimes
(\partial_{i}w^{\xi})
\right).
\end{multline}
Next since Prp.\ref{03241344} and Prp.\ref{08091051} we deduce that 
\begin{multline}
\label{10011509}
\frac{\partial}{\partial e_{i}[n]}\left[
\mf{g}_{i}\circ\left(  
\bigl(
\frac{\partial\mh{L}_{\mf{e}}}{\partial Y_{i}}\circ(\mf{d}_{O})_{\natural}\circ(u^{\xi}\boxtimes(\partial u^{\xi}))
\bigr)
\boxtimes 
w^{\xi}
\right)
\right]
=
\\
\mf{g}_{i}\circ\left(
\frac{\partial}{\partial e_{i}[n]}
\left[
\frac{\partial\mh{L}_{\mf{e}}}{\partial Y_{i}}\circ(\mf{d}_{O})_{\natural}\circ(u^{\xi}\boxtimes(\partial u^{\xi}))
\right]
\boxtimes
w^{\xi}
\right)
+
\\
\mf{g}_{i}\circ\left(
\bigl(\frac{\partial\mh{L}_{\mf{e}}}{\partial Y_{i}}\circ(\mf{d}_{O})_{\natural}\circ(u^{\xi}\boxtimes(\partial u^{\xi}))\bigr)
\boxtimes
(\partial_{i}w^{\xi})
\right).
\end{multline}
Now since Thm. \ref{09142108}, \eqref{10011519} and \eqref{10011509} we obtain
\begin{multline}
\label{10011531}
\mf{a}(\ze)^{(1)}(u)=\ze
\Leftrightarrow
(\forall v\in\mf{Z})(\forall f\in\mc{H}(W))
\\
\bigl(
\int(\xi^{-1})^{\ast}(f)\cdot 
\bigl\{
\mf{g}_{0}
\circ
\left(
\bigl(
\frac{\partial\mh{L}}{\partial Y_{0}}\circ(u^{\xi}\boxtimes(\partial u^{\xi}))
\bigr)
\boxtimes
v^{\xi}
\right)
-
\sum_{i=1}^{n}
\mf{g}_{i}\circ\left(
\frac{\partial}{\partial e_{i}[n]}
\left[
\frac{\partial\mh{L}_{\mf{e}}}{\partial Y_{i}}\circ(\mf{d}_{O})_{\natural}\circ(u^{\xi}\boxtimes(\partial u^{\xi}))
\right]
\boxtimes
v^{\xi}
\right)
\\
+
\sum_{i=1}^{n}\frac{\partial}{\partial e_{i}[n]}\left[
\mf{g}_{i}\circ\left(  
\bigl(
\frac{\partial\mh{L}_{\mf{e}}}{\partial Y_{i}}\circ(\mf{d}_{O})_{\natural}\circ(u^{\xi}\boxtimes(\partial u^{\xi}))
\bigr)
\boxtimes 
v^{\xi}
\right)
\right]\,d\lambda_{U}
\bigr\}
=\ze
\bigr).
\end{multline}
The statement then follows by \eqref{10011531}, since Rmk.\ref{10031813} and by the definition of weak integral.
\end{proof}
\begin{corollary}
[\textbf{An invariance of the action $\mf{a}$ induces a symmetry of $\mr{sEL}_{W}(\mh{L})$}]
\label{11031537}
If $\Gamma_{\lr{\mc{Y},2}{\mc{Z}}}\circ\mh{L}^{(2)}$ is bounded and $\lr{\mh{L},\phi}{\omega}$ is perfect at $\eta$, then 
\begin{equation*}
\mf{a}(\eta)=\mf{a}(\ze)\Rightarrow\phi_{\eta}(\mr{sEL}_{W}(\mh{L}))\subseteq\mr{sEL}_{W}(\mh{L}).
\end{equation*}
\end{corollary}
\begin{proof}
Since Cor.\ref{10011525}, Thm.\ref{09142108} and Thm.\ref{08280333}.
\end{proof}
Finally all the main results of the paper yield the following 
\begin{corollary}
\label{10042025}
Assume that $\mf{a}$ is constant. Then
\begin{equation*}
(\forall u\in\mr{EL}_{W}(\mh{L}))
\left(\sum_{i=1}^{n}\frac{\partial\mh{N}^{i}[u]}{\partial e_{i}[n]}=\ze\right),
\end{equation*}
in particular 
\begin{equation*}
(\forall u\in\mr{sEL}_{W}(\mh{L}))
\left(\sum_{i=1}^{n}\frac{\partial\mh{N}^{i}[u]}{\partial e_{i}[n]}=\ze\right).
\end{equation*}
Furthermore under the hypothesis and notation of Cor.\ref{05112017} and if $p=+\infty$, then we have 
\begin{equation*}
(\forall u\in\mr{EL}_{W}(\mh{L}))
\left(\frac{dQ_{V}(u)}{dt}=\ze\right),
\end{equation*}
in particular
\begin{equation*}
(\forall u\in\mr{sEL}_{W}(\mh{L}))
\left(\frac{dQ_{V}(u)}{dt}=\ze\right).
\end{equation*}
Finally if $\Gamma_{\lr{\mc{Y},2}{\mc{Z}}}\circ\mh{L}^{(2)}$ is bounded and $\lr{\mh{L},\phi}{\omega}$ is perfect at $\eta$, then 
\begin{equation*}
\phi_{\eta}(\mr{sEL}_{W}(\mh{L}))\subseteq\mr{sEL}_{W}(\mh{L}).
\end{equation*}
\end{corollary}
\begin{proof}
Since Cor.\ref{08281203}, Thm.\ref{10301959} and Cor.\ref{11031537}.
\end{proof}
\section{Examples}
\label{08300401}
In the present section we specialize to the case $\K=\K_{1}=\C$, recall \eqref{12201914} and the $\C$-action provided in 
\eqref{12152125}.
\begin{notations}
In the present section we add to Notation \ref{12152204} the following notation.
If $Z$ is a topological linear space on $\C$, we let $\ov{Z}$ be the topological linear space on $\C$
whose underlying linear space is the conjugate space of $Z$ and whose underlying topology is that of $Z$. We define 
$Z^{\dagger}\coloneqq\mc{L}(Z,\ov{\C})$ and if $Z$ is locally convex, then we define $Z_{b}^{\dagger}\coloneqq\mc{L}_{b}(Z,\ov{\C})$ 
where we recall that $_{b}$ means the topology of uniform convergence over bounded sets.
If $E$ is a topological linear space on $\mathbb{A}\in\{\R,\C\}$, then we let $\mathbb{GL}(E)$ be the group of 
$\mathbb{A}$-linear continuous isomorphism defined in and valued on $E$ where map composition is the law of composition and 
where $\mr{Id}_{E}$ is the neutral element.
We shall use Einstein convention on sums namely repeated indexes are assumed summed. If $B:\R^{n}\to\R^{n}$ is linear, then
we let $B_{a,b}\coloneqq(\Pr_{a}\circ B)(e_{b}[n])$ for every $a,b\in[1,n]\cap\Z$. If $(V,\theta)$ and $(V,\rho)$ are charts of $M$,
then we let $\Lambda(\theta,\rho)\coloneqq\rho\circ\theta^{-1}$. If $\Lambda(\theta,\rho)$ is the restriction of a linear map on
$\R^{n}$ still denoted by the same symbol, then we know that 
\begin{equation}
\label{08262239}
\partial_{\mu}^{\theta}=\Lambda(\theta,\rho)_{\nu,\mu}\partial_{\nu}^{\rho}. 
\end{equation}
\end{notations}
\begin{definition}
\label{10172011}
For any topological complex linear space $Q$, any $\rho\in\mc{L}(X,Q)$ and any smooth manifold $N$ define 
$\rho_{\intercal}:\mc{C}_{c}^{\infty}(N,X)\to\mc{C}_{c}^{\infty}(N,Q)$, $u\mapsto\rho\circ u$.
\end{definition}
\begin{lemma}
\label{08300643}
For any topological complex linear space $Q$, any $\rho\in\mc{L}(X,Q)$ injective and any smooth manifold $N$ the map 
$\rho_{\intercal}$ is a well-defined linear map. If $Q=X$, then $\rho_{\intercal}\in\mc{L}(\mf{Z})$ and 
$\|\rho_{\intercal}\|\leq\|\rho\|$. 
\end{lemma}
\begin{proof}
If $u\in\mc{C}^{\infty}(N,X)$, then $\rho\circ u\in\mc{C}^{\infty}(N,Q)$ since $\rho$ is smooth being linear and continuous;
next $u\mapsto\rho\circ u$ is a linear map from $\mc{C}^{\infty}(N,X)$ to $\mc{C}^{\infty}(N,Q)$ since $\rho$ is linear;  
finally 
\begin{equation*}
\overset{-1}{\overbrace{\rho\circ u}}(\{\ze\})=(\overset{-1}{u}\circ\overset{-1}{\rho}(\{\ze\}))=\overset{-1}{u}(\{\ze\}),
\end{equation*}
the 2th equality following by $\overset{-1}{\rho}(\{\ze\})=\{\ze\}$ since $\rho$ is linear and injective thus, 
$\mr{supp}(\rho\circ u)=\mr{supp}(u)$ which ends the proof of well-definedness of $\rho_{\intercal}$. Next let $u\in\mf{Z}$ thus,
\begin{equation*}
\partial_{\nu}^{\xi}(\rho_{\intercal}(u))
=
\partial_{\nu}(\rho\circ u\circ\xi^{-1})\circ\xi
=
\rho\circ\partial_{\nu}(u\circ\xi^{-1})\circ\xi
=
\rho\circ\partial_{\nu}^{\xi}(u);
\end{equation*}
therefore
\begin{equation*}
\widehat{q}(\rho_{\intercal}(u))
=
\max_{\nu\in[0,n]\cap\Z}\sup_{x\in W}\|(\rho\circ\partial_{\nu}^{\xi}(u))(x)\|_{X},
\leq\|\rho\|\widehat{q}(u),
\end{equation*}
and the statement follows.
\end{proof}
\begin{definition}
Define $\mf{Z}_{2}$ to be the $\C$-normed space whose underlying linear space is $\mc{C}_{c,1}^{\infty}(W,Y_{0})$ and whose norm is 
\begin{equation*}
\widehat{q}_{2}(g)\coloneqq\sup_{x\in W}\|g(x)\|_{X}.
\end{equation*}
\end{definition}
\begin{lemma}
\label{08310357}
Let $\theta\in\mc{C}^{\infty}(W,W)$ be such that $\exists\,\theta^{-1}\in\mc{C}(W,W)$. Then the map 
$\theta^{\ast}:\mc{C}_{c}^{\infty}(W,X)\to\mc{C}_{c}^{\infty}(W,X)$, $f\mapsto f\circ\theta$ is such that
$\theta^{\ast}\in\mc{L}(\mf{Z},\mf{Z}_{2})$ and $\|\theta^{\ast}\|\leq 1$.
Furthermore let $(\xi\circ\theta)_{\mu}\coloneqq\Pr_{\mu}\circ\imath_{U}^{\R^{n}}\circ\xi\circ\theta$ for any $\mu\in[1,n]\cap\Z$. 
Thus if we assume that 
\begin{equation}
\label{08310357hy}
(\forall\nu\in[0,n]\cap\Z)(\forall\mu\in[1,n]\cap\Z)(\sup_{x\in W}\left|\partial_{\nu}^{\xi}(\xi\circ\theta)_{\mu}(x)\right|<\infty),
\end{equation}
then 
\begin{equation*}
\theta^{\ast}\in\mc{L}(\mf{Z},\mf{Z}),\,\|\theta^{\ast}\|\leq nC(\theta),
\end{equation*}
where
\begin{equation*}
C(\theta)\coloneqq
\sup\left\{
1,
\max_{\nu\in[1,n]\cap\Z}
\max_{\mu\in[1,n]\cap\Z}
\sup_{x\in W}\left|\partial_{\nu}^{\xi}(\xi\circ\theta)_{\mu}(x)\right|.
\right\}
\end{equation*}
\end{lemma}
\begin{proof}
$\theta^{\ast}:\mc{C}^{\infty}(W,X)\to\mc{C}^{\infty}(W,X)$ since $\theta$ is smooth moreover $\theta^{\ast}$ is trivially linear. 
Now by Prp.\ref{12061416}, since $\overset{-1}{\theta}=\theta^{-1}$ and the fact that $\theta^{-1}$ being continous maps
the set of compacts into the set of compacts, we deduce that $\mr{supp}(\theta^{\ast}(u))=\mr{supp}(u)$ so, for what above proven
$\theta^{\ast}:\mc{C}_{c}^{\infty}(W,X)\to\mc{C}_{c}^{\infty}(W,X)$. Next let $u\in\mf{Z}$ then,
\begin{equation*}
\widehat{q}_{2}(\theta^{\ast}(u))
=
\sup_{x\in W}\|u(\theta(x))\|_{X}
=
\sup_{x\in W}\|u(x)\|_{X}
\leq
\widehat{q}(u);
\end{equation*}
the 2th equality following since $\theta$ is surjective and the first part of the statement follows.
Next let $\nu\in[1,n]\cap\Z$ and $u\in\mf{Z}$, let $\lambda\in\mc{C}^{\infty}(U,U)$ defined by 
$\lambda\coloneqq\xi\circ\theta\circ\xi^{-1}$ and $\lambda_{\mu}\coloneqq\Pr_{\mu}\circ\lambda$; thus for every $y\in U$ we have
\begin{equation*}
\begin{aligned}
\partial_{\nu}(u^{\xi}\circ\lambda)(y)
&=
(u^{\xi}\circ\lambda)^{1}(e_{\nu}[n])(y)
\\
&=
(u^{\xi})^{(1)}(\lambda(x))(\lambda^{1}(y)(e_{\nu}[n]))
\\
&=
(u^{\xi})^{(1)}(\lambda(y))(\sum_{\mu=1}^{n}\imath_{\mu}(\lambda_{\mu}^{1}(y)(e_{\nu}[n])))
\\
&=
(u^{\xi})^{(1)}(\lambda(y))(\sum_{\mu=1}^{n}e_{\mu}[n]\lambda_{\mu}^{1}(y)(e_{\nu}[n]))
\\
&=
\sum_{\mu=1}^{n}
\lambda_{\mu}^{1}(y)(e_{\nu}[n])
(u^{\xi})^{(1)}(\lambda(y))(e_{\mu}[n])
\\
&=
\sum_{\mu=1}^{n}\partial_{\nu}(\lambda_{\mu})(y)(\partial_{\mu}u^{\xi})(\lambda(y))
\end{aligned}
\end{equation*}
where the the 2th equality follows by Thm.\ref{821}, the 3th one by Prp.\ref{815} so,
\begin{equation*}
\partial_{\nu}(u^{\xi}\circ\lambda)
=
\sum_{\mu=1}^{n}\partial_{\nu}(\lambda_{\mu})\cdot[(\partial_{\mu}u^{\xi})\circ\lambda].
\end{equation*}
Now
\begin{equation*}
\partial_{\nu}^{\xi}(u\circ\theta)
=
\partial_{\nu}(u\circ\theta\circ\xi^{-1})\circ\xi
=
\partial_{\nu}(u^{\xi}\circ\lambda)\circ\xi
\end{equation*}
therefore
\begin{equation*}
\begin{aligned}
\partial_{\nu}^{\xi}(u\circ\theta)
&=
\sum_{\mu=1}^{n}
[\partial_{\nu}((\xi\circ\theta)_{\mu}\circ\xi^{-1})\circ\xi]
\cdot
[(\partial_{\mu}u^{\xi})\circ\lambda\circ\xi]
\\
&=
\sum_{\mu=1}^{n}
\partial_{\nu}^{\xi}((\xi\circ\theta)_{\mu})
\cdot
[(\partial_{\mu}u\circ\xi^{-1})\circ\xi\circ\theta]
\\
&=
\sum_{\mu=1}^{n}
\partial_{\nu}^{\xi}((\xi\circ\theta)_{\mu})
\cdot
[(\partial_{\mu}^{\xi}u)\circ\theta].
\end{aligned}
\end{equation*}
Thus 
\begin{equation*}
\begin{aligned}
\sup_{x\in W}
\|\partial_{\nu}^{\xi}(u\circ\theta)(x)\|_{X}
&
\leq
\sum_{\mu=1}^{n}
\sup_{x\in W}\|\partial_{\nu}^{\xi}((\xi\circ\theta)_{\mu})(x)\|
\cdot
\sup_{x\in W}\|(\partial_{\mu}^{\xi}u)(\theta(x))\|
\\
&
=
\sum_{\mu=1}^{n}
\sup_{x\in W}\|\partial_{\nu}^{\xi}((\xi\circ\theta)_{\mu})(x)\|
\cdot
\sup_{x\in W}\|(\partial_{\mu}^{\xi}u)(x)\|
\\
\leq
nC(\theta)
\widehat{q}(u),
\end{aligned}
\end{equation*}
and 
\begin{equation*}
\sup_{x\in W}\|(u\circ\theta)(x)\|_{X}
=
\sup_{x\in W}\|u(x)\|_{X}
\leq
\widehat{q}(u)
\leq
nC(\theta)
\widehat{q}(u),
\end{equation*}
where the equalities follow since $\theta$ is surjective.
Then the statement follows.
\end{proof}
\begin{definition}
\label{08262042}
Let $Z$ and $E$ be complex topological linear spaces, we call $(\overline{\cdot},\widehat{\cdot})$ a conjugation system on $Z$ 
and valued in $E$ iff 
\begin{enumerate}
\item
$\overline{\cdot}\in\mc{L}(Z,\mc{L}_{b}(Z,\ov{E}))$, 
\item
$\widehat{\cdot}\in\mc{L}(\mc{L}_{b}(Z),\ov{\mc{L}_{b}(Z)})$, 
\item
$(\forall z\in Z)(\forall P\in\mc{L}(Z))(\ov{Pz}=\ov{z}\circ\widehat{P})$.
\label{08262042c}
\end{enumerate}
\end{definition}
\begin{example}
Let $X=\mathbb{M}_{n,1}(\C)$ provided by the $\max$-norm, let $\mc{L}(X)$ be identified with $\mathbb{M}_{n}(\C)$ and let 
$\mc{L}(X,\C)$ be identified with $\mathbb{M}_{1,n}(\C)$. Let $\gamma_{0}\in\mc{L}(X)$ be such that $\gamma_{0}^{2}=\un_{n}$, let
$^{\intercal}$ be the transpose operation, $^{\ast}$ the complex conjugation and finally let $^{\dagger}$ be the composition of the 
previous two operations. Thus we can define $\ov{x}\coloneqq x^{\dagger}\circ\gamma_{0}$ and 
$\widehat{T}\coloneqq\gamma_{0}T^{\dagger}\gamma_{0}$; clearly $(\overline{\cdot},\widehat{\cdot})$ is a conjugation system on $X$
valued in $\C$.
\end{example}
\begin{assumption}
For the remaining of the present section \ref{08300401}, we let $(\overline{\cdot},\widehat{\cdot})$ be a conjugation system 
on $X$ and valued in $\mc{Z}$. Furthermore we assume that for every $\eta\in A$ the map $\chi(\eta)$ defined in 
Def.\ref{10051149} is the restricition of a unique element of $\mathbb{GL}(\R^{n})$ which by abuse of notation we denote with the
same symbol. 
\end{assumption}
\begin{definition}
\label{08262218}
We call $(R,\gamma)$ a Dirac system on $X$ of dimension $n$ iff 
\begin{enumerate}
\item
$\gamma:[1,n]\cap\Z\to\mc{L}(X)$,
\item
$R\in\mr{Mor_{Gr}}(\mathbb{GL}(\R^{n}),\mathbb{GL}(X))$,
\item
$(\forall P\in\mathbb{GL}(\R^{n}))(\forall\nu\in[1,n\cap\Z])$ we have 
\begin{equation*}
R(P^{-1})\circ\gamma^{\nu}\circ R(P)=P_{\nu,\mu}\gamma^{\mu}.
\end{equation*} 
\item
$(\forall P\in\mathbb{GL}(\R^{n}))(\widehat{R}(P)\circ R(P)=\mr{Id}_{X})$;
\label{08262218a}
\end{enumerate}
where $\widehat{R}(P)\coloneqq\widehat{R(P)}$ for every $P\in\mathbb{GL}(\R^{n})$.
\end{definition}
\begin{definition}
\label{08251005}
Let $(R,\gamma)$ be a Dirac system on $X$ of dimension $n$, define
\begin{equation*}
\begin{cases}
S:A\to\mathbb{GL}(X),
\\
S\coloneqq R\circ\chi.
\end{cases}
\end{equation*}
For every $\eta\in A$ define $\widehat{S}(\eta)\coloneqq\widehat{S(\eta)}$.
\end{definition}
Notice that $\chi(\eta)=\Lambda(\omega_{\natural}^{\eta},\xi)$ moreover 
\begin{equation}
\label{08251046}
S(\eta)^{-1}\circ\gamma^{\nu}\circ S(\eta)=\chi_{\nu,\mu}(\eta)\gamma^{\mu}.
\end{equation} 
\begin{convention}
If $u\in\mc{C}^{\infty}(W,Y_{0})$, then we let $u^{\xi}=u\circ\xi^{-1}$ and for any set $Q$ and any map $f:Q\to X$ we 
let $\ov{f}$ denote $\ov{\cdot}\circ f$.
\end{convention}
\begin{lemma}
\label{08251057}
If $f:U\to X$ is differentiable and $\Lambda$ is the restriction of a unique map belonging to $\mathbb{GL}(\R^{n})$, then
$\partial_{\nu}(f\circ\Lambda)=\Lambda_{\mu,\nu}(\partial_{\mu}f)\circ\Lambda$.
\end{lemma}
\begin{proof}
\begin{equation*}
\begin{aligned}
\partial_{\nu}(f\circ\Lambda)(x)
&=
(f\circ\Lambda)^{(1)}(x)(e_{\nu}[n])
\\
&=
f^{(1)}(\Lambda(x))(\Lambda^{(1)}(x)(e_{\nu}[n]))
\\
&=
f^{(1)}(\Lambda(x))(\Lambda(e_{\nu}[n]))
\\
&=
f^{(1)}(\Lambda(x))(e_{\mu}[n]\Pr_{\mu}(\Lambda(e_{\nu}[n])))
\\
&=
\Lambda_{\mu,\nu}f^{(1)}(\Lambda(x))(e_{\mu}[n])
\\
&=
\Lambda_{\mu,\nu}(\partial_{\mu}f)(\Lambda(x));
\end{aligned}
\end{equation*}
where the 1th and 6th equality follow since Prp.\ref{08091051}; the 2th one by Thm.\ref{821}; the 3th one since 
Cor.\ref{03110819}.
\end{proof}
\begin{definition}
Define $\mf{x}:\mc{L}_{b}(X,\ov{\mc{Z}})\times X\to\ov{\mc{Z}}$, $(B,x)\mapsto B(x)$.
\end{definition}
Notice that $\mf{x}$ is bilinear and continuous therefore, 
$\mf{x}\circ(\ov{(\cdot)}\boxtimes\mr{Id}_{X}):X\times X\to\ov{\mc{Z}}$ is bilinear and continuous.
\begin{assumption}
\label{09081619pre}
For the remaining of the present section \ref{08300401}, we let $(R,\gamma)$ be a Dirac system on $X$ of dimension $n$. 
\end{assumption}
\begin{definition}
Define for every $u\in\mc{C}^{\infty}(W,Y_{0})$ and $\eta\in A$ the maps $u_{\eta}^{\xi}:U\to X$ and $u_{\eta}:W\to X$ such that 
\begin{equation*}
\begin{aligned}
u_{\eta}^{\xi}&\coloneqq S(\eta)\circ u^{\xi}\circ\chi(\eta)^{-1},
\\
u_{\eta}&\coloneqq S(\eta)\circ u^{\xi}\circ\omega_{\natural}^{\eta}.
\end{aligned}
\end{equation*}
By abusing the notation we let $\un$ denote either $\mr{Id}_{\mc{C}_{c}^{\infty}(W,X)}$ or $\mr{Id}_{\mc{C}_{c}^{\infty}(U,X)}$ whenever
this will not be cause of confusion.
\end{definition}
\begin{theorem}
[Invariance of the Lagrangian]
\label{08251007}
For every $\eta\in A$, $u,v\in\mc{C}^{\infty}(W,Y_{0})$ we have 
\begin{equation*}
\begin{aligned}
\mf{x}\circ\left(\ov{v_{\eta}^{\xi}}\boxtimes\bigl(\gamma^{\nu}\circ\partial_{\nu}(u_{\eta}^{\xi})\bigr)\right)
\circ\chi(\eta)
&=
\mf{x}\circ\left(\ov{v_{\ze}^{\xi}}\boxtimes\bigl(\gamma^{\nu}\circ\partial_{\nu}(u_{\ze}^{\xi})\bigr)\right);
\\
\mf{x}\circ\left(\ov{\partial_{\nu}(v_{\eta}^{\xi})}\boxtimes\bigl(\gamma^{\nu}\circ u_{\eta}^{\xi}\bigr)\right)
\circ\chi(\eta)
&=
\mf{x}\circ\left(\ov{\partial_{\nu}(v_{\ze}^{\xi})}\boxtimes\bigl(\gamma^{\nu}\circ u_{\ze}^{\xi}\bigr)\right);
\\
\mf{x}\circ\left(\ov{v_{\eta}^{\xi}}\boxtimes u_{\eta}^{\xi}\right)\circ\chi(\eta)
&=
\mf{x}\circ\left(\ov{v_{\ze}^{\xi}}\boxtimes u_{\ze}^{\xi}\right).
\end{aligned}
\end{equation*}
\end{theorem}
\begin{proof}
Let $\eta\in A$ and $u,v\in\mc{C}^{\infty}(W,Y_{0})$ then, 
\begin{equation*}
\begin{aligned}
\mf{x}\circ\left(\ov{v_{\eta}^{\xi}}\boxtimes\bigl(\gamma^{\nu}\circ\partial_{\nu}(u_{\eta}^{\xi})\bigr)\right)
&=
\mf{x}\circ\left((\ov{v^{\xi}\circ\sigma(\eta))}
\boxtimes
\bigl(\widehat{S}(\eta)\circ\gamma^{\nu}\circ S(\eta)\circ\partial_{\nu}(u^{\xi}\circ\sigma(\eta))\bigr)
\right)
\\
&=
\mf{x}\circ\left((\ov{v^{\xi}}\circ\sigma(\eta))
\boxtimes
\bigl(
\widehat{S}(\eta)\circ S(\eta)\chi(\eta)_{\nu,\mu}\circ\gamma^{\mu}\circ\partial_{\nu}(u^{\xi}\circ\sigma(\eta))
\bigr)\right)
\\
&=
\mf{x}\circ\left((\ov{v^{\xi}}\circ\sigma(\eta))
\boxtimes
\bigl(
\chi(\eta)_{\nu,\mu}\circ\gamma^{\mu}\circ\partial_{\nu}(u^{\xi}\circ\sigma(\eta))
\bigr)\right)
\\
&=
\mf{x}\circ\left((\ov{v^{\xi}}\circ\sigma(\eta))
\boxtimes
\bigl(
\chi(\eta)_{\nu,\mu}\circ\gamma^{\mu}\circ\sigma(\eta)_{\rho,\nu}\partial_{\rho}(u^{\xi})\circ\sigma(\eta)
\bigr)\right)
\\
&=
\mf{x}\circ\left((\ov{v^{\xi}}\circ\sigma(\eta))
\boxtimes
\bigl(
\delta_{\rho,\mu}\gamma^{\mu}\circ\partial_{\rho}(u^{\xi})\circ\sigma(\eta)
\bigr)\right)
\\
&=
\mf{x}\circ\left((\ov{v^{\xi}}\circ\sigma(\eta))
\boxtimes
\bigl(
\gamma^{\mu}\circ\partial_{\mu}(u^{\xi})\circ\sigma(\eta)
\bigr)\right)
\\
&=
\mf{x}\circ\left(\ov{v^{\xi}}
\boxtimes
\bigl(
\gamma^{\mu}\circ\partial_{\mu}(u^{\xi})
\bigr)\right)
\circ\sigma(\eta).
\end{aligned}
\end{equation*}
Here the 1th equality follows by Def.\ref{08262042}\eqref{08262042c} and since Prp.\ref{08091051}, Thm.\ref{821} and 
Cor.\ref{03110819}; the 2th one by \eqref{08251046}; the 3th one by Def.\ref{08262218}\eqref{08262218a}; the 4th one 
since Lemma \ref{08251057}; the 5th one since $\sigma(\eta)\circ\chi(\eta)=\mr{Id}_{U}$.
So, the first equality in the statement follows. The remaining two equalities follow similarly.
\end{proof}
\begin{proposition}
\label{08262226}
For every $\eta\in A$, $u\in\mc{C}^{\infty}(W,Y_{0})$ we have 
\begin{equation*}
(i\gamma^{\nu}\partial_{\nu}^{\xi}-m\un)u=\ze
\Leftrightarrow
(i\gamma^{\nu}\partial_{\nu}^{\xi}-m\un)u_{\eta}=\ze.
\end{equation*}
\end{proposition}
\begin{proof}
In this proof we let $\chi(\eta)$ be $\chi(\eta)$. Let $\eta\in A$ and $u\in\mc{C}^{\infty}(W,Y_{0})$ then, 
\begin{equation*}
\begin{aligned}
\gamma^{\nu}\partial_{\nu}^{\xi}(u_{\eta})
&=
\gamma^{\nu}S(\eta)\partial_{\nu}^{\xi}(u^{\xi}\circ\omega_{\natural}^{\eta})
\\
&=
S(\eta)\chi(\eta)_{\nu,\mu}\gamma^{\mu}\partial_{\nu}^{\xi}(u^{\xi}\circ\omega_{\natural}^{\eta})
\\
&=
S(\eta)\gamma^{\mu}\partial_{\mu}^{\omega_{\natural}^{\eta}}(u^{\xi}\circ\omega_{\natural}^{\eta})
\\
&=
S(\eta)\gamma^{\mu}\partial_{\mu}(u^{\xi})\circ\omega_{\natural}^{\eta}
\\
&=
S(\eta)\gamma^{\mu}\partial_{\mu}^{\xi}(u)\circ\xi^{-1}\circ\omega_{\natural}^{\eta};
\end{aligned}
\end{equation*}
where the 1th equality follows in sequence since by \eqref{08060435}, Prp.\ref{08091051}, Thm.\ref{821} and Cor.\ref{03110819}; 
the 2th one by \eqref{08251046}; the 3th one by \eqref{08262239}; the 4th and 5th follow by \eqref{08060435}. Then the statement
follows since $S(\eta)$, $\xi^{-1}$ and $\omega_{\natural}^{\eta}$ admit inverse.
\end{proof}
\begin{lemma}
\label{08262318}
Let $u\in\mc{C}^{\infty}(W,Y_{0})$, $(W,\theta)$ be a chart of $M$ such that $\theta(W)=U$ then, 
\begin{equation*}
(i\gamma^{\nu}\partial_{\nu}^{\xi}-m\un)u=\ze
\Leftrightarrow
(i\gamma^{\nu}\partial_{\nu}-m\un)u^{\xi}=\ze
\Leftrightarrow
(i\gamma^{\nu}\partial_{\nu}^{\theta}-m\un)(u^{\xi}\circ\theta)=\ze
\end{equation*}
\end{lemma}
\begin{proof}
By \eqref{08060435} we have 
\begin{equation*}
\begin{aligned}
\partial_{\nu}^{\xi}u
&=
\partial_{\nu}(u^{\xi})\circ\xi
\\
\partial_{\nu}^{\theta}(u^{\xi}\circ\theta)
&=
\partial_{\nu}(u^{\xi})\circ\theta;
\end{aligned}
\end{equation*}
then the 1th equality above implies the 1th equivalence in the statement, while the 2th equality above implies the 2th 
equivalence in the statement.
\end{proof}
\begin{corollary}
\label{08262350}
For every $\eta\in A$ and $u\in\mc{C}^{\infty}(W,Y_{0})$ we have 
\eqref{08262350a}
$\Leftrightarrow$
\eqref{08262350b}
$\Leftrightarrow$
\eqref{08262350c}
$\Leftrightarrow$
\eqref{08262350d}
$\Leftrightarrow$
\eqref{08262350e}
$\Leftrightarrow$
\eqref{08262350f},
where
\begin{enumerate}
\item
$(i\gamma^{\nu}\partial_{\nu}^{\xi}-m\un)u=\ze$,
\label{08262350a}
\item
$(i\gamma^{\nu}\partial_{\nu}-m\un)u^{\xi}=\ze$,
\label{08262350b}
\item
$(i\gamma^{\nu}\partial_{\nu}^{\omega_{\natural}^{\eta}}-m\un)(u^{\xi}\circ\omega_{\natural}^{\eta})=\ze$,
\label{08262350c}
\item
$(i\gamma^{\nu}\partial_{\nu}^{\xi}-m\un)u_{\eta}=\ze$,
\label{08262350d}
\item
$(i\gamma^{\nu}\partial_{\nu}-m\un)u_{\eta}^{\xi}=\ze$,
\label{08262350e}
\item
$(i\gamma^{\nu}\partial_{\nu}^{\omega_{\natural}^{\eta}}-m\un)(u_{\eta}\circ\zeta(\eta))=\ze$.
\label{08262350f}
\end{enumerate}
\end{corollary}
\begin{proof}
Since Prp.\ref{08262226} and Lemma \ref{08262318}.
\end{proof}
\begin{definition}
If $X=\prod_{i\in\{1,2\}}X_{i}$ where $X_{i}$ is a normed space for any $i\in\{1,2\}$, then 
define $\mc{G}\coloneqq\{X_{k}\}_{k\in\{1,2\}}$, and for every $i\in\{1,2\}$ define 
$\widehat{X}_{i}\coloneqq\imath_{i}^{\mc{G}}(X_{i})$ and $K_{i}^{\mc{G}}\coloneqq\imath_{i}^{\mc{G}}\circ\Pr_{i}^{\mc{G}}$. 
\end{definition}
\begin{assumption}
\label{09081619}
For the remaining of the present section \ref{08300401}, we assume that there exists a normed space $X_{0}$ such that 
$X=\prod_{i\in\{1,2\}}X_{i}$ with $X_{i}=X_{0}$ for $i\in\{1,2\}$. Furthermore by letting $S$ be defined as in Def.\ref{08251005}, 
we assume that 
\begin{enumerate}
\item
the bilinear map $\mf{x}\circ((\ov{(\cdot)}\circ\imath_{\widehat{X}_{2}}^{X})\times\imath_{\widehat{X}_{1}}^{X})$ is separating on 
$\widehat{X}_{2}$ and separating on $\widehat{X}_{1}$,
\label{09081619a}
\item
$\widehat{\gamma}^{\nu}=\gamma^{\nu}$ for every $\nu\in[1,n]\cap\Z$,
\label{09081619b}
\item
$S(\eta)\circ K_{i}^{\mc{G}}=K_{i}^{\mc{G}}\circ S(\eta)$, 
for every $i\in\{1,2\}$ and every $\eta\in A$.
\label{09081619c}
\end{enumerate}
\end{assumption}
\begin{definition}
Define for every $u\in\mc{C}^{\infty}(W,Y_{0})$, $v\in\mc{C}^{\infty}(U,Y_{0})$ and $i\in\{1,2\}$ the maps $u[i]:W\to X_{i}$ and 
$v[i]:U\to X_{i}$, such that 
\begin{equation*}
\begin{aligned}
u\{i\}&\coloneqq\Pr_{i}^{\mc{G}}\circ u
\\
v\{i\}&\coloneqq\Pr_{i}^{\mc{G}}\circ v
\\
u[i]&\coloneqq K_{i}^{\mc{G}}\circ u
\\
v[i]&\coloneqq K_{i}^{\mc{G}}\circ v.
\end{aligned}
\end{equation*}
\end{definition}
Let $i\in\{1,2\}$ and $\eta\in A$, notice that for any $\natural\in\{\emptyset,c,c1\}$ we have 
$u\in\mc{C}_{\natural}^{\infty}(W,Y_{0})\Rightarrow u_{\eta}\in\mc{C}_{\natural}^{\infty}(W,Y_{0})$,  
$u\in\mc{C}_{\natural}^{\infty}(W,Y_{0})\Rightarrow u[i]\in\mc{C}_{\natural}^{\infty}(W,Y_{0})$ and 
$v\in\mc{C}_{\natural}^{\infty}(U,Y_{0})\Rightarrow v[i]\in\mc{C}_{\natural}^{\infty}(U,Y_{0})$.
Furthermore since $\Pr_{j}^{\mc{G}}\circ\imath_{i}^{\mc{G}}=\delta_{i,j}\mr{Id}_{X_{i}}$
we have $u[i]\{j\}=u[j]\{i\}=\delta_{i,j}u\{i\}$ for every $u\in\mc{C}_{\natural}^{\infty}(W,Y_{0})$ and $j\in[1,2]\cap\Z$
moreover $u_{\eta}^{\xi}=u_{\eta}\circ\xi^{-1}$ so, no confusion might arise with the above convention, finally 
$u_{\eta}[i]=u[i]_{\eta}$ and $u_{\eta}^{\xi}[i]=u_{\eta}[i]^{\xi}=u[i]_{\eta}^{\xi}$.
\begin{definition}
\label{08300640}
We call $\lr{\mh{L}}{\phi}$ a preappropriate Dirac system of mass $m$ iff 
\begin{enumerate}
\item
$m\in\R_{\geq 0}$,
\item 
$\mh{L}:O\to\mc{Z}$ such that 
\begin{equation*}
\mh{L}:(a,b)
\mapsto
\frac{i}{2}
\left(
\mf{x}\left(\ov{K_{2}^{\mc{G}}(a)},\sum_{\nu=1}^{n}\gamma^{\nu}(K_{1}^{\mc{G}}(\Pr_{\nu}^{\mc{B}}b))\right)
-
\sum_{\nu=1}^{n}\mf{x}\left(\ov{K_{2}^{\mc{G}}(\Pr_{\nu}^{\mc{B}}b)},\gamma^{\nu}(K_{1}^{\mc{G}}(a))\right)
\right)
-m\mf{x}\left(\ov{K_{2}^{\mc{G}}(a)},K_{1}^{\mc{G}}(a)\right); 
\end{equation*}
\item
for every $\eta\in A$ and $u\in\mf{Z}$ we have 
\begin{equation*}
\phi_{\eta}(u)=S(\eta)\circ u^{\xi}\circ\omega_{\natural}^{\eta}.
\end{equation*}
\end{enumerate}
\end{definition}
\begin{assumption}
For the remaining of the present section \ref{08300401} we let $\lr{\mh{L}}{\phi}$ be a preappropriate Dirac 
system of mass $m$. 
\end{assumption}
\begin{remark}
\label{09081747}
Notice that 
\begin{equation*}
\mh{L}=
\frac{i}{2}
\bigl(
\mf{x}\circ(T_{1}\boxtimes T_{2})
-
\sum_{\nu=1}^{n}\mf{x}\circ(R_{1,\nu}\boxtimes R_{2,\nu})
\bigr)\circ\imath_{O}^{\mc{Y}}
-
m\mf{x}\circ(T_{1}\boxtimes S_{2})\circ\imath_{O}^{\mc{Y}};
\end{equation*}
where
\begin{equation*}
\begin{aligned}
T_{1}&=\ov{(\cdot)}\circ K_{2}^{\mc{G}}\circ\Pr_{1}^{\mc{E}}
\\
T_{2}&=\sum_{\nu=1}^{n}\gamma^{\nu}\circ K_{1}^{\mc{G}}\circ\Pr_{\nu}^{\mc{B}}\circ\Pr_{2}^{\mc{E}}
\\
R_{1,\nu}&=\ov{(\cdot)}\circ K_{2}^{\mc{G}}\circ\Pr_{\nu}^{\mc{B}}\circ\Pr_{2}^{\mc{E}}
\\
R_{2,\nu}&=\gamma^{\nu}\circ K_{1}^{\mc{G}}\circ\Pr_{1}^{\mc{E}} 
\\
S_{2}&=K_{1}^{\mc{G}}\circ\Pr_{1}^{\mc{E}}.
\end{aligned} 
\end{equation*}
\end{remark}
\begin{proposition}
\label{08310755}
We have
\begin{equation}
\label{08310755a}
\mh{L}\in\mc{C}^{\infty}(O,\mc{Z}).
\end{equation}
Next by letting $\omega_{\mu}^{\eta}\coloneqq\Pr_{\mu}\circ\omega^{\eta}$ for any $\mu\in[1,n]\cap\Z$, we assume that 
\begin{equation}
\label{08310749}
(\forall\mu\in[1,n]\cap\Z)
(\forall\eta\in A)
(\forall\nu\in[0,n]\cap\Z)(\forall\mu\in[1,n]\cap\Z)(\sup_{x\in W}\left|\partial_{\nu}^{\xi}(\omega^{\eta}_{\mu})(x)\right|<\infty).
\end{equation}
Thus, 
\begin{equation*}
\begin{aligned}
\phi: A&\to\mc{L}(\mf{Z},\mf{Z}),
\\
\psi: A&\to\mc{L}(\mf{Z},\mf{T}).
\end{aligned}
\end{equation*}
\end{proposition}
\begin{proof}
$\mh{L}\in\mc{C}^{\infty}(O,\mc{Z})$ since Rmk.\ref{09081747}, Prp.\ref{01101646}, by \eqref{12152057} and since continuous linear
maps are smooth. Let $\eta\in A$. Trivially $\phi_{\eta}=S(\eta)_{\intercal}\circ(\upzeta(\eta))^{\ast}$. 
Now $(\upzeta(\eta))^{\ast}\in\mc{L}(\mf{Z},\mf{Z})$ since Lemma \ref{08310357}, while 
$S(\eta)_{\intercal}\in\mc{L}(\mf{Z},\mf{Z})$ since Lemma \ref{08300643} therefore 
$S(\eta)_{\intercal}\circ(\upzeta(\eta))^{\ast}\in\mc{L}(\mf{Z},\mf{Z})$ so, 
\begin{equation*}
\phi_{\eta}\in\mc{L}(\mf{Z},\mf{Z}).
\end{equation*}
Next $\psi_{\eta}=\partial^{\xi}\circ\phi_{\eta}$ but $\partial^{\xi}\in\mc{L}(\mf{Z},\mf{T})$ indeed 
\begin{equation*}
\widehat{t}(\partial^{\xi}u)
=\max_{\mu\in[1,n]\cap\Z}\sup_{x\in W}\|\partial_{\mu}^{\xi}u\|_{X}
\leq\widehat{q}(u),
\end{equation*}
therefore $\psi_{\eta}\in\mc{L}(\mf{Z},\mf{T})$.
\end{proof}
\begin{lemma}
\label{09011311}
Assume that $S\in\mc{C}^{3}( A,\mc{L}(X))$ and $\widehat{\sigma}\in\mc{C}^{3}(A\times U,U)$. 
Then, $\Upphi(u)\in\mc{C}^{3}(A\times U,Y_{0})$ and $\Uppsi(u)\in\mc{C}^{2}(A\times U,\mf{Y})$,
for every $u\in\mf{Z}$. 
\end{lemma}
\begin{proof}
$\Upphi(u)=\mf{g}_{0}\circ\left((S\circ\beta)\boxtimes(u\circ\xi^{-1}\circ\widehat{\sigma})\right)$ 
so the first part of the statement follows since Prp.\ref{01101646} and $\beta$ being smooth. 
Next trivially we have $\Pr_{\nu}^{\mc{B}}\circ\Uppsi(u)=\frac{\partial}{\partial e_{\nu}^{2}[n]}(\Upphi)$ which is 
$\mc{C}^{2}$ since the first part of the statement and Cor.\ref{04221226} therefore the second part of the statement follows 
since Cor.\ref{04121839}.
\end{proof}
\begin{corollary}
\label{09020340}
Let $\omega_{\mu}^{\eta}\coloneqq\Pr_{\mu}\circ\omega^{\eta}$ for any $\mu\in[1,n]\cap\Z$. 
Assume that 
\begin{enumerate}
\item
\eqref{08310749} holds true,
\item
$S\in\mc{C}^{3}( A,\mc{L}(X))$, 
\item
$\widehat{\sigma}\in\mc{C}^{3}(A\times U,U)$, 
\item
$\widehat{\chi}\in\mc{C}^{2}(A\times U,U)$.
\end{enumerate}
Thus, $\lr{\mh{L}}{\phi,\omega}$ is a perfect appropriate transformation system on $X$ moreover for every $\eta\in A$ 
and $u\in\mf{Z}$ we have
\begin{enumerate}
\item
$\phi: A\to\mc{L}(\mf{Z},\mf{Z})$ in particular $\phi_{\eta}\in\mc{C}^{\infty}(\mf{Z},\mf{Z})$,
\item
$\phi_{\eta}(u)^{-1}=S(\eta)^{-1}\circ u\circ\upzeta(\eta)^{-1}$, 
\item
$\Upphi(u)\in\mc{C}^{3}(A\times U,Y_{0})$,
\item
$\mh{L}\in\mc{C}^{\infty}(\mc{Y},\mc{Z})$.
\end{enumerate}
\end{corollary}
\begin{proof}
Since Prp.\ref{08310755}, Lemma \ref{09011311} and \eqref{12152057}. 
\end{proof}
\begin{convention}
In what follows often we shall implement the following convention. Let $a\in X$ and $b\in\mf{Y}$, then for any $i\in[1,2]\cap\Z$
and $\nu\in[1,n]\cap\Z$ we let $a[i]$ be $K_{i}^{\mc{G}}(a)$; $b_{\nu}$ be $\Pr_{\nu}^{\mc{B}}(b)$.
\end{convention}
In particular $b_{\nu}[i]=(K_{i}^{\mc{G}}\circ\Pr_{\nu}^{\mc{B}})(b)$.
\begin{lemma}
\label{09100940}
Let $(a,b)\in O$, $(h,l)\in Y_{0}\times\mf{Y}$ thus, 
\begin{equation*}
\begin{aligned}
\mh{L}^{(1)}(a,b)(h,l)
=
\frac{i}{2}\mf{x}\left(\ov{a[2]},\sum_{\nu=1}^{n}\gamma^{\nu}(l_{\nu}[1])\right)
&+
\frac{i}{2}\mf{x}\left(\ov{h[2]},\sum_{\nu=1}^{n}\gamma^{\nu}(b_{\nu}[1])\right)
\\
-
\frac{i}{2}\sum_{\nu=1}^{n}\mf{x}\left(\ov{b_{\nu}[2]},\gamma^{\nu}(h[1])\right)
&-
\frac{i}{2}\sum_{\nu=1}^{n}\mf{x}\left(\ov{l_{\nu}[2]},\gamma^{\nu}(a[1])\right)
\\
-m\mf{x}(\ov{a[2]},h[1])
&-
m\mf{x}(\ov{h[2]},a[1]).
\end{aligned}
\end{equation*}
\end{lemma}
\begin{proof}
Since Rmk.\ref{09081747}, Prp.\ref{03241344} and Prp.\ref{813} we have 
\begin{equation*}
\begin{aligned}
\mh{L}^{(1)}(a,b)(h,l)
=
\frac{i}{2}\mf{x}(T_{1}(a,b),T_{2}(h,l))
&+
\frac{i}{2}\mf{x}(T_{1}(h,l),T_{2}(a,b))
\\
-
\frac{i}{2}\sum_{\nu=1}^{n}\mf{x}(R_{1,\nu}(a,b),R_{2,\nu}(h,l))
&-
\frac{i}{2}\sum_{\nu=1}^{n}\mf{x}(R_{1,\nu}(h,l),R_{2,\nu}(a,b))
\\
-m\mf{x}(T_{1}(a,b),S_{2}(h,l))
&-
m\mf{x}(T_{1}(h,l),S_{2}(a,b));
\end{aligned}
\end{equation*}
then the statement follows.
\end{proof}
\begin{lemma}
\label{09101848}
Let $(a,b)\in O$, $h_{i}\in\widehat{X}_{i}$ for any $i\in[1,2]\cap\Z$, and $\nu\in[1,n]\cap\Z$ thus,
\begin{equation*}
\begin{aligned}
\frac{\partial\mh{L}}{\partial Y_{0}}(a,b)(h_{1})
&=
\mf{x}\left(\ov{\sum_{\nu=1}^{n}\frac{i}{2}\gamma^{\nu}(b_{\nu}[2])-ma[2]},h_{1}\right)
\\
\frac{\partial\mh{L}}{\partial Y_{0}}(a,b)(h_{2})
&=
\mf{x}\left(\ov{h_{2}},\sum_{\nu=1}^{n}\frac{i}{2}\gamma^{\nu}(b_{\nu}[1])-ma[1]\right)
\\
\left(\frac{\partial\mh{L}_{\mf{e}}}{\partial Y_{\nu}}\circ(\mf{d}_{O})_{\natural}\right)
(a,b)(h_{1})
&=
\mf{x}\left(\ov{-\frac{i}{2}\gamma^{\nu}(a[2])},h_{1}\right)
\\
\left(\frac{\partial\mh{L}_{\mf{e}}}{\partial Y_{\nu}}\circ(\mf{d}_{O})_{\natural}\right)
(a,b)(h_{2})
&=
\mf{x}\left(\ov{h_{2}},-\frac{i}{2}\gamma^{\nu}(a[1])\right).
\end{aligned}
\end{equation*}
\end{lemma}
\begin{proof}
Let $h\in X$ thus, since \eqref{08310755a} and Lemma \ref{10181658} we have 
\begin{equation*}
\begin{aligned}
\frac{\partial\mh{L}}{\partial Y_{0}}(a,b)(h)
&=
\mh{L}^{(1)}(a,b)(h,\ze)
\\
\left(\frac{\partial\mh{L}_{\mf{e}}}{\partial Y_{\nu}}\circ(\mf{d}_{O})_{\natural}\right)
(a,b)(h)
&=
\mh{L}^{(1)}(a,b)(\ze,l(\nu,h));
\end{aligned}
\end{equation*}
where $l(\nu,h)\in\mf{Y}$ such that $\Pr_{\mu}^{\mc{B}}(l(\nu,h))=\delta_{\mu,\nu}h$ for every $\mu\in[1,n]\cap\Z$.
Thus, 
\begin{equation*}
\begin{aligned}
\frac{\partial\mh{L}}{\partial Y_{0}}(a,b)(h_{1})
&=
-\frac{i}{2}\sum_{\nu=1}^{n}\mf{x}(\ov{b_{\nu}[2]},\gamma^{\nu}(h_{1}))-m\mf{x}(\ov{a[2]},h_{1})
\\
&=
-\frac{i}{2}\sum_{\nu=1}^{n}\mf{x}(\ov{\gamma^{\nu}(b_{\nu}[2])},h_{1})-m\mf{x}(\ov{a[2]},h_{1})
\\
&=
\sum_{\nu=1}^{n}\mf{x}(\ov{\frac{i}{2}\gamma^{\nu}(b_{\nu}[2])-ma[2]},h_{1})
\\
&=
\mf{x}\left(\ov{\sum_{\nu=1}^{n}\frac{i}{2}\gamma^{\nu}(b_{\nu}[2])-ma[2]},h_{1}\right),
\end{aligned}
\end{equation*}
where the 1th equality follows since Lemma \ref{09100940}; the 2th one since Assumption \ref{09081619}\eqref{09081619b};
the 3th one since $\mf{x}\circ(\ov{(\cdot)}\boxtimes\mr{Id}_{X}):X\times X\to\ov{\mc{Z}}$ is bilinear, since the complex structure
on $\ov{\mc{Z}}$ and since $i^{\ast}=-i$ and $m^{\ast}=m$; the 4th one since 
$\mf{x}\circ(\ov{(\cdot)}\boxtimes\mr{Id}_{X}):X\times X\to\ov{\mc{Z}}$ is bilinear.
Thus, the first equality in the statement follows. The remaining equalities in the statement can be proven similarly.
\end{proof}
\begin{theorem}
[Dirac Equation 1]
\label{09102044}
We have 
\begin{equation*}
\begin{aligned}
u\in\mr{EL}_{W}(\mh{L})
&\Leftrightarrow 
(\forall j\in[1,2]\cap\Z)\left((\sum_{\nu=1}^{n}i\gamma^{\nu}\circ\partial_{\nu}-m\un)\phi_{\ze}^{\xi}(u[j])=\ze\right)
\\
&\Leftrightarrow
(\forall j\in[1,2]\cap\Z)\left((\sum_{\nu=1}^{n}i\gamma^{\nu}\circ\partial_{\nu}-m\un)\phi_{\eta}^{\xi}(u[j])=\ze\right)
\\
&\Leftrightarrow
(\forall j\in[1,2]\cap\Z)\left((\sum_{\nu=1}^{n}i\gamma^{\nu}\circ\partial_{\nu}^{\xi}-m\un)\phi_{\ze}(u[j])=\ze\right)
\\
&\Leftrightarrow
(\forall j\in[1,2]\cap\Z)\left((\sum_{\nu=1}^{n}i\gamma^{\nu}\circ\partial_{\nu}^{\xi}-m\un)\phi_{\eta}(u[j])=\ze\right)
\\
&\Leftrightarrow
(\forall j\in[1,2]\cap\Z)\left((\sum_{\nu=1}^{n}i\gamma^{\nu}\circ\partial_{\nu}^{\omega_{\natural}^{\eta}}-m\un)
(\phi_{\eta}(u[j])\circ\upzeta(\eta))=\ze\right)
\\
&\Leftrightarrow
(\forall j\in[1,2]\cap\Z)\left((\sum_{\nu=1}^{n}i\gamma^{\nu}\circ\partial_{\nu}^{\omega_{\natural}^{\eta}}-m\un)
(\phi_{\ze}^{\xi}(u[j])\circ\omega_{\natural}^{\eta})=\ze\right).
\end{aligned}
\end{equation*}
\end{theorem}
\begin{proof}
Let $x\in U$. Since $h=h[1]+h[2]$ for every $h\in X$, we have that 
\begin{equation}
\label{09102022}
\begin{aligned}
(\forall h\in X)
\left(\frac{\partial\mh{L}}{\partial Y_{0}}\circ(u^{\xi}\boxtimes\partial^{\xi}u)\right)(x)(h)
&=
\sum_{\nu=1}^{n}
\frac{\partial}{\partial e_{\nu}[n]}
\left(\frac{\partial\mh{L}_{\mf{e}}}{\partial Y_{\nu}}
\circ(\mf{d}_{O})_{\natural}
\circ(u^{\xi}\boxtimes\partial^{\xi}u)\right)
(x)(h)
\\
&
\Leftrightarrow
\\
(\forall i\in[1,2]\cap\Z)
(\forall h\in X)
\bigl(
\left(\frac{\partial\mh{L}}{\partial Y_{0}}\circ(u^{\xi}\boxtimes\partial^{\xi}u)\right)(x)(h[i])
&=
\sum_{\nu=1}^{n}
\frac{\partial}{\partial e_{\nu}[n]}\left(
\frac{\partial\mh{L}_{\mf{e}}}{\partial Y_{\nu}}\circ(\mf{d}_{O})_{\natural}\circ(u^{\xi}\boxtimes\partial^{\xi}u)\right)(x)
(h[i])
\bigr).
\end{aligned}
\end{equation}
Next since Lemma \ref{09101848} we have 
\begin{equation*}
\left(\frac{\partial\mh{L}}{\partial Y_{0}}\circ(u^{\xi}\boxtimes\partial^{\xi}u)\right)(x)(h[1])
=
\mf{x}\left(\ov{\sum_{\nu=1}^{n}\frac{i}{2}\gamma^{\nu}(\partial_{\nu}u^{\xi}[2])(x)-mu^{\xi}[2](x)},h[1]\right);
\end{equation*}
while 
\begin{equation}
\label{09102113}
\left(\frac{\partial\mh{L}_{\mf{e}}}{\partial Y_{\nu}}\circ(\mf{d}_{O})_{\natural}\circ(u^{\xi}\boxtimes\partial^{\xi}u)\right)(x)(h[1])
=
\mf{x}\left(\ov{-\frac{i}{2}\gamma^{\nu}(u^{\xi}[2](x))},h[1]\right),
\end{equation}
but by Rmk.\ref{10231117} applied to $\mr{ev}_{h[1]}^{\mc{L}_{b}(X,\ov{\mc{Z}})}$, Prp.\ref{813} and Thm.\ref{821} we have 
\begin{equation*}
\frac{\partial}{\partial e_{\nu}[n]}\left(\frac{\partial\mh{L}_{\mf{e}}}{\partial Y_{\nu}}\circ(\mf{d}_{O})_{\natural}
\circ(u^{\xi}\boxtimes\partial^{\xi}u)\right)(x)
(h[1])
=
\frac{\partial}{\partial e_{\nu}[n]}
\left(\frac{\partial\mh{L}_{\mf{e}}}{\partial Y_{\nu}}\circ(\mf{d}_{O})_{\natural}\circ(u^{\xi}\boxtimes\partial^{\xi}u)(\cdot)(h[1])
\right)(x),
\end{equation*}
and then by Prp.\ref{05091652} and \eqref{09102113} we obtain 
\begin{equation*}
\frac{\partial}{\partial e_{\nu}[n]}\left(
\frac{\partial\mh{L}_{\mf{e}}}{\partial Y_{\nu}}\circ(\mf{d}_{O})_{\natural}\circ(u^{\xi}\boxtimes\partial^{\xi}u)\right)(x)
(h[1])
=
\mf{x}\left(\ov{-\frac{i}{2}\gamma^{\nu}((\partial_{\nu}u^{\xi}[2])(x))},h[1]\right),
\end{equation*}
and then 
\begin{equation*}
\begin{aligned}
(\forall h\in X)
\bigl(
\left(
\frac{\partial\mh{L}}{\partial Y_{0}}\circ(u^{\xi}\boxtimes\partial^{\xi}u)\right)(x)(h[1])
&=
\sum_{\nu=1}^{n}
\frac{\partial}{\partial e_{\nu}[n]}\left(\frac{\partial\mh{L}_{\mf{e}}}{\partial Y_{\nu}}\circ(\mf{d}_{O})_{\natural}
\circ(u^{\xi}\boxtimes\partial^{\xi}u)\right)(x)
(h[1])
\bigr)
\\
&\Leftrightarrow
\\
(\forall h\in X)
&(\mf{x}\left(\ov{\sum_{\nu=1}^{n}i\gamma^{\nu}(\partial_{\nu}u^{\xi}[2])(x)-mu^{\xi}[2](x)},h[1]\right)=\ze).
\end{aligned}
\end{equation*}
Next clearly $\widehat{X}_{i}=\{h[i]\,\vert\,h\in X\}$ thus, by Assumption \ref{09081619}\eqref{09081619a} we obtain 
by the above equivalence that 
\begin{equation*}
\begin{aligned}
(\forall h\in X)
\bigl(
\left(\frac{\partial\mh{L}}{\partial Y_{0}}\circ(u^{\xi}\boxtimes\partial^{\xi}u)\right)(x)(h[1])
&=
\sum_{\nu=1}^{n}
\frac{\partial}{\partial e_{\nu}[n]}\left(
\frac{\partial\mh{L}_{\mf{e}}}{\partial Y_{\nu}}\circ(\mf{d}_{O})_{\natural}\circ(u^{\xi}\boxtimes\partial^{\xi}u)\right)(x)
(h[1])
\bigr)
\\
&\Leftrightarrow
\\
\sum_{\nu=1}^{n}i\gamma^{\nu}(&\partial_{\nu}(u^{\xi}[2]))(x)-mu^{\xi}[2](x)=\ze.
\end{aligned}
\end{equation*}
With a similar proof and we can prove that 
\begin{equation*}
\begin{aligned}
(\forall h\in X)
\bigl(
\left(\frac{\partial\mh{L}}{\partial Y_{0}}\circ(u^{\xi}\boxtimes\partial^{\xi}u)\right)(x)(h[2])
&=
\sum_{\nu=1}^{n}
\frac{\partial}{\partial e_{\nu}[n]}\left(
\frac{\partial\mh{L}_{\mf{e}}}{\partial Y_{\nu}}\circ(\mf{d}_{O})_{\natural}\circ(u^{\xi}\boxtimes\partial^{\xi}u)\right)(x)
(h[2])
\bigr)
\\
&\Leftrightarrow
\\
\sum_{\nu=1}^{n}i\gamma^{\nu}(&\partial_{\nu}(u^{\xi}[1]))(x)-mu^{\xi}[1](x)=\ze;
\end{aligned}
\end{equation*}
therefore the first equivalence of the statement follows since \eqref{09102022} and Def.\ref{07230430}; the remaining 
equivalences of the statement follow by the first one and since Cor.\ref{08262350}.
\end{proof}
\begin{corollary}
[Dirac Equation 2]
\label{09102044bis}
We have 
\begin{equation*}
\begin{aligned}
u\in\mr{EL}_{W}(\mh{L})
&\Leftrightarrow 
(\forall j\in[1,2]\cap\Z)\left((\sum_{\nu=1}^{n}i\gamma^{\nu}\circ\partial_{\nu}-m\un)\phi_{\ze}^{\xi}(u)\{j\}=\ze\right)
\\
&\Leftrightarrow
(\forall j\in[1,2]\cap\Z)\left((\sum_{\nu=1}^{n}i\gamma^{\nu}\circ\partial_{\nu}-m\un)\phi_{\eta}^{\xi}(u)\{j\}=\ze\right)
\\
&\Leftrightarrow
(\forall j\in[1,2]\cap\Z)\left((\sum_{\nu=1}^{n}i\gamma^{\nu}\circ\partial_{\nu}^{\xi}-m\un)\phi_{\ze}(u)\{j\}=\ze\right)
\\
&\Leftrightarrow
(\forall j\in[1,2]\cap\Z)\left((\sum_{\nu=1}^{n}i\gamma^{\nu}\circ\partial_{\nu}^{\xi}-m\un)\phi_{\eta}(u)\{j\}=\ze\right)
\\
&\Leftrightarrow
(\forall j\in[1,2]\cap\Z)\left((\sum_{\nu=1}^{n}i\gamma^{\nu}\circ\partial_{\nu}^{\omega_{\natural}^{\eta}}-m\un)
(\phi_{\eta}(u)\{j\}\circ\upzeta(\eta))=\ze\right)
\\
&\Leftrightarrow
(\forall j\in[1,2]\cap\Z)\left((\sum_{\nu=1}^{n}i\gamma^{\nu}\circ\partial_{\nu}^{\omega_{\natural}^{\eta}}-m\un)
(\phi_{\ze}^{\xi}(u)\{j\}\circ\omega_{\natural}^{\eta})=\ze\right).
\end{aligned}
\end{equation*}
\end{corollary}
\begin{proof}
The statement follows since Thm.\ref{09102044}; since $u[j]_{\eta}=u_{\eta}[j]$ and $u[j]_{\eta}^{\xi}=u_{\eta}^{\xi}[j]$;
since $\Pr_{k}\circ(u[j])=\delta_{j,k}u\{j\}$; since Prp.\ref{813} and Thm.\ref{821}.
\end{proof}
In what follows recall that $\mf{a}$ is defined in Def.\ref{08280414}.
\begin{theorem}
Let $u\in\mf{Z}$. Then
\begin{multline}
\label{09061904}
(\forall\eta\in A)
\bigl(
(\forall j\in[1,2]\cap\Z)\left((\sum_{\nu=1}^{n}i\gamma^{\nu}\circ\partial_{\nu}^{\xi}-m\un)\phi_{\ze}(u[j])=\ze\right)
\\
\Rightarrow
\\
(\forall j\in[1,2]\cap\Z)\left((\sum_{\nu=1}^{n}i\gamma^{\nu}\circ\partial_{\nu}^{\xi}-m\un)\phi_{\eta}(u[j])=\ze\right)
\bigr).
\end{multline}
Next assume the hypothesis of Cor.\ref{09020340} and that $\widehat{J}$ is the constant map equal to $1$.
Then $\mf{a}$ is constant moreover
\begin{equation*}
(\forall j\in[1,2]\cap\Z)\left((\sum_{\nu=1}^{n}i\gamma^{\nu}\circ\partial_{\nu}^{\xi}-m\un)\phi_{\ze}(u[j])=\ze\right)
\Rightarrow
\sum_{i=1}^{n}\frac{\partial\mh{N}^{i}[u]}{\partial e_{i}[n]}=\ze,
\end{equation*}
and under the hypothesis and notation of Cor.\ref{05112017} and if $p=+\infty$, then we have 
\begin{equation*}
(\forall j\in[1,2]\cap\Z)\left((\sum_{\nu=1}^{n}i\gamma^{\nu}\circ\partial_{\nu}^{\xi}-m\un)\phi_{\ze}(u[j])=\ze\right)
\Rightarrow
\frac{dQ_{V}(u)}{dt}=\ze.
\end{equation*}
\end{theorem}
\begin{proof}
\eqref{09061904} follows since Cor.\ref{08262350}(\eqref{08262350a}$\Rightarrow$\eqref{08262350d}).
Next $\mf{a}$ is constant since Thm.\ref{08251007} and Prp.\ref{08161352}, while $\lr{\mh{L}}{\phi,\omega}$ is an 
appropriate transformation system on $X$ since Cor.\ref{09020340}, then the remaining of the statement follows since 
Thm.\ref{09102044}, and Rmk.\ref{08281125}, Cor.\ref{08110158} and Cor.\ref{05112017}.
\end{proof}
\begin{definition}
[Physical Dirac Solutions]
Define
\begin{equation*}
\mr{EL}_{W}^{ph}(\mh{L})\coloneqq\left\{u\in\mr{EL}_{W}(\mh{L})\,\big\vert\,u\{1\}=u\{2\}\right\}.
\end{equation*}
\end{definition}
\section{Appendix}
In this section for the sake of completeness and reader's benefit we gather the results we employ in the body of the 
paper mostly regarding differential calculus for locally convex space valued maps defined on open sets of normed spaces. 
Made exception for Cor.\ref{01031245b} which is by far more general than similar results, the material present in the Appendix 
is more or less standard and the order in which here is collected is the one of quotation in the paper. 
All the linear spaces involved in this appendix are tacitly assumed to be over $\K\in\{\R,\C\}$.
\begin{corollary} 
[Partial Differential Expansion of the Differential and Viceversa]
\label{02211942}
Let $n\in\Z_{>0}$, $\mc{E}=\{E_{i}\}_{i=1}^{n}$ be a family of normed spaces, $F$ be a Hausdorff locally convex space, 
$A$ be an open set of $E=\prod_{i=1}^{n}E_{i}$ and $f:A\to F$ be a continuous map. 
If $f$ is differentiable at $a\in A$, then $f$ is differentiable at $a$ with 
respect to the subspace $E_{i}$ for every $i\in[1,n]\cap\Z$ and \eqref{02281718} holds true for $x=a$.
Moreover in order that $f$ be continuously differentiable in $A$ it is necessary and sufficient that 
for every $i\in[1,n]\cap\Z$ the map  $f$ be differentiable at each point with respect to the subspace $E_{i}$ 
and that $\mr{D}_{i}f\in\mc{C}(A,\mc{L}(E_{i},F))$. In such a case we have for all $x\in A$ and $k\in[1,n]\cap\Z$
\begin{equation}
\label{02281718}
\begin{aligned}
(\mr{D}_{k}f)(x)&=f^{\prime}(x)\circ\mf{i}_{k},
\\
f^{\prime}(x)&=\sum_{i=1}^{n}(\mr{D}_{i}f)(x)\circ\Pr_{i};
\end{aligned}
\end{equation}
namely
\begin{equation}
\label{02281719}
\begin{aligned}
\mr{D}_{k}f&=\mf{i}_{k}^{\dagger}\circ f^{\prime},
\\
f^{\prime}&=\sum_{i=1}^{n}\Pr_{i}^{\dagger}\circ(\mr{D}_{i}f).
\end{aligned}
\end{equation}
\end{corollary}
\begin{proposition}
\label{813}
Let $E$ be a normed space, $F$ be a Hausdorff locally convex space, and $u\in\mc{L}(E,F)$ thus $u$ is differentiable and 
$u^{\prime}(x)=u$ for every $x\in E$.
\end{proposition}
\begin{proposition}
\label{05231704}
Let $E$ be a normed space,$F$ be a Hausdorff locally convex space,  $A$ and $B$ be open sets of $E$ such that $B\subset A$ and 
$x_{0}\in B$. Thus $f:A\to F$ is differentiable at $x_{0}$ if and only if 
$f\circ\imath_{B}^{A}$ is differentiable at $x_{0}$ and in such cases 
the differential of $f$ at $x_{0}$ equals the differential of $f\circ\imath_{B}^{A}$ at $x_{0}$.
\end{proposition}
\begin{theorem}
[Differential of the Composition of Differentiable Maps]
\label{821}
Let $E,F$ be normed spaces, $G$ be a Hausdorff locally convex space, $x_{0}\in E$, $A$ be an open neighborhood of $x_{0}$, 
$f$ be a continuous map of $A$ into $F$, $B$ be an open neighborhood of $f(x_{0})$ and $g$ be a continuous map of $B$ into $G$.
Let $D\coloneqq f^{-1}(B)$ , thus $x_{0}\in D$ and since $f$ is continuous we have that
$D$ is open relatively to $A$, hence open relatively to $E$ since $A$ is open. Let 
$h\coloneqq(g\circ\imath_{f(D)}^{B})\circ(f\circ\imath_{D}^{A})_{\natural}$,
Thus if $f$ is differentiable at $x_{0}$ and $g$ is differentiable
at $f(x_{0})$, then $h$ is differentiable at $x_{0}$ and we have 
\begin{equation*}
h^{\prime}(x_{0})=g^{\prime}(f(x_{0}))\circ f^{\prime}(x_{0}).
\end{equation*}
\end{theorem}
\begin{proposition}
\label{812}
Let $E$ be a normed space, $F$ be a Hausdorff locally convex space and $f$ be a map defined on an open set of $E$.
If $f$ is constant thus it is differentiable and its differential is the constant map equal to $\ze$.
\end{proposition}
\begin{proposition}
[Differential of a $\boxtimes$]
\label{815}
Let $E$ be a normed space, $A$ be an open set of $E$, $x_{0}\in A$, $n\in\Z_{>0}$, 
$\mc{F}=\{F_{i}\}_{i=1}^{n}$ be a family of Hausdorff locally convex spaces, $F=\prod_{i=1}^{n}F_{i}$ be endowed by the product 
topology and $f:A\to F$. In order that $f$ is differentiable at $x_{0}\in A$ is necessary and sufficient that for every 
$i\in[1,n]\cap\Z$ the map $f_{i}\coloneqq\Pr_{i}^{\mc{F}}\circ f$ is differentiable at $x_{0}$.
If $f$ is differentiable at $x_{0}$, then $\Pr_{i}^{\mc{F}}\circ(f^{\prime}(x_{0}))=f_{i}^{\prime}(x_{0})$ for every $i\in[1,n]\cap\Z$. 
\end{proposition}
\begin{proposition}
[Differentiability implies directional derivability]
\label{08091051}
Let $E$ be a Hausdorff locally convex space, $X$ a normed space, $A$ be an open set of $X$, $x\in A$ and $f:A\to E$. 
If $f$ is differentiable, then for every $v\in X$ such that $\|v\|=1$ we have that $f$ is directionally derivable at $x$ in the 
direction $v$ and 
\begin{equation*}
D_{v}(f)(x)=f^{(1)}(x)(v).
\end{equation*}
\end{proposition}
\begin{corollary}
[Higher differentiation commutes with projection]
\label{8126}
Let $m,p\in\Z_{>0}$, $\mc{F}=\{F_{k}\}_{k=1}^{m}$ be a family of Hausdorff locally convex spaces, $F=\prod_{k=1}^{m}F_{k}$
be endowed by the product topology, let $A$ be an open set of $E$ and $f:A\to F$. In order that $f$ be $p$ times differentiable, 
it is necessary and sufficient that for every $k\in[1,m]\cap\Z$ the map $\Pr_{k}^{\mc{F}}\circ f$ be $p$ times differentiable.
\end{corollary}
\begin{theorem}
[Continuous higher differential of a composition]
\label{81210}
Let $E,F$ be normed spaces, $A$ be an open set of $E$, $p\in\Z_{>0}$, $f:A\to F$, $B$ an open set of $F$ such that 
$f(A)\subseteq B$, $G$ be a Hausdorff locally convex space, $g:B\to G$ and let $h=g\circ\imath_{f(A)}^{B}\circ f_{\natural}$. 
If $f\in\mc{C}^{p}(A,F)$ and $g\in\mc{C}^{p}(B,G)$, then $h\in\mc{C}^{p}(A,G)$.
\end{theorem}
\begin{corollary}
[Necessary and sufficient conditions for $f$ to be $N$ times continuously differentiable]
\label{04221226}
Let $F$ be a Hausdorff locally convex space, $n\in\Z_{>0}$, $N\in\Z_{>0}$, $\mc{E}=\{E_{i}\}_{i=1}^{n}$ be a sequence of 
normed spaces, let $E=\prod_{i=1}^{n}E_{i}$, $A$ be an open set of $E$ and $f:A\to F$. In order that $f$ be $N$ times continuously 
differentiable it is necessary and sufficient that for every $r\in[1,n]\cap\Z$ the map $\mr{D}_{r}f:A\to\mc{L}(E_{r},F)$ 
be $N-1$ times continuously differentiable. 
\end{corollary}
\begin{proposition}
[Partial differential and directional derivative]
\label{08091052}
Let $E$ be a Hausdorff locally convex space, $n\in\Z_{>0}$, $X_{i}$ be a normed space for any $i\in[1,n]\cap\Z$, 
$X=\prod_{i=1}^{n}X_{i}$, $A$ be an open set of $X$, $x\in A$, $f:A\to E$ be differentiable at $x$, and $e_{k}\in X_{k}$ with 
$k\in[1,n]\cap\Z$. 
Then 
\begin{equation*}
(\mr{D}_{k}f)(x)(e_{k})=(D_{\mf{i}_{k}(e_{k})}f)(x).
\end{equation*}
Here $(\mr{D}_{k}f)(x)\in\mc{L}(X_{k},F)$ is the partial differential of $f$ at $x$ with respect to the space $X_{k}$, while 
$(D_{\mf{i}_{k}(e_{k})}f)(x)\in F$ is the directional derivative of $f$ at $x$ in the direction $\mf{i}_{k}(e_{k})$.
\end{proposition}
\begin{proposition}
[$n$-times continuous differentiability of $b\circ(h\boxtimes g)$]
\label{01101646}
Let $n\in\Z_{>0}$, $E,F_{1},F_{2}$ be normed spaces, $G$ be a Hausdorff locally convex space, $B$ be an open set of $E$, 
$h:B\to F_{1}$, $g:B\to F_{2}$ and $b:F_{1}\times F_{2}\to G$ be continuous and bilinear. Assume that $h\in\mc{C}^{n}(B,F_{1})$ and 
$g\in\mc{C}^{n}(B,F_{2})$, thus $b\circ(h\boxtimes g)\in\mc{C}^{n}(B,G)$.
\end{proposition}
\begin{lemma}
\label{12311713}
Let $H$ be a locally convex space, $E$ be a quasi-complete Hausdorff locally convex space, 
and $\mf{G}$ be a cover of $H$ made of bounded sets, thus $\mc{L}_{\mf{G}}(H,E)$ is a quasi-complete and Hausdorff
locally convex space.
\end{lemma}
\begin{corollary}
\label{06162059}
Let $T$ be a locally compact space, $\mu$ be a complex valued measure on $T$ and $f\in\mc{C}(T,E)$. Then the map
\begin{equation*}
\mc{H}(T;\C)\ni g\mapsto\int g\cdot f\,d\mu\in E,
\end{equation*}
is well-defined and belongs to $\mc{M}(T,E)$.
\end{corollary}
\begin{corollary}
\label{08111037}
Let $X,Y$ be topological spaces, $E$ be a topological linear space, $T:X\to Y$ be a homeomorphism of topological spaces and 
$Q\subset Y$. Then $T^{\ast}\mc{H}(Y,Q;E)=\mc{H}(X,T^{-1}(Q);E)$, where $T^{\ast}:g\mapsto g\circ T$.
Furthermore let $\rho_{Q}^{Y}$ be the injection of $\mc{H}(Y,Q;E)$ into $\mc{H}(Y;E)$, 
if either $Q$ is compact or $Q=Y$, then $T^{\ast}\circ\rho_{Q}^{Y}\in\mc{L}(\mc{H}(Y,Q;E),\mc{H}(X;E))$. 
\end{corollary}
\begin{corollary}
[Symmetry of higher partial differentials]
\label{06300307}
Let $F$ be a Hausdorff locally convex space, $n\in\Z_{>0}$, $N\in\Z_{>1}$, $\mc{E}=\{E_{i}\}_{i=1}^{n}$ be a sequence of normed 
spaces, let $E=\prod_{i=1}^{n}E_{i}$, $A$ be an open set of $E$ and $f:A\to F$ be $N$ times differentiable. Thus for 
every $r\in([1,n]\cap\Z)^{[1,N]\cap\Z}$ and every $p\in\mathbb{S}_{N}$ the group of permutations of $N$ elements, we have 
\begin{equation*}
\mr{D}_{r_{1}}(\mr{D}_{r_{2}}(\dots(\mr{D}_{r_{N}}f)\dots))
=
\mr{D}_{r_{p(1)}}(\mr{D}_{r_{p(2)}}(\dots(\mr{D}_{r_{p(N)}}f)\dots)).
\end{equation*}
\end{corollary}
\begin{corollary}
[Change of variable in multiple integrals.]
\label{09261537}
Let $n\in\Z_{>0}$, $U$ and $U^{\prime}$ be open sets of $\R^{n}$, $u:U\to U^{\prime}$ be a $C^{1}$ diffeomorphism.
Let $J$ be the Jacobian map of $\imath_{U^{\prime}}^{\R^{n}}\circ u$, let $\lambda_{U}$ and 
$\lambda_{U^{\prime}}$ be the restrictions at $U$ and $U^{\prime}$ respectively of the Lebesgue measure $\lambda$ on $\R^{n}$. 
Assume in addition that $F$ is a quasi-complete Hausdorff locally convex space and $\mr{f}:U^{\prime}\to F$. Thus 
\begin{equation*}
\mr{f}\in\mc{L}_{F}^{1}(\lambda_{U^{\prime}})
\Leftrightarrow
(\mr{f}\circ u)\cdot|J|
\in\mc{L}_{F}^{1}(\lambda_{U}),
\end{equation*}
and in such a case we have 
\begin{equation*}
\int\mr{f}\,d\lambda_{U^{\prime}}=\int(\mr{f}\circ u)\cdot|J|\,d\lambda_{U}.
\end{equation*}
\end{corollary}
\begin{proposition}
[Differential of $b\circ(h\boxtimes g)$] 
\label{03241344}
Let $E,F_{1},F_{2}$ be normed spaces, $B$ be an open set of $E$, $G$ be a Hausdorff locally convex space, $h:B\to F_{1}$, 
$g:B\to F_{2}$ and $b:F_{1}\times F_{2}\to G$ be continuous and bilinear.
If $h$ and $g$ are differentiable at $x$, then $b\circ(h\boxtimes g)$ is differentiable at $x$ and for every $x\in B$ and 
$v\in E$ we have
\begin{equation*}
(b\circ(h\boxtimes g))^{1}(x)(v)=b\left(h(x),g^{1}(x)v\right)+b\left(h^{1}(x)v,g(x)\right).
\end{equation*}
\end{proposition}
\begin{corollary}
\label{09292048}
Let $n\in\Z_{>0}$, $X$ be a topological space, $E_{i}$ be a topological linear space, $A_{i}\subseteq E_{i}$ for every 
$i\in[1,n]\cap\Z$, $F$ be a topological linear space and $b:\prod_{i\in[1,n]\cap\Z}E_{i}\to F$ be multilinear and continuous. 
If there exists $k\in[1,n]\cap\Z$ such that $f_{k}\in\mc{H}(X,A_{k};E_{k})$ and $f_{j}\in\mc{C}(X,E_{j})$ such that 
$\mr{supp}(f_{j})\subseteq A_{j}$ for every $j\in\complement_{[1,n]\cap\Z}\{k\}$, then $f\in\mc{H}(X,A,F)$ where 
$f=b\circ\prod_{i\in[1,n]\cap\Z}^{\boxtimes}f_{i}$ and $A=\bigcap_{i\in[1,n]\cap\Z}A_{i}$. 
\end{corollary}
\begin{corollary}
[\textbf{Differential commutes with integral. Compactly supported continuous maps and $F$ quasi-complete space}]
\label{01031245b}
Let $\mu$ be a measure on a locally compact space $T$, $G$ be a normed space, $F$ be a Hausdorff locally convex space,
$A$ be an open set of $G$ and $h:A\times T\to F$. Assume that 
\begin{enumerate}
\item
$F$ is quasi-complete;
\item
$(\exists W\subset T)((|\mu|^{\bullet}(\complement_{T}W)=0)$;
\item
$h^{t}$ is differentiable for every $t\in W$, where $h^{t}:A\to F$, $x\mapsto h(x,t)$. Let us define the map 
\begin{equation*}
\begin{aligned}
\frac{\partial h}{\partial G}:A\times W&\to\mc{L}_{b}(G,F),
\\
(x,t)&\mapsto\mr{d}(h^{t})(x).
\end{aligned}
\end{equation*}
\item
$(\exists\,\mr{S}:A\to\mc{P}(W))(\forall x\in A)(|\mu|^{\bullet}(\complement_{T}\mr{S}_{x})=0)$;
\item
$\mr{S}_{x}$ is $\mu$-measurable for every $x\in A$;
\item
for every $x\in A$ there exists a map $r_{x}\in\mc{C}_{c}(T,F)$ such that 
$h_{x}\circ\imath_{\mr{S}_{x}}^{T}=r_{x}\circ\imath_{\mr{S}_{x}}^{T}$;
\label{01031245hy7b}
\item
for every $x\in A$ there exists a map $j_{x}\in\mc{C}_{c}(T,\mc{L}_{b}(G,F))$ such that 
$\left(\frac{\partial h}{\partial G}\right)_{x}\circ\imath_{\mr{S}_{x}}^{W}=j_{x}\circ\imath_{\mr{S}_{x}}^{T}$;
\label{01031245hy8b}
\item
$(\exists\,\mr{T}:\Gamma_{F}\times A\to\mc{P}(W))(\forall(q,x)\in\Gamma_{F}\times A)(|\mu|^{\bullet}(\complement_{T}\mr{T}_{q,x})=0)$,
where $\Gamma_{F}$ is a fundamental set of seminorms of $F$.
\end{enumerate}
Thus 
\begin{equation*}
f:A\to F,\,x\mapsto\int h_{x}\,d\mu,
\end{equation*}
and 
\begin{equation*}
z:A\to\mc{L}_{b}(G,F),\,x\mapsto\int\left(\frac{\partial h}{\partial G}\right)_{x}\,d\mu,
\end{equation*}
are well-defined maps, $f$ is differentiable and 
\begin{equation*}
(\forall x\in A)(f^{1}(x)=z(x)).
\end{equation*}
\end{corollary}
\begin{remark}
\label{08190319}
Under the hypothesis of Cor.\ref{01031245b} let $w\in G$. Thus, $\mr{ev}_{w}\in\mc{L}(\mc{L}_{b}(G,F),F)$ therefore we have that 
\begin{equation*}
(\forall x\in A)\left(z(x)(w)=\int\mr{ev}_{w}\circ\bigl(\frac{\partial h}{\partial G}\bigr)_{x}\,d\mu\right).
\end{equation*}
\end{remark}
\begin{proposition}
[Partial differential and directional derivative. Numerical case]
\label{08091052p}
Let $E$ be a Hausdorff locally convex space, $n\in\Z_{>0}$, $X=\prod_{i=1}^{n}X_{i}$ with $X_{i}=\R$ for every $i\in[1,n]\cap\Z$, 
$A$ be an open set of $X$, $x\in A$ and $f:A\to E$ be differentiable at $x$, let $e[n]$ be the standard basis of $X$. 
Then for every $k\in[1,n]\cap\Z$ we have 
\begin{equation*}
(\mr{D}_{k}f)(x)(1)=(D_{e_{k}[n]}f)(x).
\end{equation*}
\end{proposition}
\begin{corollary}
[Divergence in a chart. Semi-Riemannian case]
\label{03290108}
If $(U,\phi)$ is a \emph{positively oriented} chart of a semi-Riemannian manifold $(M,g)$, then for every $X\in\mr{X}_{U}$ by 
letting $X_{\phi}^{i}\coloneqq\widehat{dx_{i}^{\phi}}(X)$ we have 
\begin{equation*}
\mr{div}_{g[U]}(X)
=
\sum_{i\in\mathbb{I}_{n}}\partial_{i}^{\phi}X_{\phi}^{i}
+
\frac{(-1)^{s}}{2}
\sum_{i\in\mathbb{I}_{n}}
\frac{X_{\phi}^{i}}{\mr{det}(M^{\phi}(g))}
\partial_{i}^{\phi}(\mr{det}(M^{\phi}(g))).
\end{equation*}
\end{corollary}
\begin{corollary}
[Divergence Theorem]
\label{03161712}
Assume that $M$ is oriented and has boundary and that either $\zeta\in(\mr{X}_{M})_{c}$ and $\omega\in\mr{Vl}(M)$ or 
$\zeta\in\mr{X}_{M}$ and $\omega\in\mr{Vl}_{c}(M)$ thus, 
\begin{equation*}
\int\mr{div}_{\omega}(\zeta)\omega=\int((\imath_{\partial M}^{M})^{\ast}\circ\mf{i}_{\zeta})(\omega).
\end{equation*}
Here $\mr{Vl}(M)$ (respectively $\mr{Vl}_{c}(M)$) is the set of volume forms on $M$ (respectively compactly supported).
\end{corollary}
\begin{definition}
\label{03011632}
Let $n\in\Z_{>0}$, $\mc{E}=\{E_{i}\}_{i=1}^{n}$ and $\mc{G}=\{G_{i}\}_{i=1}^{n}$ be sequences of normed spaces, 
$F$ be a Hausdorff locally convex space, and $Z=\bigtimes_{i=1}^{n}Q_{i}$ with $Q\in\prod_{i=1}^{n}\mc{L}(E_{i},G_{i})$. 
Thus we can define
\begin{equation*}
\begin{cases}
Z^{\intercal}:\mc{L}(G_{1},\dots,G_{n};F)\to\mc{L}(E_{1},\dots,E_{n};F);
\\
f\mapsto f\circ Z,
\end{cases}
\end{equation*}
where $Z:\prod_{i\in[1,n]\cap\Z}E_{i}\to\prod_{i\in[1,n]\cap\Z}G_{i}$, $(x_{1},\dots,x_{n})\mapsto(Q_{1}(x_{1}),\dots,Q_{n}(x_{n}))$, while 
$\mc{L}(E_{1},\dots,E_{n};F)$ is the space of continuous $n$-linear maps from $\prod_{i\in[1,n]\cap\Z}E_{i}$ to $F$.
\end{definition}
\begin{proposition}
\label{06301252}
Let $n\in\Z_{>0}$, $\mc{E}=\{E_{i}\}_{i=1}^{n}$ and $\mc{G}=\{G_{i}\}_{i=1}^{n}$ be sequences of normed 
spaces, $F$ be a Hausdorff locally convex space. Thus Def. \ref{03011632} is well-set. 
Moreover if $Q\in\prod_{i=1}^{n}\mc{L}(E_{i},G_{i})$, then for every continuous seminorm $\|\cdot\|$ of $F$ 
we have $\|(\bigtimes_{i=1}^{n}Q_{i})^{\intercal}(f)\|_{\sup}\leq\|f\|_{\sup}\prod_{i=1}^{n}\|Q_{i}\|$ and 
$(\bigtimes_{i=1}^{n}Q_{i})^{\intercal}$ is linear therefore by the above inequality continuous and 
$\|(\bigtimes_{i=1}^{n}Q_{i})^{\intercal}\|\leq\prod_{i=1}^{n}\|Q_{i}\|$. 
Here by letting $E=\prod_{i\in[1,n]\cap\Z}E_{i}$ we set for every continuous seminorm $\|\cdot\|$ of $F$ 
\begin{equation*}
\begin{cases}
\|\cdot\|_{\sup}:\mc{L}(E_{1},\dots,E_{n};F)\to\R_{+};
\\
f\mapsto\sup\bigl\{\|f(x)\|\,\vert\,x\in E\wedge(\forall i\in[1,n]\cap\Z)(\|\Pr_{i}x\|_{E_{i}}\leq 1)\bigr\}.
\end{cases}
\end{equation*}
\end{proposition}
\begin{theorem}
[Taylor expansion]
\label{8143}
Let $p\in\Z_{\geq 2}$, $E$ be a normed space, $F$ be a Hausdorff locally convex space, $A$ be an open set of $E$ and
$f:A\to F$ be $p$ times continuously differentiable. Thus for every $x\in A$ and 
$t\in E$ such that $[x,x+t]\subset A$, we have
\begin{equation}
\label{8143a}
f(x+t)=f(x)+\sum_{k=1}^{p-1}\frac{(\Upgamma_{\lr{E}{k,F}}\circ f^{k})(x)(t^{k})}{k!}
+
\left(\int_{0}^{1}\frac{(1-\zeta)^{p-1}}{(p-1)!}
(\Upgamma_{\lr{E}{p,F}}\circ f^{p})(x+\zeta t)\,d\zeta\right)(t^{p});
\end{equation}
where $t^{k}\in E^{k}$ such that $\Pr_{s}(t^{k})=t$ for every $s\in[1,k]\cap\Z$ and $k\in[1,p-1]\cap\Z$.
Moreover we have 
\begin{equation}
\label{8143b}
\lim_{t\to\ze,\,t\neq\ze}\|t\|^{-p}(f(x+t)-f(x)-\sum_{k=1}^{p}\frac{(\Upgamma_{\lr{E}{k,F}}\circ f^{k})(x)(t^{k})}{k!})=\ze.
\end{equation}
Here for every $r\in\Z_{>0}$ we let $\Upgamma_{\lr{E}{r,F}}$ be the inverse of the following linear continuous isomorphism
\begin{equation*}
\Uptheta_{\mc{E},F}\in\mc{L}\bigl(\mc{L}_{b}(E_{1},\dots,E_{r};F),\mc{L}_{b}(E_{1},\mc{L}(E_{2},\dots,\mc{L}(E_{r},F)\dots))\bigr),
\end{equation*}
defined by $\Uptheta_{\mc{E},F}(f)(x_{1})(x_{2})\dots(x_{r})=f(x_{1},x_{2},\dots,x_{r})$ for every $x\in\prod_{i=1}^{r}E_{i}$
with $E_{i}=E$ for every $i\in[1,r]\cap\Z$.
\end{theorem}
\begin{proposition}
\label{12031719}
Let $X$ be a topological space, $E$ be a topological linear space and $\mc{A}$ be a topological algebra. 
Thus for every $f,g:X\to E$ and every $a,b:X\to\mc{A}$ we have 
\begin{equation*}
\begin{aligned}
\mr{supp}(f+g)&\subseteq\mr{supp}(f)\cup\mr{supp}(g);
\\
\mr{supp}(a\cdot b)&\subseteq\mr{supp}(a)\cap\mr{supp}(b).
\end{aligned}
\end{equation*}
\end{proposition}
\begin{proposition}
\label{12061416}
Let $X,Y$ be topological spaces, $E$ be a topological linear space, $T:X\to Y$ and $f:Y\to E$, thus
$\mr{supp}(f\circ T)\subseteq\ov{\overset{-1}{T}(\mr{supp}(f))}$, in particular 
if $T$ is continuous, then $\mr{supp}(f\circ T)\subseteq\overset{-1}{T}(\mr{supp}(f))$.
\end{proposition}
\begin{corollary}
\label{03110819}
Let $E$ be a normed space, $F$ be a Hausdorff locally convex space, $u\in\mc{L}(E,F)$, $A$ be an open set of $E$, then
$u\circ\imath_{A}^{E}$ is differentiable and $(u\circ\imath_{A}^{E})^{\prime}(x)=u$ for every $x\in A$.
\end{corollary}
\begin{corollary}
\label{04121839}
Let $E$ be a normed space, $A$ be an open set of $E$, $m,p\in\Z_{>0}$, $\mc{F}=\{F_{k}\}_{k=1}^{m}$ be a family of 
Hausdorff locally convex spaces, $F=\prod_{k=1}^{m}F_{k}$ be endowed by the product topology
and $f:A\to F$. In order that $f$ be $p$ times continuously differentiable, it is necessary and sufficient that 
for every $k\in[1,m]\cap\Z$ the map $\Pr_{k}^{\mc{F}}\circ f$ be $p$ times continuously differentiable.
\end{corollary}
\begin{proposition}
[Derivative at a point of a product]
\label{05091652}
Let $J\subseteq\R$, $\xi_{0}\in J$, $n\in\Z_{\geq 2}$ and for every $i\in[1,n]\cap\Z$, let $E_{i}$ be a Hausdorff topological 
vector space over $\R$, $f_{i}:J\to E_{i}$ and $[\cdot,\dots,\cdot]:\prod_{i=1}^{n}E_{i}\to E$ be $n$-linear and continuous. 
If $f_{i}$ is derivable at $\xi_{0}$ for every $i\in[1,n]\cap\Z$, then 
$[\cdot,\dots,\cdot]\circ\boxtimes_{i=1}^{n}f_{i}$ is derivable at $\xi_{0}$ and by defining for every 
$(i,r)\in([1,n]\cap\Z)^{2}$ the map $u_{i}^{r}:\{\xi_{0}\}\to E_{i}$ such that 
\begin{equation*}
u_{i}^{r}(\xi_{0})\coloneqq
\begin{cases}
f_{i}(\xi_{0}),\,r\neq i,
\\
\mr{D}(f_{i})(\xi_{0}),\,r=i;
\end{cases}
\end{equation*}
we have 
\begin{equation*}
\mr{D}([\cdot,\dots,\cdot]\circ\boxtimes_{i=1}^{n}f_{i})(\xi_{0})
=
\sum_{r=1}^{n}([\cdot,\dots,\cdot]\circ\boxtimes_{i=1}^{n}u_{i}^{r})(\xi_{0}).
\end{equation*}
\end{proposition}

\bibliography{references}
\bibliographystyle{amsplain}
\end{document}